\documentclass[11pt]{article}
 
\usepackage[sort,numbers]{natbib}

\usepackage[bookmarks=true,bookmarksnumbered=true,colorlinks=true,allcolors=black]{hyperref}

\usepackage{graphicx}
\usepackage{amsmath,amssymb,amsfonts,amsthm}
\usepackage{lscape}
\usepackage{arydshln}
\renewcommand*{\baselinestretch}{1.25}

\usepackage[capitalize]{cleveref}

\usepackage{geometry}
\usepackage{indentfirst}

\usepackage{enumitem,color}

\newtheorem{theorem}{Theorem}[section]
\newtheorem{lemma}{Lemma}[section]
\newtheorem{proposition}{Proposition}[section]
\newtheorem{corollary}{Corollary}[section]
\theoremstyle{definition}
\newtheorem{definition}{Definition}[section]
\newtheorem*{rmk*}{Remark}
\newtheorem{rmk}{Remark}[section]

\newtheorem{example}{Example}[section]

\DeclareMathOperator{\Var}{Var}
\DeclareMathOperator{\Cov}{Cov}

\DeclareMathOperator{\E}{E}

\allowdisplaybreaks

\numberwithin{equation}{section}

\makeatletter
    \renewcommand*{\section}{\@startsection{section}{1}{\z@}%
    {6pt}{3pt}{\reset@font\normalsize\bfseries}}
\makeatother
    
\makeatletter
    \renewcommand*{\subsection}{\@startsection{subsection}{2}{\z@}%
    {3pt}{3pt}{\reset@font\normalsize\mdseries\itshape}}
\makeatother

\makeatletter
    \renewcommand*{\subsubsection}{\@startsection{subsubsection}{3}{\z@}%
    {3pt}{3pt}{\reset@font\normalsize\mdseries\itshape}}
\makeatother

\makeatletter
\def\@listi{\leftmargin\leftmargini
  \topsep=.5\baselineskip 
  \partopsep=0pt \parsep=0pt \itemsep=0pt}
\let\@listI\@listi
\@listi
\def\@listii{\leftmargin\leftmarginii
  \labelwidth\leftmarginii \advance\labelwidth-\labelsep
  \topsep=0pt \partopsep=0pt \parsep=0pt \itemsep=0pt}
\def\@listiii{\leftmargin\leftmarginiii
  \labelwidth\leftmarginiii \advance\labelwidth-\labelsep
  \topsep=0pt \partopsep=0pt \parsep=0pt \itemsep=0pt}
\def\@listiv{\leftmargin\leftmarginiv
  \labelwidth\leftmarginiv \advance\labelwidth-\labelsep
  \topsep=0pt \partopsep=0pt \parsep=0pt \itemsep=0pt}
\makeatother

\makeatletter
\newcommand{\opnorm}{\@ifstar\@opnorms\@opnorm}
\newcommand{\@opnorms}[1]{%
  \left|\mkern-1.5mu\left|\mkern-1.5mu\left|
   #1
  \right|\mkern-1.5mu\right|\mkern-1.5mu\right|
}
\newcommand{\@opnorm}[2][]{%
  \mathopen{#1|\mkern-1.5mu#1|\mkern-1.5mu#1|}
  #2
  \mathclose{#1|\mkern-1.5mu#1|\mkern-1.5mu#1|}
}
\makeatother

\makeatletter
  \renewenvironment{proof}[1][\proofname]{\par
    \pushQED{\qed}%
    \normalfont \topsep6\p@\@plus6\p@\relax
    \trivlist
    \item\relax
          {\bfseries
      #1\@addpunct{.}}\hspace\labelsep\ignorespaces
  }{%
    \popQED\endtrivlist\@endpefalse
  }
\makeatother

\def\be#1{\begin{equation*}#1\end{equation*}}
\def\ben#1{\begin{equation}#1\end{equation}}
\def\bes#1{\begin{equation*}\begin{split}#1\end{split}\end{equation*}}
\def\besn#1{\begin{equation}\begin{split}#1\end{split}\end{equation}}

\def\bm#1{\begin{multline*}#1\end{multline*}}

\def\ba#1{\begin{align*}#1\end{align*}}
\def\ban#1{\begin{align}#1\end{align}}

\newcommand{\eps}{\varepsilon}

\newcommand{\res}{\beta}
\newcommand{\Res}{B}
\newcommand{\weight}{m}

\newcommand{\bs}[1]{\boldsymbol{#1}}
\newcommand{\ol}[1]{\overline{#1}}
\newcommand{\ul}[1]{\underline{#1}}
\newcommand{\wh}[1]{\widehat{#1}}
\newcommand{\wt}[1]{\widetilde{#1}}

\newcommand{\mcl}[1]{\mathcal{#1}}

\newcommand{\norm}[1]{\left\|#1\right\|}
\newcommand{\bra}[1]{\left(#1\right)}
\newcommand{\cbra}[1]{\left\{#1\right\}}
\newcommand{\sbra}[1]{\left[#1\right]}
\newcommand{\abra}[1]{\left\langle#1\right\rangle}
\newcommand{\abs}[1]{\left|#1\right|}

\makeatletter
\newcommand{\pushright}[1]{\ifmeasuring@#1\else\omit\hfill$\displaystyle#1$\fi\ignorespaces}
\newcommand{\pushleft}[1]{\ifmeasuring@#1\else\omit$\displaystyle#1$\hfill\fi\ignorespaces}
\makeatother

\def\be#1{\begin{equation*}#1\end{equation*}}
\def\ben#1{\begin{equation}#1\end{equation}}
\def\bes#1{\begin{equation*}\begin{split}#1\end{split}\end{equation*}}
\def\besn#1{\begin{equation}\begin{split}#1\end{split}\end{equation}}

\def\bm#1{\begin{multline*}#1\end{multline*}}

\def\ba#1{\begin{align*}#1\end{align*}}
\def\ban#1{\begin{align}#1\end{align}}

\title{High-dimensional bootstrap and asymptotic expansion}
\author{Yuta Koike
\thanks{Graduate School of Mathematical Sciences, University of Tokyo}
\thanks{CREST, Japan Science and Technology Agency}
}

\begin{document}

\maketitle

\begin{abstract}

The recent seminal work of Chernozhukov, Chetverikov and Kato has shown that bootstrap approximation for the maximum of a sum of independent random vectors is justified even when the dimension is much larger than the sample size. 
In this context, numerical experiments suggest that third-moment matching bootstrap approximations would outperform normal approximation even without studentization, but the existing theoretical results cannot explain this phenomenon. 
In this paper, we develop an asymptotic expansion formula for the bootstrap coverage probability and show that it can give an explanation for the above phenomenon. 
In particular, we find the following interesting blessing of dimensionality phenomenon: The third-moment matching wild bootstrap is second-order accurate in high dimensions even without studentization if the covariance matrix has identical diagonal entries and bounded eigenvalues. 
We also show that a double wild bootstrap method is second-order accurate regardless of the covariance structure. 
The validity of these results is established under the assumption that the underlying distributions admit Stein kernels. 
\vspace{2mm}


\end{abstract}

\section{Introduction}

Let $X_1,\dots,X_n$ be independent centered random vectors in $\mathbb R^d$ with finite variance. Set
\[
S_n:=\frac{1}{\sqrt n}\sum_{i=1}^nX_i.
\]
The aim of this paper is to investigate the accuracy of bootstrap approximation for the maximum statistic
\[
T_n:=\max_{1\leq j\leq d}S_{n,j},
\]
when both $n$ and $d$ tend to infinity. 
The seminal work of \citet*{CCK13} has established a Gaussian-type approximation for this statistic under very mild assumptions when the dimension $d$ is possibly much larger than the sample size $n$. 
To be precise, let $Z$ be a centered Gaussian vector in $\mathbb R^d$ with the same covariance matrix as $S_n$, say $\Sigma$. 
The Gaussian analog of $T_n$ is given by
\[
Z^\vee:=\max_{1\leq j\leq d}Z_{j}.
\]
Under mild moment assumptions, \citet*{CCK13} have shown that 
\ben{\label{eq:cck13}
\sup_{t\in\mathbb R}|P(T_n\leq t)-P(Z^\vee\leq t)|=O\bra{\bra{\frac{\log^a(dn)}{n}}^{b}}
}
holds with $a=7$ and $b=1/8$. 
This result implies that given a significance level $\alpha\in(0,1)$, the probability $P(T_n\geq c^G_{1-\alpha})$ is approximately equal to $\alpha$ as long as $\log d=o(n^{1/7})$, where $c^G_{1-\alpha}$ is the $(1-\alpha)$-quantile of $Z^\vee$. 
Therefore, we can use $c^G_{1-\alpha}$ as a critical value to construct asymptotically $(1-\alpha)$-level simultaneous confidence intervals or $\alpha$-level tests for a high-dimensional vector of parameters; see \cite{BCCHK18,CCKK23} for details. 
In practice, $c^G_{1-\alpha}$ is not computable because $\Sigma$ is generally unknown, so we need to replace it by an estimate. 
In \cite{CCK13}, this is implemented by the Gaussian wild (or multiplier) bootstrap: Let $w_1,\dots,w_n$ be i.i.d.~standard normal variables independent of the data $X_1,\dots,X_n$. 
Define the Gaussian wild bootstrap version of $S_n$ as follows: 
\ben{\label{eq:wb}
S_n^*:=\frac{1}{\sqrt n}\sum_{i=1}^nw_i(X_i-\bar X),\quad\text{where }\bar X=\frac{1}{n}\sum_{i=1}^nX_i.
}
We may naturally expect that $c^G_{1-\alpha}$ would be well-approximated by the $(1-\alpha)$-quantile of the conditional law of $T_n^*:=\max_{1\leq j\leq d}S^*_{n,j}$ given the data, say $\hat c_{1-\alpha}$. This is formally justified by \cite{CCK13}: They essentially prove
\ben{\label{cck-boot}
P(T_n\geq \hat c_{1-\alpha})=\alpha+O\bra{\bra{\frac{\log^a(dn)}{n}}^{b}}
}
with $a=7$ and $b=1/8$. 
Subsequently, \citet{CCK17} have improved the convergence rates of \eqref{eq:cck13} and \eqref{cck-boot} to $b=1/6$. They also proved the left hand side of \eqref{eq:cck13} can be replaced by $\sup_{A\in\mcl R}|P(S_n\in A)-P(Z\in A)|$, where $\mcl R:=\{\prod_{j=1}^d[a_j,b_j]:a_j\leq b_j,~j=1,\dots,d\}$ is the class of rectangles in $\mathbb R^d$. 
 
It is easy to see that the conditional law of $S_n^*$ given the data is $N(0,\wh\Sigma_n)$, where $\wh\Sigma_n$ is the sample covariance matrix: $\wh\Sigma_n:=n^{-1}\sum_{i=1}^n(X_i-\bar X)(X_i-\bar X)^\top$. 
Hence, the Gaussian wild bootstrap is essentially a feasible version of the normal approximation for $T_n$. 
Then, it is natural to ask whether the approximation accuracy can be improved by more sophisticated bootstrap methods such as the empirical and non-Gaussian wild bootstraps. 
In the fixed-dimensional setting, it is well-known that standard bootstrap methods improve the approximation accuracy in the coverage probabilities over the normal approximation only when the statistic of interest is asymptotically pivotal (cf.~\cite[Chapter 3]{Ha92}). 
Even worse, they can be harmful for the univariate sample mean without studentization; see \cite[Section 3]{LiSi87}. 
Despite these facts, numerical experiments suggest that third-moment matching bootstrap methods would outperform normal approximation (cf.~\cite{DeZh20,CCKK22}). 
To appreciate this, we depict in \cref{fig:pp} the P--P plot for the rejection rate $P(T_n\geq \hat c_{1-\alpha})$ against the nominal significance level $\alpha$ when $n=200$ and $d=400$, where $\hat c_{1-\alpha}$ is computed by either the Gaussian wild bootstrap or a wild bootstrap with third-moment matching. 
We can clearly see that the latter performance is much better than the former. 

\begin{figure}[ht]
\centering
\includegraphics[scale=0.6]{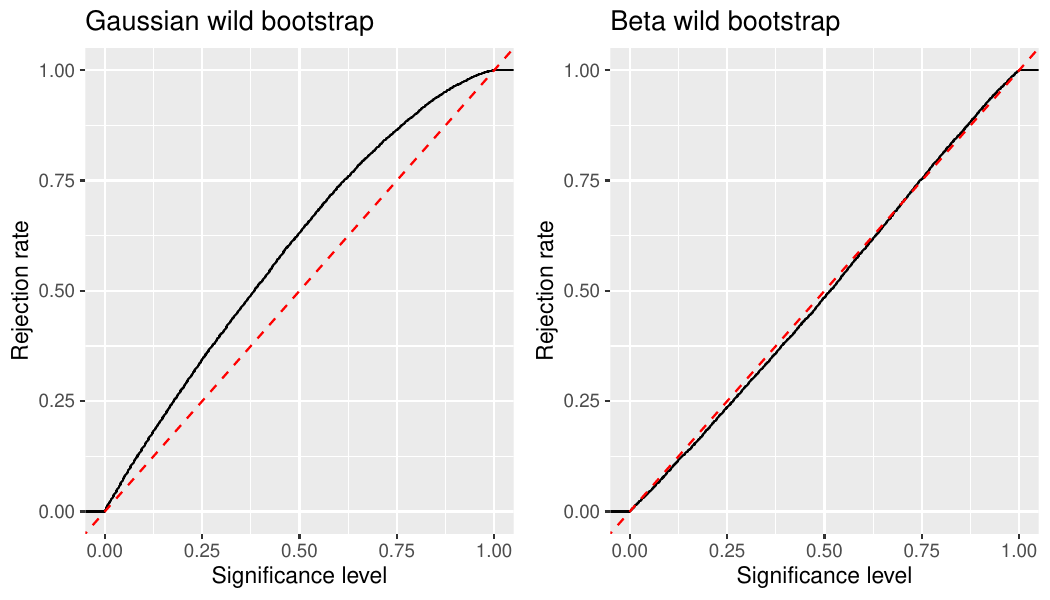}
\caption{\small
P--P plots for the rejection rate $P(T_n\geq \hat c_{1-\alpha})$ against the nominal significance level $\alpha$ when $n=200$ and $d=400$. 
The rejection rate is evaluated based on 20,000 Monte Carlo iterations. 
The critical value $\hat c_{1-\alpha}$ is computed by the Gaussian wild bootstrap for the left panel and the wild bootstrap with $w_1$ generated from the standardized beta distribution with parameters $\alpha,\beta$ given by \eqref{beta-param} with $\nu=0.1$ for the right panel, respectively. 
The number of bootstrap replications is 499. 
$X_1,\dots,X_n$ are generated from a Gaussian copula model with gamma marginals as in the simulation study of \cref{sec:simulate}. The parameter matrix is $R=(0.2^{|j-k|})_{1\leq j,k\leq d}$.}
\label{fig:pp}
\end{figure}

\citet{DeZh20} tried to explain this phenomenon by showing that convergence rates of third-moment matching bootstrap approximations have a better dimension dependence, i.e.~they achieve $a=5$ and $b=1/6$ in \eqref{cck-boot}. 
Later, however, it was shown in \cite{Ko21} that the same convergence rate is achieved by normal approximation, i.e.~\eqref{eq:cck13} holds with $a=5$ and $b=1/6$. 
\citet{CCKK22} have further improved the convergence rate to $a=5$ and $b=1/4$ for both normal and bootstrap approximations. 
Moreover, if we require $\Sigma$ to be invertible, it is possible to achieve the Berry--Esseen rate $n^{-1/2}$ up to a log factor even in the high-dimensional setting. 
Results in this direction first appeared in \citet{FaKo21}, where the following result is obtained when $X_1,\dots,X_n$ are log-concave:
\ben{\label{eq:opt-normal}
\sup_{A\in\mcl R}|P(S_n\in A)-P(Z\in A)|=O\bra{\sqrt{\frac{\log^3d}{n}}\log n}.
}
This rate is known to be optimal up to the $\log n$ factor in terms of both $n$ and $d$; see Proposition 1.1 in \cite{FaKo21} and also \cref{coro:lb-ga}. 
This type of result has been further investigated in \cite{Lo22,KuRi20,CCK23}. 
In particular, \citet{CCK23} have obtained the above nearly optimal rate when $\max_{i,j}|X_{ij}|$ is bounded. 
Here, the boundedness condition can be replaced with the sub-exponential condition by a simple truncation argument; see Appendix \ref{sec:cck23}. 
Further, in some situations, the rate $n^{-1/2}$ is (nearly) attainable even when $\Sigma$ is (asymptotically) degenerate; see \cite{LLM20,FaKo24,FKLZ23}. 
Nevertheless, all of these improvements are valid for normal approximation and thus do not explain the superior performance of third-moment matching bootstrap approximations. 

In this paper, we aim to explain the superior performance of bootstrap approximation in high dimensions using Edgeworth expansion and related techniques. 
Specifically, we develop an asymptotic expansion formula for $P(T_n\geq \hat c_{1-\alpha})$ in \cref{thm:coverage}. 
One main implication of this formula is that when $d\geq n$ and $\Sigma$ has identical diagonal entries and bounded eigenvalues, the wild bootstrap with third-moment matching is second-order accurate in the sense that
\[
P(T_n\geq \hat c_{1-\alpha})=\alpha+O\bra{\frac{\log^a(dn)}{n}}
\]
for some constant $a>0$ (see \cref{coro:coverage}). 
This shows that high dimensionality can be beneficial to the accuracy of the third-moment matching wild bootstrap, revealing the blessing of dimensionality in this context. 
By contrast, the Gaussian wild bootstrap does not benefit from the high dimensionality in this situation, which clearly explains the performance difference in \cref{fig:pp}. 
At the same time, our asymptotic expansion formula also shows that this is not always the case: The structure of $\Sigma$ strongly determines whether the above improvement occurs. In particular, when $\Sigma$ is an equicorrelation matrix, the third-moment matching wild bootstrap could be inferior to the Gaussian wild bootstrap; see \cref{coro:factor} and the simulation results in \cref{sec:simulate}. 
For this reason, we also develop an alternative approximation to the quantiles of $T_n$ that is second-order accurate regardless of the structure of $\Sigma$.  
%
A classical solution to this problem is bootstrapping the studentized version of $S_n$, but this is impossible in high dimensions since the sample covariance matrix $\wh\Sigma_n$ is degenerate whenever $d\geq n$. 
Instead, we achieve this by \citet{Be88}'s double bootstrap method, another classical technique to improve the approximation accuracy for non-pivotal statistics; see \cref{sec:db}. 

Despite the fact that Edgeworth expansion is a standard tool to analyze the performance of bootstrap in the classical setting (cf.~\cite{Ha92}), this approach has not been investigated for the above problem so far. 
One main reason is the lack of valid Edgeworth expansion for $T_n$ in the high-dimensional setting. 
While asymptotic expansions for statistics of high-dimensional data have been actively studied in multivariate statistics (see \cite{FuUl20} for an overview), results developed there seem inapplicable to our problem. 
One main reason is that $T_n$ may not have any limit distribution as $n,d\to\infty$ even after properly scaled. 
In fact, this is one of the motivations for the development of Chernozhukov--Chetverikov--Kato's theory. 
In view of \eqref{eq:opt-normal}, we are concerned with Edgeworth expansion of $P(S_n\in A)$ over $A\in\mcl R$. 
In the fixed-dimensional setting, a valid Edgeworth expansion of $P(S_n\in A)$ is conventionally derived from an asymptotic expansion of the characteristic function of $S_n$ via Fourier analysis (see e.g.~\cite{BhRa10}). 
Such an argument makes the dimension dependence of the error bound extremely complicated, so it is rarely given explicitly. 
One exceptional work is \citet{AHT98}, but their proof technique seems to inherently require the condition $d\ll n$ and is thus inapplicable to our setting.    
In fact, in the high-dimensional setting, the geometry of the set $A$ plays a key role to obtain an improved dimension dependence of error bounds, and it is unclear how to incorporate such information into Fourier analytic arguments. 
We also mention the recent work by \citet{Zh22} who establishes explicit, computable error bounds for $\sup_{A\in\mcl A}|P(S_n\in A)-P(S_n'\in A)|$ where $S_n'$ is another sum of independent random vectors and $\mcl A$ is either the class of balls or half-spaces. However, apart from other technical issues, these error bounds contain $1/\sqrt n$ terms and cannot be used for second-order analysis. 

To circumvent the above issue, we develop valid asymptotic expansions using Stein's method. 
The use of Stein's method for asymptotic expansion was initiated by \citet{Ba86} who derived an asymptotic expansion of $\E[h(S_n)]$ when $d=1$ and $h$ is a smooth function. 
To drop the smoothness of the test function $h$, the so-called Cram\'er's condition is usually assumed in the Fourier analytic approach, but it is unknown how to (directly) incorporate Cram\'er's condition into Stein's method based arguments. 
Instead, we assume that the underlying random vectors have Stein kernels, motivated by the recent development of this approach by \citet{FaLi22} in the univariate case (see Lemma 2.1 ibidem). 
Apart from the technical difficulty, Cram\'er's condition is violated whenever the underlying statistic has a singular covariance matrix. This is unsuitable for application to bootstrap statistics in high-dimensions, so Stein kernels will be a more appropriate tool for our problem (see \cref{rmk:cramer}).  

In addition to the above development, we also establish two new inequalities for high-dimensional normal distributions in order to overcome further difficulties that arise specifically when proving the validity of asymptotic expansions in high dimensions. 
The first inequality is an anti-concentration inequality on rectangles for the higher-order terms of the Edgeworth expansion. 
Such a bound is necessary when controlling the remainder term appearing in an application of the so-called smoothing inequality, which is a standard initial step in justifying multivariate Edgeworth expansions (see e.g.~\cite[page 91]{BhRa10}). 
Existing bounds, such as \cite[Corollary 3.2]{BhRa10}, involve constants that grow polynomially with the dimension and therefore do not work in ultra-high-dimensional regimes where the dimension is far larger than the sample size. 
To address this limitation, in \cref{lem:anti} we derive a new anti-concentration inequality for the higher-order terms of the Edgeworth expansion that depends only poly-logarithmically on the dimension. See also \cref{rmk:coro1}. 
The second inequality is necessary for establishing the validity of the Cornish--Fisher expansion used in justifying the asymptotic expansion formula for $P(T_n\geq\hat c_{1-\alpha})$. 
In the conventional derivation of the Cornish--Fisher expansion, one assumes that the distribution function of the statistic admits an absolutely continuous limit distribution function $G$. 
The Cornish--Fisher expansion is then obtained by transforming the Edgeworth expansion through the inverse of $G$ and applying a Taylor expansion (see \cite[Section 5.2]{FuUl20}). 
In this situation, since $G$ does not depend on $n$, the validity of the Cornish--Fisher expansion follows almost automatically from that of the Edgeworth expansion. 
However, as noted earlier, in high-dimensional settings the maximum statistic $T_n$ generally does not possess a non-degenerate limit distribution. 
Consequently, one must instead transform the Edgeworth expansion using the distribution function $F_Z$ of $Z^\vee$, which depends on $n$, thereby requiring a sufficiently sharp \emph{non-asymptotic} bound for the inverse of $F_Z$. 
To resolve this issue, we derive in \cref{gmax-quantile} a new isoperimetric-type inequality for Gaussian maxima that yields precisely the control needed; see also \cref{rmk:coverage}.

The remainder of the paper is organized as follows. 
\cref{sec:ae} collects the main results of this paper: 
In \cref{sec:2nd-main}, we establish valid Edgeworth expansions for $S_n$ and $S_n^*$ in high dimensions. 
Then, we develop an asymptotic expansion formula for $P(T_n\geq\hat c_{1-\alpha})$ and discuss its implications in \cref{sec:coverage}. 
Finally, we show in \cref{sec:db} that a double wild bootstrap method is second-order accurate. 
\cref{sec:simulate} contains a small simulation study. 
Most proofs are collected in Sections \ref{sec:proof-sec2} and \ref{sec:proof-2nd-order}. 
The appendix contains additional proofs and auxiliary results. 

\paragraph{Notation}

For a vector $x\in\mathbb R^d$, we set $|x|:=\sqrt{\sum_{j=1}^dx_j^2}$ and $x^\vee:=\max_{1\leq j\leq d}x_j$. 
We denote by $\bs1_d=(1,\dots,1)^\top\in\mathbb R^d$ the all-ones vector in $\mathbb R^d$. 
For $r\in\mathbb N$, $(\mathbb R^d)^{\otimes r}$ denotes the set of real-valued $d$-dimensional $r$-arrays $V=(V_{j_1,\dots,j_r})_{1\leq j_1,\dots,j_r\leq d}$. 
In particular, $(\mathbb R^d)^{\otimes1}=\mathbb R^d$ and $(\mathbb R^d)^{\otimes2}$ is the set of $d\times d$ matrices. 
For $U\in(\mathbb R^d)^{\otimes q}$ and $V\in(\mathbb R^d)^{\otimes r}$, we set $U\otimes V:=(U_{i_1,\dots,i_q}V_{j_1,\dots,j_r})_{1\leq i_1,\dots,i_q,j_1,\dots,j_r\leq d}\in(\mathbb R^{d})^{\otimes(q+r)}$. We write $U^{\otimes2}=U\otimes U$ for short. 
When $q=r$, we also set $\langle U,V\rangle:=\sum_{j_1,\dots,j_r=1}^dU_{j_1,\dots,j_r}V_{j_1,\dots,j_r}$. In particular, when $q=r=1$, $\langle U,V\rangle$ is the Euclidean inner product of $U$ and $V$, which we also write $U\cdot V$. 
In addition, we set $\|V\|_1:=\sum_{j_1,\dots,j_r=1}^d|V_{j_1,\dots,j_r}|$ and $\|V\|_\infty:=\max_{1\leq j_1,\dots,j_r\leq d}|V_{j_1,\dots,j_r}|.$ 
Further, for $x\in\mathbb R^d$, we define $x^{\otimes r}:=(x_{j_1}\cdots x_{j_r})_{1\leq j_1,\dots,j_r\leq d}\in(\mathbb R^d)^{\otimes r}$. 
Finally, we set
\[
\ol{X^r}:=\frac{1}{n}\sum_{i=1}^nX_i^{\otimes r}.
\]

Given an $r$-times differentiable function $h:\mathbb R^d\to\mathbb R$, we set $\nabla^rh(x):=(\partial_{j_1,\dots,j_r}h(x))_{1\leq j_1,\dots,j_r\leq d}\in(\mathbb R^d)^{\otimes r}$ for $x\in\mathbb R^d$, where $\partial_{j_1,\dots,j_r}=\frac{\partial^r}{\partial x_{j_1}\cdots\partial x_{j_r}}$. 
For $m\in\mathbb N\cup\{\infty\}$, $C^m_b(\mathbb R^d)$ denotes the set of bounded $C^m$ functions with bounded derivatives. 

For an invertible matrix $V$, $\phi_V$ denotes the density of $N(0,V)$. We write $\phi_d=\phi_{I_d}$ for short, where $I_d$ is the $d\times d$ identity matrix. 
Further, we write $\phi=\phi_1$ for short. $\Phi$ denotes the standard normal distribution function. 
Also, for a distribution function $F:\mathbb R\to[0,1]$, its (generalized) inverse is defined as
$
F^{-1}(p)=\inf\{t\in\mathbb R:F(t)\geq p\},~ p\in(0,1). 
$
We refer to Appendix A.1 in \cite{BoLe19} for useful properties of inverse distribution functions. 

For a random vector $\xi$ and $p\in(1,\infty)$, we set $\|\xi\|_p:=(\E[|\xi|^p])^{1/p}$ (recall that $|\cdot|$ is the Euclidean norm). 
Further, for $\alpha>0$, we set $\|\xi\|_{\psi_\alpha}:=\inf\{t>0:\E[\exp\{(|\xi|/t)^\alpha\}]\leq2\}$. 
For two random vectors $\xi$ and $\eta$, we write $\xi\overset{d}{=}\eta$ if $\xi$ has the same law as $\eta$.

We assume $d\geq3$ whenever we consider an expression containing $\log d$. 
A similar convention is applied to $n$.

\section{Main results}\label{sec:ae}

Throughout the paper, we assume that $S_n$ has an invertible covariance matrix $\Sigma$ and denote by $\sigma_*$ the square root of the minimum eigenvalue of $\Sigma$. 
We also set $\ol\sigma=\max_{j=1,\dots,d}\sqrt{\Sigma_{jj}}$ and $\ul\sigma=\min_{j=1,\dots,d}\sqrt{\Sigma_{jj}}$. 
Further, $w_1,\dots,w_n$ denote i.i.d.~random variables independent of $X_1,\dots,X_n$. They are used to define the wild bootstrap statistic $S_n^*$ in \eqref{eq:wb}. We always assume $\E[w_1]=0$ and $\E[w_1^2]=1$. 
Also, $P^*$ and $\E^*$ denote the conditional probability and expectation given the data $X_1,\dots,X_n$, respectively. 
For $p\in(0,1)$, $\hat c_p$ denotes the conditional $p$-quantile of $T_n^*$ given the data, i.e.~$\hat c_p:=\inf\{t\in\mathbb R:P^*(T_n^*\leq t)\geq p\}$.

\subsection{Valid Edgeworth expansion in high dimensions}\label{sec:2nd-main}

We begin by introducing appropriate (second-order) Edgeworth expansions for $S_n$ and $S_n^*$. 
The former is standard. That is, our Edgeworth expansion for $S_n$ is defined as
\ba{
p_n(z)&=\phi_\Sigma(z)-\frac{1}{6\sqrt n}\langle\E[\ol{X^{3}}],\nabla^3\phi_\Sigma(z)\rangle,\quad z\in\mathbb R^d.
}
The situation is different for the latter. 
In the low-dimensional setting, a natural bootstrap version of $p_n(z)$ would be obtained by replacing $\Sigma$ and $\E[\ol{X^{3}}]$ with their sample counterparts $\wh\Sigma_n$ and $\ol{X^{3}}$, respectively. 
However, when $d\geq n$, $\wh\Sigma_n$ is always degenerate, so $\phi_{\wh\Sigma_n}$ is not well-defined. 
For this reason, we consider an Edgeworth expansion ``around $\phi_\Sigma$''. Formally, our Edgeworth expansion for $S_n^*$ is defined as
\ba{
\hat{p}_{n,\gamma}(z)&=\phi_\Sigma(z)+\frac{1}{2}\langle\ol{X^2}-\Sigma,\nabla^2\phi_\Sigma(z)\rangle-\frac{\gamma}{6\sqrt n}\langle \ol{X^3},\nabla^3\phi_\Sigma(z)\rangle,\quad z\in\mathbb R^d,
}
where $\gamma\in\mathbb R$ is a constant determined by the construction of $S_n^*$. 
Typically, $\gamma=1$ for third-moment matching bootstrap methods. 

Next, we formally define the notion of Stein kernel. 
\begin{definition}[Stein kernel]\label{def:sk}
Let $\xi$ be a random vector in $\mathbb R^d$ with $\E[\|\xi\|_\infty]<\infty$. 
A measurable function $\tau:\mathbb R^d\to\mathbb R^d\otimes\mathbb R^d$ is called a \emph{Stein kernel} for (the law of) $\xi$ if $\E[\|\tau(\xi)\|_\infty]<\infty$ and 
\begin{equation}\label{eq:sk}
\E[(\xi-\E[\xi])\cdot\nabla h(\xi)]=\E[\langle\tau(\xi),\nabla^2h(\xi)\rangle]
\end{equation}
for any $h\in C^2_b(\mathbb R^d)$.
\end{definition}

The concept of Stein kernel was originally introduced in \citet[Lecture VI]{St86} for the univariate case. 
Although its partial multivariate extension dates back to \cite{CaPa92}, general treatments have started in more recent studies of \cite{NPS14,LNP15}, stemming from the discovery of a connection to Malliavin calculus due to \citet{NoPe09} (the so-called \emph{Malliavin--Stein method}). We refer to \cite{MRRS23} for the recent development. 

\begin{rmk}[Alternative definition]
Our definition of Stein kernel is taken from \cite{LNP15}. 
In the literature, the definition of Stein kernel often requires \eqref{eq:sk} to hold with $\nabla h$ on both sides replaced by any bounded $C^1$ function $h:\mathbb R^d\to\mathbb R^d$ with bounded derivatives. 
Except for the case $d=1$, this requirement is slightly stronger than ours. Nevertheless, as far as the author knows, this stronger requirement has so far been met by all known constructions of Stein kernels, including all the examples of this paper. 
\end{rmk}

The validity of Edgeworth expansion for $S_n$ is ensured if the summands have Stein kernels:
\begin{theorem}[Edgeworth expansion for $S_n$]\label{coro1}
Suppose that $X_i$ has a Stein kernel $\tau^X_i$ for every $i=1,\dots,n$. 
Suppose also that there exists a constant $b>0$ such that
\ben{\label{ass:psi1}
\max_{1\leq i\leq n}\max_{1\leq j\leq d}\|X_{ij}\|_{\psi_1}\leq b
}
and
\ben{\label{ass:sk}
\max_{1\leq i\leq n}\max_{1\leq j,k\leq d}\|\tau^X_{i,jk}(X_{i})\|_{\psi_{1/2}}\leq b^2.
}
Further, assume $\log^3d\leq n$. Then,
\ben{\label{eq:ae}
\sup_{A\in\mcl R}\abs{P(S_n\in A)-\int_{A}p_n(z)dz}
\leq C\frac{b^5}{\sigma_*^5}\frac{\log^3 d}{n}\log n.
}
\end{theorem}

\begin{rmk}
Here and below, we do not intend to optimize the dependence of bounds on $b$ and $\sigma_*$. 
\end{rmk}

\begin{rmk}[Proof strategy]\label{rmk:coro1}
One main tool to prove \eqref{eq:ae} and its bootstrap counterpart \eqref{eq:ae-boot} is an \emph{explicit} decomposition formula of $\E[h(S_n)]-\int_{\mathbb R^d}h(z)p_n(z)dz$ for a suitable test function $h:\mathbb R^d\to\mathbb R$ using Stein kernel (see \cref{lem:decomp}). 
In the univariate case, such a decomposition has recently been derived in \cite{FaLi22} for \emph{any} bounded $h$ (see Eqs.(2.10) and (2.11) ibidem). 
Although their argument can be extended more or less naturally to the multivariate case as long as $h$ is sufficiently smooth, it would generally fail if $h$ is just bounded. 
This is because when $d\geq2$, the Stein equation associated with $h$ is a second-order partial differential equation, so its solution is less smooth than in the univariate case (see also \cite[Remark 2.7]{Ga16} for a related discussion). 
For this reason, we prove the decomposition formula with $h$ replaced by a smoothed version $h_t(x)=\E[h(\sqrt{1-t}x+\sqrt tZ)]$ with $t\in(0,1]$. 
For $h=1_A$ with $A\in\mcl R$, this formula along with \cref{coro:aht} gives an error bound for $\E[h_t(W)]-\int_{\mathbb R^d}h_t(z)p_n(z)dz$ depending (poly-)logarithmically on $d$ and $t$ 

Therefore, in contrast to the univariate case, we need an additional argument to make the smoothing error negligible. 
Conventionally, this is accomplished by the so-called \emph{smoothing inequality} and a suitable estimate for $\sup_A\int_{(\partial A)^t}|p_n(z)|dz$ (called an \emph{anti-concentration inequality}), where the supremum is taken over $\mcl R$ in our case. 
See e.g.~\cite[Lemma 11.4]{BhRa10} for the former and \cite[Corollary 3.2]{BhRa10} for the latter, respectively. 
However, to the author's knowledge, all the existing bounds for $\sup_A\int_{(\partial A)^t}|p_n(z)|dz$ depend polynomially on $d$ and are inadequate for our purpose. 
To resolve this issue, we establish a novel bound for $\sup_A|\int_{(\partial A)^t}p_n(z)dz|$ which depends only poly-logarithmically on $d$ (see \cref{lem:anti}). 
This together with a modified smoothing inequality (\cref{l3}) shows that the above smoothing error is negligible.  
\end{rmk}

Below we give a few examples satisfying \eqref{ass:sk}.
\begin{example}[Log-concave distribution]
Log-concave distributions are prominent examples having Stein kernels with the desired properties. 
The following result follows from \cite[Proposition A.5]{AGB15}, \cite[Theorem 2.3 and Proposition 3.2]{Fa19} and \cref{moment-psi}. 
\begin{lemma}[Stein kernel of log-concave distribution]
Suppose that a random vector $X$ in $\mathbb R^d$ has a log-concave density. 
Then $X$ has a Stein kernel $\tau$ and
\[
\max_{1\leq j\leq d}\|X_{j}\|_{\psi_1}\leq C\max_{1\leq j\leq d}\sqrt{\E[X_{j}^2]},\qquad
\max_{1\leq j,k\leq d}\|\tau_{jk}(X)\|_{\psi_{1/2}}\leq C\max_{1\leq j\leq d}\E[X_{j}^2]
\]
for some universal constant $C>0$.
\end{lemma}
\end{example}

\begin{example}[Gaussian copula model]\label{ex:copula}
Let $R$ be a $d\times d$ positive semidefinite symmetric matrix with unit diagonals. 
Also, for every $j=1,\dots,d$, let $\mu_j$ be a non-degenerate probability distribution on $\mathbb R$ (i.e.~$\mu_j$ is not the unit mass at a point), and denote by $F_j$ its distribution function. 
The Gaussian copula model $U=(U_1,\dots,U_d)^\top$ with parameter matrix $R$ and marginal distributions $\mu_1,\dots,\mu_d$ is defined as $U_j=F_j^{-1}(\Phi(Z_j))$ for $j=1,\dots,d$, where $Z\sim N(0,R)$. 
%
\begin{proposition}[Stein kernel of Gaussian copula model]\label{prop:copula}
Suppose that there exists a constant $\kappa>0$ such that for every $j=1,\dots,d$ and any Borel set $B\subset\mathbb R$, 
\ben{\label{eq:cheeger}
\liminf_{h\downarrow0}\frac{\mu_j(B^h)-\mu_j(B)}{h}
\geq \kappa\min\{\mu_j(B),1-\mu_j(B)\},
}
where $B^h:=\{t\in\mathbb R:|t-s|< h\text{ for some }s\in B\}$. 
Then $X:=U-\E[U]$ has a Stein kernel $\tau$ and
\[
\max_{1\leq j\leq d}\|X_{j}\|_{\psi_1}\leq C\kappa^{-1},\qquad
\max_{1\leq j,k\leq d}\|\tau_{jk}(X)\|_{\psi_1}\leq C\kappa^{-2}
\]
for some universal constant $C>0$.
\end{proposition}
%
The maximal constant $\kappa$ satisfying \eqref{eq:cheeger} is called the \emph{Cheeger (isoperimetric) constant} of $\mu_j$. 
We refer to \cite[Theorem 1.3]{BoHo97} for a useful equivalent formulation in the univariate case. 
When $\mu_j$ is log-concave, then \eqref{eq:cheeger} is satisfied with $\kappa=1/\sqrt{3\Var[X_j]}$ by Proposition 4.1 in \cite{Bo99}. 
Since the gamma distribution with shape parameter $\geq1$ is log-concave, \cref{prop:copula} shows that the simulated model in the introduction satisfies the assumptions of \cref{coro1}. 
We can actually show that any gamma distribution has a positive Cheeger constant; see \cref{cheeger-gamma}. 
\end{example}

\begin{example}[Affine transformation]\label{sk-affine}
Let $\xi$ be a random vector in $\mathbb R^r$ with a Stein kernel $\tau$. Then, a straightforward computation shows that for any $a\in\mathbb R^d$ and $d\times r$ matrix $V$, $V\xi+a$ has a Stein kernel given by $x\mapsto\E[V\tau(\xi)V^\top\mid V\xi+a=x]$. 
\end{example}

\begin{example}[Multiplicative perturbation]\label{ex:regression}
Let $X$ be a random vector in $\mathbb R^d$ and $\epsilon$ a centered random variable independent of $X$ and having a Stein kernel $\tau$. 
Then $\epsilon X$ has a Stein kernel $x\mapsto\E[\tau(\epsilon) X^{\otimes2}\mid\epsilon X=x]$, provided that $\E[\|\epsilon X\|_\infty]+\E[\|\tau(\epsilon) X^{\otimes2}\|_\infty]<\infty$. 
This easily follows by applying \cref{sk-affine} conditional on $X$. 
This type of random vector arises in high-dimensional regression; see \cite[Section 4]{CCK13}. 
Note that univariate Stein kernels can be written down explicitly and their properties are well-investigated in the literature; see \cite{LRS17,Do25} for example. 
\end{example}

\begin{example}[General continuous distribution]
Let $X$ be a centered random vector in $\mathbb R^d$ with density $f$. 
If $\E[|X|^2]<\infty$ and the support of $f$ is a (possibly unbounded) rectangle, we can construct a Stein kernel for $X$ by modifying the construction given in \cite[Remark 4.13]{MRRS23} as follows (see also the proof of \cite[Theorem 4]{ABBN04}). 
For a vector $x\in\mathbb R^d$ and $j\in\{1,\dots,d\}$, we set $x_{1:j}:=(x_1,\dots,x_j)^\top$ and $x_{j:d}:=(x_j,\dots,x_d)^\top$. 
Also, define a function $g_j:\mathbb R^j\to\mathbb R$ as $g_j(u)=\int_{\mathbb R^{d-j}}f(u,v)dv$. 
Thanks to the fact that $f$ is supported by a rectangle, we can define a function $\tau_{1j}:\mathbb R^d\to\mathbb R$ such that
\ba{
\tau_{1j}(x)f(x)&=g_{j-1}(x_{1:(j-1)})\int_{\mathbb R^{j-1}\times[x_j,\infty)} u_1f(u,x_{(j+1):d})du\\
&\quad-\int_{x_j}^\infty g_j(x_{1:(j-1)},u)du\int_{\mathbb R^{j}}u_1f(u,x_{(j+1):d})du
}
a.e.~$x$ with respect to the Lebesgue measure, where we set $g_0\equiv1$ by convention. 
By construction, it is straightforward to verify that $\sum_{j=1}^d\partial_j(\tau_{1j}f)(x)=-x_1f(x)$ a.e.~$x$. 
Moreover, the assumptions $\E[|X|^2]<\infty$ and $\E[X]=0$ respectively imply $\E[|\tau_{1j}(X)|]<\infty$ and $\tau_{1j}(x)f(x)\to0$ as $|x_j|\to\infty$ a.e.~$(x_1,\dots,x_{j-1},x_{j+1},\dots,x_d)$. 
Hence, integration by parts gives $\sum_{j=1}^d\E[\tau_{1j}(X)\partial_jh(X)]=\E[X_1h(X)]$ for any $h\in C^1_b(\mathbb R^d)$. 
Rotating the indices, we can analogously construct functions $\tau_{ij}:\mathbb R^d\to\mathbb R$ $(i,j=1,\dots,d)$ satisfying $\E[|\tau_{ij}(X)|]<\infty$ and $\sum_{j=1}^d\E[\tau_{ij}(X)\partial_jh(X)]=\E[X_ih(X)]$ for any $h\in C^1_b(\mathbb R^d)$. 
Consequently, $(\tau_{ij})_{1\leq i,j\leq d}$ is a Stein kernel for $X$. 

Let us give a sufficient condition to bound $\|\tau_{ij}(X)\|_{\psi_{1/2}}$ (we will actually bound $\|\tau_{ij}(X)\|_{\psi_{1}}$). 
For a non-empty subset $I\subset\{1,\dots,d\}$, we denote by $f_I$ the density of $(X_i)_{i\in I}$. 
Suppose that there exists a constant $K\geq1$ such that for any disjoint partition $\{1,\dots,d\}=I_1\cup I_2\cup I_3$,
\ben{\label{eq:dens-ratio}
\frac{1}{K}\leq\frac{f_{I_1}((X_i)_{i\in I_1})f_{I_2}((X_i)_{i\in I_2})f_{I_3}((X_i)_{i\in I_3})}{f(X)}\leq K\quad\text{a.s.}
}
Also, suppose that the Cheeger constant of the law of $X_j$ is bounded below by $1/K$ for all $j\in\{1,\dots,d\}$ (cf.~\cref{ex:copula}). Then we have $|\tau_{ij}(X)|\leq2(K+|X_i|+\E[|X_i|])K^3$ a.s. 
Hence, there exists a universal constant $C$ such that $\|\tau_{ij}(X)\|_{\psi_1}\leq CK^4$. 

Condition \eqref{eq:dens-ratio} can be interpreted as a weak dependence condition between the components of $X$. Though this condition seems somewhat restrictive, it is satisfied e.g.~when the copula of $X$ belongs to the Farlie--Gumbel--Morgenstern family with appropriate parameter values (see \cite[Example 3.31]{Ne06}). 
\end{example}



\begin{rmk}[Relation to classical conditions]\label{rmk:cramer}
(a) In the univariate case, if a non-degenerate distribution has a Stein kernel, then it has a non-zero absolutely continuous part; see \cite[Theorem 1.1]{Do25}. In particular, it must satisfy Cram\'er's condition. 

\noindent(b) In the multivariate case, a non-degenerate distribution may not satisfy Cram\'er's condition even when it has a Stein kernel: A simple example is a multivariate normal distribution with singular covariance matrix. This example is indeed important in the high-dimensional setting when analyzing the Gaussian wild bootstrap. 

\noindent(c) If a probability distribution has a Stein kernel, its support is convex. This follows from the fact that any one-dimensional marginal has a Stein kernel and thus is supported by an interval according to \cite[Theorem 1.1]{Do25}. 
\end{rmk}

As a first application, we derive the exact convergence rate of the coverage error of (infeasible) Gaussian approximation in the spherical case. This will serve as a benchmark. 
Recall that $c_p^G$ denotes the $p$-quantile of $Z^\vee$. 
\begin{corollary}[Coverage error of Gaussian approximation in the spherical case]\label{coro:lb-ga}
Under the assumptions of \cref{coro1}, suppose additionally that $\Sigma=\sigma I_d$ for some $\sigma>0$. 
Suppose also that $d\to\infty$, $\log^3d=o(n/\log^2n)$ and $b/\sigma=O(1)$ as $n\to\infty$. 
Further, suppose that $(nd)^{-1}\sum_{i=1}^n\sum_{j=1}^d\E[(X_{ij}/\sigma)^3]\to\gamma_X$ as $n\to\infty$. Then, for any $\alpha\in(0,1)$,
\ba{
\sqrt{\frac{n}{\log^3d}}\bra{P(T_n\geq c^G_{1-\alpha})-\alpha}\to-\frac{\sqrt 2}{3}\gamma_X\log(1-\alpha)\quad(n\to\infty).
}
\end{corollary} 

We turn to Edgeworth expansion for $S_n^*$. Its validity is ensured if the weight variables have Stein kernels:
\begin{theorem}[Edgeworth expansion for $S_n^*$]\label{prop:boot}
Assume \eqref{ass:psi1}. 
Also, suppose that $w_1$ satisfies either of the following conditions:
\begin{enumerate}[label=(\roman*)]

\item\label{ass:weight} $w_1$ has a Stein kernel $\tau^*$ and there exists a constant $b_w\geq1$ such that $|w_1|\leq b_w$ and $|\tau^*(w_1)|\leq b_w^2$. 

\item $w_1\sim N(0,1)$. We set $b_w=1$ in this case. 

\end{enumerate}
Further, assume $\log^3d\leq n$. 
Set $\gamma:=\E[w_1^3]$. Then we have
\ben{\label{eq:ae-boot}
\sup_{A\in\mcl R}\abs{P^*(S_n^*\in A)-\int_{A}\hat{p}_{n,\gamma}(z)dz}
\leq C\frac{b_w^5b^5}{\sigma_*^5}\frac{\log ^3(dn)}{n}\log n
}
with probability at least $1-1/n$.
\end{theorem}

We can construct a random variable $w_1$ satisfying Condition \ref{ass:weight} and $\E[w_1^3]=1$ as follows: Let $\eta$ be a random variable following the beta distribution with parameters $\alpha,\beta>0$. 
Then $w:=(\eta-\E[\eta])/\sqrt{\Var[\eta]}$ satisfies \ref{ass:weight} by \cite[Example 4.9]{LRS17} and \cref{sk-affine}. Also, we have
\[
\E[w_1^3]=\frac{2(\beta-\alpha)\sqrt{\alpha+\beta+1}}{(\alpha+\beta+2)\sqrt{\alpha\beta}}
=\frac{2(1-2\mu)\sqrt{1+\nu}}{(2+\nu)\sqrt{\mu(1-\mu)}},
\]
where $\mu=\alpha/(\alpha+\beta)$ and $\nu=\alpha+\beta$. 
From this expression, given a positive constant $\nu>0$, we have $\E[w_1^3]=1$ if we set
\ben{\label{beta-param}
\alpha=\nu\frac{c-(2+\nu)\sqrt{c}}{2c},\quad
\beta=\nu\frac{c+(2+\nu)\sqrt{c}}{2c}\quad
\text{with }c=\nu^2+20\nu+20.
}

A drawback of \cref{prop:boot} is that two-point distributions do not admit Stein kernels  (cf.~\cite[Theorem 1]{Do25}). 
In particular, it does not cover Mammen's wild bootstrap (cf.~Eq.\eqref{eq:mammen}) examined in the simulation study of \cite{DeZh20}. 
However, the above standardized beta distribution becomes closer to Mammen's two-point distribution as $\nu$ is closer to 0, and their numerical difference virtually vanishes. 
Our simulation study shows that the beta wild bootstrap with $\nu=0.1$ performs very similarly to Mammen's one. 

\subsection{Asymptotic expansion of coverage probability}\label{sec:coverage}

Next, we develop an asymptotic expansion of the bootstrap coverage probability $P(T_n\geq\hat c_{1-\alpha})$ and investigate its implications. 
Before proceeding, we introduce some notation used throughout this section. 
For $t\in\mathbb R$, we set $A(t):=(-\infty,t]^d$. 
We denote by $f_\Sigma$ the density of $Z^\vee$, where $Z\sim N(0,\Sigma)$. 
Note that $f_\Sigma$ is a $C^\infty$ function since $\Sigma$ is invertible. 
Also, we set $\varsigma_d:=\sqrt{\Var[Z^\vee]\log d}$. 
By \cref{gmax-var}, $\varsigma_d$ is bounded from below by a positive constant depending only on $\ol\sigma$ and $\ul\sigma$. 
By \cref{gmax-mom}, $\varsigma_d$ is generally bounded by $\ol\sigma\sqrt{\log d}$, but we often have $\varsigma_d=O(1)$ as $d\to\infty$, known as a \emph{superconcentration} phenomenon (cf.~\cite{Ch14}). 
For example, this is the case when $\Sigma_{jj}=1$ for all $j$ and there exists a constant $C>0$ such that $\Sigma_{jk}\leq C/\log(2+|j-k|)$ for all $j,k$. This follows from \cite[Theorem 9.12]{Ch14}. 


%
\begin{theorem}[Asymptotic expansion of bootstrap coverage probability]\label{thm:coverage}
Under the assumptions of Theorems \ref{coro1} and \ref{prop:boot}, let $\lambda>0$ be a constant such that $b/\sigma_*\leq\lambda$. 
Then, for any $\eps\in(0,1/2)$, there exist positive constants $c$ and $C$ depending only on $\lambda,\eps$ and $b_w$ such that if
\ben{\label{cf-boot-ass}
\frac{\varsigma_d^3}{\sigma_*^3}\frac{\log^3(dn)}{n}\log n\leq c,
}
then
\ba{
\sup_{\eps<\alpha<1-\eps}\abs{P(T_n\geq \hat c_{1-\alpha})-\bra{\alpha-(1-\gamma)Q_n(c_{1-\alpha}^G)-\E[R_n(\alpha)]}}
\leq C\frac{\varsigma_d^3}{\sigma_*^3}\frac{\log ^3(dn)}{n}\log n,
}
where 
\[
Q_n(t):=\int_{A(t)}\{p_n(z)-\phi_\Sigma(z)\}dz=-\frac{1}{6\sqrt n}\langle\E[\ol{X^{3}}],\int_{A(t)}\nabla^3\phi_\Sigma(z)dz\rangle,\qquad t\in\mathbb R,
\]
and
\[
R_n(\alpha):=\frac{1}{\sqrt n}\frac{\abra{\ol{X^3}\otimes\bs1_d,\Psi_{\alpha}^{\otimes2}}}{2f_{\Sigma}(c_{1-\alpha}^G)},\qquad
\Psi_\alpha:=\int_{A(c_{1-\alpha}^G)}\nabla^2\phi_\Sigma(z)dz.
\]
\end{theorem}

\begin{rmk}[Univariate case]\label{rmk:uni-ae}
When $d=1$ and $\Sigma=1$, the above asymptotic expansion formula reduces to 
\[
\begin{cases}
\alpha-\frac{\E[\ol{X^3}]}{6\sqrt n}\{2(c_{1-\alpha}^G)^2+1\}\phi(c_{1-\alpha}^G) &\text{if }\gamma=0,\\
\alpha-\frac{\E[\ol{X^3}]}{2\sqrt n}(c_{1-\alpha}^G)^2\phi(c_{1-\alpha}^G) &\text{if }\gamma=1.
\end{cases}
\] 
These recover the asymptotic expansion formulae for normal and empirical bootstrap coverage probabilities, respectively; see e.g.~\cite[Eqs.(2)--(3)]{LiSi87} (note that $c_{1-\alpha}^G=\Phi^{-1}(1-\alpha)=-\Phi^{-1}(\alpha)$ when $d=1$). 
\end{rmk}

\begin{rmk}[Proof strategy]\label{rmk:coverage}
A basic strategy to prove \cref{thm:coverage} is the same as the classical one (cf.~Section 3.5.2 in \cite{Ha92}): We replace the ``complicated'' random variable $\hat c_{1-\alpha}$ by its asymptotic expansion (known as \emph{Cornish--Fisher expansion}). 
After this replacement, it turns out that the proof amounts to Edgeworth expansion for the maximum component of a sum of independent random vectors with (approximate) Stein kernels, which can be handled by a similar strategy to the proof of \cref{coro1}. 

A major difficulty specific to our setting arises when we derive an error bound for Cornish--Fisher expansion. This is due to the fact that the ``centering'' distribution function $F_Z(t)=P(Z^\vee\leq t)$ for our Edgeworth expansions depends on $d$ and may not converge to a non-degenerate distribution function as $d\to\infty$. 
In particular, bounding derivatives of $F_Z^{-1}$ is non-trivial in our setting. 
To resolve this issue, we develop a novel isoperimetric-type inequality for $Z^\vee$ suitable for our purpose (see \cref{gmax-quantile}). 
We remark that a Gaussian-type isoperimetric inequality for $Z^\vee$ can be derived from Gaussian isoperimetry, but the dimension dependence is less sharp than ours; see \cref{rmk:gmax-quantile} for details. 
\end{rmk}

Now we discuss implications of \cref{thm:coverage}. 
An easy consequence is that any wild bootstrap approximation is second-order accurate when $\E[\ol{X^3}]=0$ as long as $w_1$ satisfies the assumptions in \cref{prop:boot}. 
However, simulation results suggest that the choice of $w_1$ would affect the performance even when $\E[\ol{X^3}]=0$, so there is still room to investigate. 

The following corollary gives a more interesting implication:
\begin{corollary}\label{coro:coverage}
Under the assumptions of \cref{thm:coverage}, suppose additionally that $\E[w_1^3]=1$, $\eps\geq 2e^{-d/2}$, $\ul\sigma=\ol\sigma=:\sigma$ and the maximum eigenvalue of $\Sigma$ is bounded by $K\sigma^2$ with some constant $K>0$. Then there exist a constant $C>0$ depending only on $\lambda,\eps,K$ and $b_w$ such that
\ben{\label{eq:coro:coverage}
\sup_{\eps<\alpha<1-\eps}\abs{P(T_n\geq \hat c_{1-\alpha})-\alpha}
\leq C\bra{\frac{\varsigma_d^3}{\sigma_*^3}\frac{\log ^3(dn)}{n}\log n+\frac{\varsigma_d}{\sigma}\sqrt{\frac{\log^3d}{dn}}}.
}
\end{corollary}
%
Observe that the second term on the right hand side of \eqref{eq:coro:coverage} is divided by $\sqrt d$. Hence, \cref{coro:coverage} implies that the third-moment matching wild bootstrap is second-order accurate if $d\geq n$ and $\Sigma$ has identical diagonal entries and bounded eigenvalues with respect to $d$. 
This seems to be a new result on the blessing of dimensionality, although excessively high dimensionality is harmful due to the first term of the bound. 
We also remark that such an improvement generally does not occur without third-moment matching:
\begin{corollary}
\label{coro:lb-gwb}
Under the assumptions of \cref{coro:lb-ga}, if $w_i$ satisfies the conditions in \cref{prop:boot} with $b_w=O(1)$, then for any $\alpha\in(0,1)$,
\ba{
\sqrt{\frac{n}{\log^3d}}\bra{P(T_n\geq \hat c_{1-\alpha})-\alpha}\to-(1-\E[w_1^3])\frac{\sqrt 2}{3}\gamma_X\log(1-\alpha)\quad(n\to\infty).
}
\end{corollary} 

\begin{rmk}
The additional assumptions on $\Sigma$ in \cref{coro:coverage} are necessary to ensure that $T_n$ is truly high-dimensional:
\begin{enumerate}[label=(\alph*)]

\item The assumption $\ul\sigma = \ol\sigma$ ensures that no single component solely affects variations of $T_n$. 
Formally, this assumption is mainly used to prove $P(Z_j>c_{1-\alpha}^G)=O(d^{-1}\sqrt{\log d})$ as $d\to\infty$ for every $j$ (corresponding to the second inequality of \eqref{quantile-bounds}). 

\item The boundedness of the maximum eigenvalue of $\Sigma$ rules out situations where the components of $X_i$ have a common factor. In such a situation, the common factor solely affects variations of $T_n$. In fact, we obtain the following corollary when this is the case:

\end{enumerate} 
\end{rmk}

\begin{corollary}\label{coro:factor}
Under the assumptions of \cref{thm:coverage}, suppose additionally that $X_1,\dots,X_n$ are i.i.d.~and each has the same law as $\sqrt{\rho}U\bs1_d+\sqrt{1-\rho}V$, where $\rho\in(0,1)$ is a constant, $U$ is a random variable with mean 0 and variance 1, and $V$ is a random vector in $\mathbb R^d$ independent of $U$ with $\E[V]=0$ and $\Cov[V]=I_d$. 
Moreover, suppose that $\rho$ and $U$ do not depend on $d,n$.  
Then, for any $\alpha\in(0,1)$,
\besn{\label{eq:factor}
P(T_n\geq \hat c_{1-\alpha})-\alpha
&=\frac{\E[U^3]}{\sqrt n}\bra{\frac{\gamma-1}{6}(z_\alpha^2-1)+\frac{z_\alpha^2}{2}}\phi(z_\alpha)\\
&\quad+(1-\rho)^{\frac{3}{2}}\Upsilon_{n,\gamma}(\alpha)
+o(n^{-1/2})
}
as $d,n\to\infty$, provided that $b=O(1)$ and $\log^9d=o(n/\log^2n)$. 
Here, we set $z_\alpha:=\Phi^{-1}(\alpha)$ and
\[
\Upsilon_{n,\gamma}(\alpha):=
-\frac{1-\gamma}{6\sqrt n}\langle\E[V^{\otimes3}],\int_{A(c^G_{1-\alpha})}\nabla^3\phi_\Sigma(z)dz\rangle
+\frac{1}{2\sqrt n}\frac{\abra{\E[V^{\otimes3}]\otimes\bs1_d,\Psi_{\alpha}^{\otimes2}}}{f_{\Sigma}(c_{1-\alpha}^G)}.
\]
Moreover, if $\gamma=1$, then $\Upsilon_{n,\gamma}(\alpha)=o(n^{-1/2})$ as $d,n\to\infty$. 
\end{corollary}

If $\E[V^{\otimes3}]=0$ in \cref{coro:factor}, the asymptotic expansion formula \eqref{eq:factor} has the same form as in the univariate case (cf.~\cref{rmk:uni-ae}). 
Therefore, in such a situation the third-moment matching wild bootstrap underperforms the Gaussian wild bootstrap if $|z_\alpha|>1$ according to \cite[Section 3]{LiSi87}. 
We expect that a similar phenomenon would occur when $\rho$ is close to 1; see \cref{sec:simulate} for numerical evidence. 
On the other hand, \cref{coro:factor} also shows that the coverage error of the third-moment matching wild bootstrap is always of order $O(n^{-1/2})$ regardless of the value of $\E[V^{\otimes3}]$. 
The simulation results in \cref{sec:simulate} suggest that the Gaussian wild bootstrap would not enjoy this property because its performance in Design (I) significantly worsens when the value of $\rho$ is close to 0. 


\subsection{Double wild bootstrap}\label{sec:db}

As mentioned in the introduction, the lack of second-order accuracy in standard bootstrap methods is due to the fact that $T_n$ is not asymptotically pivotal. 
If we knew the distribution function of $T_n$, say $F_n$, then $F_n(T_n)$ would give an (exactly) pivotal statistic. \citet{Be88} suggested estimating $F_n$ by the bootstrap distribution function $\hat F_n(t)=P^*(T_n^*\leq t)$ and use $\hat F_n(T_n)$ to construct critical values. 
This method is called \emph{bootstrap prepivoting}. Note that $\hat F_n$ can be computed by simulating the conditional law of $T_n^*$ given the data. 
To estimate the law of $\hat F_n(T_n)$, we use the following nested double wild bootstrap procedure following \cite{Be88}: 
Let $v_1,\dots,v_n$ be i.i.d.~variables independent of everything else and such that $\E[v_1]=0$ and $\E[v_1^2]=1$. 
We define the wild bootstrap statistic of $S_n^*$ as
\[
S_n^{**}=\frac{1}{\sqrt n}\sum_{i=1}^nv_i(X_i^*-\bar X^*),\quad\text{where }X^*_i=w_i(X_i-\bar X),~\bar X^*=\frac{1}{n}\sum_{i=1}^nX_i^*.
\]
Then define $\hat F^*_n(t)=P^{**}(T_n^{**}\leq t)$ for $t\in\mathbb R$, where $T_n^{**}:=\max_{1\leq j\leq d}S_{n,j}^{**}$ and $P^{**}$ is the conditional probability given $X_1,\dots,X_n,w_1,\dots,w_n$. 
We regard $\hat F^*_n(T_n^*)$ as a bootstrap version of $\hat F_n(T_n)$ and estimate the law of $\hat F_n(T_n)$ by the conditional law of $\hat F^*_n(T_n^*)$. 
Formally, given a significance level $\alpha\in(0,1)$, let $\hat\beta_\alpha$ be the conditional $(1-\alpha)$-quantile of $\hat F_n^{*}(T_n^*)$ given the data. 
We expect that $P(\hat F_n(T_n)\geq\hat\beta_\alpha)=P(T_n\geq\hat c_{\hat\beta_\alpha})$ would be close to $\alpha$. This is formally justified by the following theorem: 
\begin{theorem}[Second-order accuracy of double bootstrap coverage probability]\label{thm:db}
Suppose that the assumptions of \cref{thm:coverage} are satisfied. 
Suppose also that $v_1$ has a Stein kernel $\tau^{**}$ and there exists a constant $b_v\geq1$ such that $|v_1|\leq b_v$ and $|\tau^{**}(v_1)|\leq b_v^2$. 
Further, assume $\E[w_1^3]=\E[v_1^3]=1$. 
Then, for any $\eps\in(0,1/4)$, there exists a constant $C>0$ depending only on $\lambda,\eps,b_w$ and $b_v$ such that
\ben{\label{db-aim}
\sup_{2\eps<\alpha<1-2\eps}\abs{P\bra{T_n\geq \hat c_{\hat\beta_\alpha}}-\alpha}
\leq C\frac{\varsigma_d^4}{\sigma_*^4}\frac{\log ^3(dn)}{n}\log n.
}
\end{theorem}


\begin{rmk}[Proof strategy]\label{rmk:db-proof}
A key fact to prove \cref{thm:db} is that the coverage probability for the second-level bootstrap enjoys the same asymptotic expansion formula as $P(T_n\geq\hat c_{1-\alpha})$ with high probability; see \cref{db-coverage}. 
Equivalently, the distribution functions of $\hat F_n(T_n)$ and $\hat F_n^*(T_n^*)$ share the same asymptotic expansion formula with high probability. 
This enables us to prove \eqref{db-aim} by a standard argument to prove the accuracy of bootstrap approximation. 
\end{rmk}

\begin{rmk}[$p$-value]
One can easily check that $T_n\geq \hat c_{\hat\beta_\alpha}$ is equivalent to $P^*(\hat p^*_n\leq\hat p_n)\leq\alpha$, where $\hat p_n:=1-\hat F_n(T_n)$ and $\hat p_n^*:=1-\hat F_n^*(T_n^*)$ are the $p$-values of the first- and second-level bootstraps, respectively. 
Hence the $p$-value of the double bootstrap method is $P^*(\hat p^*_n\leq\hat p_n)$. 
\end{rmk}

\begin{rmk}[Computational cost]\label{db-cost}
It is well-known that the double bootstrap often entails a prohibitively high computational cost, as it requires nested resampling of the data (cf.~\cite[Section 3.11.1]{Ha92}). 
Specifically, if $B_1$ and $B_2$ denote the numbers of the first- and second-level bootstrap replications, respectively, a total of $1+B_1+B_1B_2$ test statistics must be computed. 
For this reason, several authors have investigated ways to reduce its computational burden. 
For instance, in the context of constructing confidence intervals, \citet{LeYo99} showed that the coverage error due to the Monte Carlo approximation in the double bootstrap is of order $o(B_1^{-1/2})+O(B_2^{-1})$, suggesting that $B_2$ may be chosen much smaller than $B_1$. 
Moreover, for a fixed confidence level, \citet{Na05} proposed stopping rules for the second-level bootstrap to avoid unnecessary computation that does not affect the resulting confidence interval. 
As an alternative approach, \citet{DaMa07} proposed the \emph{fast double bootstrap (FDB)}, which requires only a single second-level bootstrap sample for each first-level one. While the FDB dramatically reduces the computational burden relative to the standard double bootstrap, it does not necessarily improve the order of accuracy in approximating coverage probabilities from a theoretical point of view (see \citet{ChHa15}). 
\end{rmk}

\section{Simulation study}\label{sec:simulate}

This section conducts a small Monte Carlo study to supplement our theoretical findings. 
We adopt a similar simulation design to \cite{DeZh20}: We set $d=400$ and generate the data from a Gaussian copula model, i.e.~$X_1,\dots,X_n$ are i.i.d.~with the same law as $U-\E[U]$, where $U$ is defined as in \cref{ex:copula}. 
The marginal distributions $\mu_j$ are the gamma distribution with shape parameter 1 and unit scale. 
As the parameter matrix $R$, we consider two designs: (I) $R=\rho\bs1_d^{\otimes2}+(1-\rho)I_d$ and (II) $R=(\rho^{|j-k|})_{1\leq j,k\leq d}$. Here, the parameter $\rho$ is varied as $\rho\in\{0.2,0.8\}$. 
We also vary the sample size as $n\in\{200,400\}$. 
We compute the rejection rate $P(T_{n}\geq \hat c)$ at the 10\% significance level based on 20,000 Monte Carlo iterations, where $\hat c$ is an estimated 90\% quantile of the corresponding statistic using various bootstrap methods. 
In addition, to assess the performance when the skewness of the data is zero, we also consider the case that $X_i\overset{d}{=}U-U'$, where $U'$ is an independent copy of $U$. To keep the marginal kurtosis at the same level, we change the shape parameter of the gamma distribution to $0.5$ in this case. 
We remark that these models satisfy \eqref{ass:psi1} and \eqref{ass:sk} thanks to \cref{prop:copula}. 

For the bootstrap methods, we consider the empirical bootstrap (EB), wild bootstrap and double wild bootstrap (DB) methods.  For the wild bootstrap, we consider the following 4 types of weight variables:
\begin{description}


\item[\bf GB] $w_1$ is a standard normal variable. 

\item[\bf MB] $w_1$ follows Mammen's two point distribution \cite{Ma93}: 
\ben{\label{eq:mammen}
P\bra{w_1=\frac{\sqrt5+1}{2}}=1-P\bra{w_1=-\frac{\sqrt5-1}{2}}=\frac{\sqrt5-1}{2\sqrt 5}.
}

\item[\bf RB] $w_1$ is a Rademacher variable: $P(w_1=\pm1)=1/2$. 

\item[\bf BB] $w_1$ follows the standardized beta distribution with parameters given by \eqref{beta-param} with $\nu=0.1$. 

\end{description}
The double wild bootstrap is implemented with both $w_1$ and $v_1$ generated from the standardized beta distribution with parameters given by \eqref{beta-param} with $\nu=0.1$. 
Note that our theoretical results are applicable only to GB, BB and DB. 
We include EB, MB and RB in our assessment because they are commonly used in the literature. 
The number of bootstrap replications is set to 499 for the first-level bootstrap and 99 for the second-level bootstrap in DB (see \cite[Appendix II.8]{Ha92} and \cite{LeYo99} for discussions about the number of bootstrap replications). 

We summarize the simulation results in Tables \ref{table:asym} and \ref{table:sym}. 
First, \cref{table:asym} reports empirical rejection rates at the 10\% level when the laws of $X_i$ are asymmetric. 
We find that the difference in performance between GB and BB is largely in line with our Corollaries \ref{coro:coverage} and \ref{coro:factor}: 
First, in Design (I), GB outperforms BB when $\rho=0.8$, while BB slightly outperforms GB when $\rho=0.2$. 
This is consistent with \cref{coro:factor}: When $\rho$ is close to 1, the first term on the right hand side of \eqref{eq:factor} dominates the coverage error, and this term has the same form as the univariate case. 
According to computations in \cite[Section 3]{LiSi87}, this implies that the coverage error of GB would be smaller than that of BB. 
On the other hand, when $\rho$ is closer to 0, the second term on the right hand side of \eqref{eq:factor} would be non-negligible for GB, so the above argument would be invalid in this case. 
Next, in Design (II), BB performs much better than GB, which is consistent with \cref{coro:coverage}: Since $\Sigma$ has no large eigenvalues in this design, BB is second-order accurate by \cref{coro:coverage}. 

Turning to the performance of DB, it uniformly outperforms the other methods when $n=400$. 
It tends to over-reject when $n=200$, but it still outperforms the others in Design (I). 
In Design (II), BB outperforms DB when $n=200$. 
This would be because BB is also second-order accurate in this design \emph{thanks to the high dimensionality} by \cref{coro:coverage}. 
Since our theory suggests that an increase in the sample size $n$ would improve the performance of DB more than the others, it is expected that DB performs much better when $n=400$. 

Next, \cref{table:sym} reports empirical rejection rates at the 10\% level when the laws of $X_i$ are symmetric. 
Recall that our \cref{thm:coverage} implies that both GB and BB are second-order accurate in this case. 
Reflecting this fact, GB clearly performs better than in the asymmetric case. 
The performance of BB is improved in Design (I) but not in Design (II). The latter would be due to the same reasoning as above, i.e.~ BB is second-order accurate in Design (II) even when the skewness is not zero by \cref{coro:coverage}. 
By contrast, the performance of DB is not improved. This is not surprising because DB is already second-order accurate in the asymmetric case and the zero skewness condition would not contribute to its performance. 
When comparing GB and BB, BB still outperforms GB. 
This may be due to an effect of kurtosis, but we will need higher-order asymptotic expansions for the formal discussion and leave it to future work. 


Finally, we briefly discuss the performances of EB, MB and RB. 
First, EB tends to under-reject and its performance is not pronounced compared to other methods. 
In fact, we can observe similar phenomena in the simulation results of \cite{DeZh20,CCKK22}. 
Formally, this does not contradict our theory because we have no valid Edgeworth expansion for EB in high-dimensions, although it is unclear whether this is an artifact of our proof strategy. 
Next, although MB is not covered by our theory, its performance is similar to BB. 
This is perhaps explained by the fact that their weights are very close numerically. 
Third, RB performs remarkably well in the symmetric case. 
This is already observed in the simulation study of \cite{CCKK22} who explain this phenomenon by their Theorem 2.3. 
Another possible explanation is the match of higher moments, but we have no formal theoretical result for this so far. 

\begin{table}[ht]
\centering
\caption{Rejection rate at the 10\% level (Asymmetric case)} 
\label{table:asym}
\begin{tabular}{cccccccc}
  \hline
$n$ & $\rho$ & EB & GB & MB & RB & BB & DB \\ 
  \hline &&\multicolumn{6}{l}{(I) $R_{jk}=\rho+(1-\rho)1_{\{j=k\}}$}\\200 & 0.2 & 0.061 & 0.124 & 0.080 & 0.155 & 0.078 & 0.114 \\ 
   & 0.8 & 0.071 & 0.090 & 0.072 & 0.093 & 0.071 & 0.101 \\ 
  400 & 0.2 & 0.072 & 0.122 & 0.083 & 0.140 & 0.085 & 0.107 \\ 
   & 0.8 & 0.080 & 0.095 & 0.081 & 0.096 & 0.080 & 0.102 \\ 
   &&\multicolumn{6}{l}{(II) $R_{jk}=\rho^{|j-k|}$}\\200 & 0.2 & 0.065 & 0.146 & 0.092 & 0.195 & 0.091 & 0.117 \\ 
   & 0.8 & 0.069 & 0.139 & 0.089 & 0.177 & 0.088 & 0.113 \\ 
  400 & 0.2 & 0.074 & 0.135 & 0.090 & 0.160 & 0.088 & 0.101 \\ 
   & 0.8 & 0.079 & 0.135 & 0.092 & 0.155 & 0.091 & 0.104 \\ 
   \hline
\end{tabular}
\end{table}

\begin{table}[ht]
\centering
\caption{Rejection rate at the 10\% level (Symmetric case)} 
\label{table:sym}
\begin{tabular}{cccccccc}
  \hline
$n$ & $\rho$ & EB & GB & MB & RB & BB & DB \\ 
  \hline &&\multicolumn{6}{l}{(I) $R_{jk}=\rho+(1-\rho)1_{\{j=k\}}$}\\200 & 0.2 & 0.065 & 0.076 & 0.083 & 0.100 & 0.082 & 0.114 \\ 
   & 0.8 & 0.089 & 0.092 & 0.091 & 0.096 & 0.091 & 0.105 \\ 
  400 & 0.2 & 0.081 & 0.088 & 0.090 & 0.099 & 0.092 & 0.105 \\ 
   & 0.8 & 0.095 & 0.099 & 0.098 & 0.099 & 0.095 & 0.103 \\ 
   &&\multicolumn{6}{l}{(II) $R_{jk}=\rho^{|j-k|}$}\\200 & 0.2 & 0.062 & 0.071 & 0.085 & 0.101 & 0.084 & 0.114 \\ 
   & 0.8 & 0.067 & 0.076 & 0.088 & 0.100 & 0.086 & 0.109 \\ 
  400 & 0.2 & 0.079 & 0.084 & 0.092 & 0.103 & 0.092 & 0.106 \\ 
   & 0.8 & 0.083 & 0.087 & 0.096 & 0.102 & 0.094 & 0.106 \\ 
   \hline
\end{tabular}
\end{table}


\section{Proofs for Section \ref{sec:2nd-main}}\label{sec:proof-sec2}

We use the following notation in the remainder of the paper: For two random variables $\xi$ and $\eta$, we write $\xi\lesssim\eta$ or $\eta\gtrsim\xi$ if there exists a \emph{universal} constant $C>0$ such that $\xi\leq C\eta$. 
Also, given real numbers $\theta_1,\dots,\theta_m$, we use $C_{\theta_1,\dots,\theta_m}$ to denote positive constants, which depend only on $\theta_1,\dots,\theta_m$ and may be different in different expressions.

\subsection{General error bounds via approximate Stein kernel}\label{sec:ae-general}

We first develop error bounds for high-dimensional Edgeworth expansion in a somewhat general form. 
For later use, we consider a situation where a Stein kernel exists in an approximate sense: 
\begin{definition}[Approximate Stein kernel]
Let $\xi$ be a random vector in $\mathbb R^d$ with $\E[\|\xi\|_\infty]<\infty$. 
Given measurable functions $\tau:\mathbb R^d\to\mathbb R^d\otimes\mathbb R^d$ and $\res:\mathbb R^d\to\mathbb R^d$, we call $(\tau,\res)$ an \emph{approximate Stein kernel} for (the law of) $\xi$ if $\E[\|\tau(\xi)\|_\infty]+\E[\|\res(\xi)\|_\infty]<\infty$ and 
\begin{equation*}
\E[(\xi-\E[\xi])\cdot\nabla h(\xi)]=\E[\langle\tau(\xi),\nabla^2h(\xi)\rangle]+\E[\res(\xi)\cdot\nabla h(\xi)]
\end{equation*}
for any $h\in C^2_b(\mathbb R^d)$.
\end{definition}

With this definition, we have the following results:
\begin{theorem}[Error bound for Edgeworth expansion via approximate Stein kernel]\label{thm:main}
Let $\xi_1,\dots,\xi_n$ be independent random vectors in $\mathbb R^d$ with mean 0 and finite variance. 
Set $W:=\sum_{i=1}^n\xi_i$ and $\Sigma_W:=\Cov[W]$. 
For each $i=1,\dots,n$, suppose that $\xi_i$ has an approximate Stein kernel $(\tau_i,\res_i)$ such that $\E[\|\xi_i\|_\infty^5]+\E[\|\tau_i(\xi_i)\|_\infty^{5/2}]+\E[\|\res_i(\xi_i)\|_\infty^{5/2}]<\infty$ for all $i=1,\dots,n$. 
Set $T=\sum_{i=1}^n\tau_i(\xi_i)$, $\bar T=T-\Sigma$, $\Res=\sum_{i=1}^n\res_i(\xi_i)$ 
and
\ben{\label{def:pw}
p_W(z)=\phi_\Sigma(z)+\frac{1}{2}\langle\Sigma_W-\Sigma,\nabla^2\phi_\Sigma(z)\rangle-\frac{1}{6}\langle\E[W^{\otimes3}],\nabla^3\phi_\Sigma(z)\rangle,\quad z\in\mathbb R^d.
}
Then, there exists a universal constant $C>0$ such that for any $t\in(0,1/2]$,
\ban{
&\sup_{A\in\mcl R}\abs{P(W\in A)-\int_{A}p_W(z)dz}
\notag\\
&\leq C|\log t|\frac{\log^2d}{\sigma_*^4}\bra{
\E\|\bar T\|_\infty^2
+\E\norm{\sum_{i=1}^n\tau_i(\xi_i)^{\otimes2}}_\infty
+\E\norm{\sum_{i=1}^n\xi_i^{\otimes4}}_\infty
}
\notag\\
&\quad+C\frac{\log^{5/2}d}{\sigma_*^{5}}\bra{
\E\norm{\bar T\otimes\sum_{i=1}^n\xi_i^{\otimes3}}_\infty
+\E\norm{\sum_{i=1}^n\xi_i^{\otimes3}\otimes\tau_i(\xi_i)}_\infty
}
\notag\\
&\quad+C\frac{\log^{2}d}{\sigma_*^{4}}\left(
\E\norm{\Res\otimes\sum_{i=1}^n \xi_i^{\otimes3}}_\infty
+\E\norm{\sum_{i=1}^n \xi_i^{\otimes3}\otimes\res_i(\xi_i)}_\infty\right.
\notag\\
&\qquad\qquad\qquad\left.+\E\norm{\bar T\otimes\sum_{i=1}^n \res_i(\xi_i)\otimes\xi_i}_\infty
+\E\norm{\sum_{i=1}^n\res_i(\xi_i)\otimes\xi_i\otimes\tau_i(\xi_i)}_\infty
\right)
\notag\\
&\quad+C\frac{\log d}{\sigma_*^2}\bra{
\E\|\Res\|_\infty^2
+\E\norm{\sum_{i=1}^n\res_i(\xi_i)^{\otimes2}}_\infty
}
\notag\\
&\quad+C\frac{\log^{3/2} d}{\sigma_*^3}\bra{
\E\norm{\Res\otimes\sum_{i=1}^n \res_i(\xi_i)\otimes\xi_i}_\infty
+\E\norm{\sum_{i=1}^n\res_i(\xi_i)^{\otimes2}\otimes\xi_i}_\infty
}
\notag\\
&\quad+C\ol\sigma\sqrt t\bra{\frac{\log d}{\ul\sigma}
+\frac{\log^2d}{\sigma_*^3}\|\Sigma_W-\Sigma\|_\infty
+\frac{\log^{5/2}d}{\sigma_*^4}\norm{\sum_{i=1}^n\E[\xi_i^{\otimes3}]}_\infty
}.
\label{eq:main}
}
\end{theorem}

\begin{theorem}[Refined Gaussian comparison inequality]\label{thm:gcomp}
Let $W$ be a centered Gaussian vector in $\mathbb R^d$ with covariance matrix $\Sigma_W$. 
Then, there exists a universal constant $C>0$ such that for any $t\in(0,1/2]$,
\ba{
&\sup_{A\in\mcl R}\abs{P(W\in A)-\int_{A}p_W(z)dz}\\
&\leq C|\log t|\frac{\log^2d}{\sigma_*^4}\|\Sigma_W-\Sigma\|_\infty^2
+C\ol\sigma\sqrt t\bra{\frac{\log d}{\ul\sigma}
+\frac{\log^2d}{\sigma_*^3}\|\Sigma_W-\Sigma\|_\infty},
}
where $p_W$ is defined by \eqref{def:pw}; note that $\E[W^{\otimes3}]=0$ in the present case.  
\end{theorem}

The second result will be used for the Gaussian wild bootstrap. 

The remainder of this subsection is devoted to the proofs of Theorems \ref{thm:main} and \ref{thm:gcomp}. 
As usual, the proof starts with a smoothing inequality. We will use the following version.  
\begin{lemma}\label{l3}
Let $\mu$ be a finite measure, $\nu$ a finite signed measure, and $K$ a probability measure on $\mathbb{R}^d$. 
Let $\eps>0$ be a constant such that 
$
\alpha:=K([-\eps,\eps]^d)>1/2.
$ 
Let $h: \mathbb{R}^d\to \mathbb{R}$ be a bounded measurable function. 
Then we have
\be{
\abs{\int h d(\mu-\nu) }\leq (2\alpha-1)^{-1}[\gamma^*(h;\eps)+\tau^*(h; \eps)+\alpha\tilde\tau^*(h; \eps)],
}
where
\ba{
\gamma^*(h;\eps)=\sup_{y\in \mathbb{R}^d} \gamma(h_y;\eps),\qquad
\tau^*(h;\eps)=\sup_{y\in \mathbb{R}^d} \tau(h_y;\eps),\qquad
\tilde\tau^*(h;\eps)=\sup_{y\in \mathbb{R}^d} \tilde\tau(h_y;\eps),
}
with $h_y(x)=h(x+y)$, 
\ba{
\gamma(h;\eps)&=\max\cbra{\int M_h(x; \eps) (\mu-\nu)*K(dx), -\int m_h(x; \eps)  (\mu-\nu)*K(dx)   },\\
\tau(h;\eps)&= \max\cbra{\int [M_h(x; \eps)-h(x)]\nu(dx),\int [h(x)-m_h(x; \eps)]\nu(dx)},\\
\tilde\tau(h;\eps)&= \sup_{y\in [-\eps,\eps]^d}\abs{\int[h(x+y)-h(x)]\nu(dx)},\\
M_h(x;\eps)&=\sup_{y: \norm{y-x}_\infty\leq \eps} h(y),\qquad m_h(x;\eps)=\inf_{y: \norm{y-x}_\infty\leq \eps} h(y),
}
and $*$ denotes the convolution of two finite signed measures.
\end{lemma}

The proof of this lemma is a straightforward modification of \cite[Lemma 11.4]{BhRa10} and given in Appendix \ref{sec:smoothing}, but its statement contains an important difference from the original one: The bound does not contain the positive part of the signed measure $\nu$. This is important for bounding $\tau^*(h;\eps)$ and $\tilde\tau^*(h;\eps)$ in our setting. To bound these quantities, we will use the following anti-concentration inequality. 
For $A=\prod_{j=1}^d[a_j,b_j]\in\mcl R$ and $u,v\in\mathbb R^d_+:=[0,\infty)^d$, define $A^{u,v}=\prod_{j=1}^d[a_j-u_j,b_j+v_j]$. 
\begin{lemma}\label{lem:anti}
Let $r\in\mathbb N$. Then
\ben{\label{eq:anti}
\sup_{A\in\mcl R}\sup_{\eps>0}\sup_{u,v\in [0,\eps]^d}\frac{1}{\eps}\norm{\int_{A^{u,v}\setminus A}\nabla^r\phi_\Sigma(z) dz}_1\leq C_r\frac{(\log d)^{(r+1)/2}}{\sigma_*^{r+1}},
}
where $C_r>0$ is a constant depending only on $r$.
\end{lemma}

\begin{proof}
See Appendix \ref{sec:anti}. 
\end{proof}

We will apply \cref{l3} with $K=N(0,t\Sigma)$. To bound the quantity $\gamma^*(h;\eps)$, we introduce some notation and lemmas. 
Given a bounded measurable function $h:\mathbb R^d\to\mathbb R$ and $s\in[0,1]$, we define a function $h_s:\mathbb R^d\to\mathbb R$ as
\ben{\label{def:smoothed}
h_s(x)=\E[h(\sqrt{1-s}x+\sqrt sZ)],\qquad x\in\mathbb R^d,
}
where $Z\sim N(0,\Sigma)$. 
When $s>0$, $h_s(x)$ can be rewritten as
\be{
h_s(x)=s^{-d/2}\int_{\mathbb R^d}h(z)\phi_{\Sigma}\bra{\frac{z-\sqrt{1-s}x}{\sqrt s}}dz.
}
By this expression, $h_s$ is infinitely differentiable and
\ban{
\nabla^rh_s(x)
&=\bra{-\sqrt{\frac{1-s}{s}}}^r\int_{\mathbb R^d}h(\sqrt{1-s}x+\sqrt sz)\nabla^r\phi_{\Sigma}(z)dz
\label{eq:ht-deriv}
}
for any $r\in\mathbb N$. 
In particular, $h_s\in C^\infty_b(\mathbb R^d)$.
We will use the following lemmas to bound $\gamma^*(h;\eps)$. 
\begin{lemma}\label{lem:decomp}
Let $h:\mathbb R^d\to\mathbb R$ be a bounded measurable function and $t\in(0,1]$. 
Then, under the assumptions of \cref{thm:main},
\ban{
&\E[h_t(W)]-\int_{\mathbb R^d} h_t(z)p_W(z)dz
\notag\\
&=\frac{1}{4}\int_t^1\bra{\int_s^1\weight_u^4\bra{\E[\langle \bar T^{\otimes2},\nabla^4 h_u(W)\rangle]-\sum_{i=1}^n\E[\langle \tau_i(\xi_i)^{\otimes2},\nabla^4 h_u(W)\rangle]}du}ds
\notag\\
&\quad-\frac{1}{8}\int_t^1\bra{\int_s^1\weight_u^4\sum_{i=1}^n\E[\langle \tau_i(\xi_i)\otimes \xi_i^{\otimes2},\nabla^4 h_u(W)\rangle]du}ds
\notag\\
&\quad+\frac{1}{16}\int_t^1\bra{\int_s^1\weight_u\bra{\int_u^1\weight_v^5\sum_{i=1}^n\E[\langle \xi_i^{\otimes3}\otimes(\xi_i+B-\res_i(\xi_i)),\nabla^4 h_v(W)\rangle]dv}du}ds
\notag\\
&\quad+\frac{1}{16}\int_t^1\bra{\int_s^1\weight_u\bra{\int_u^1\weight_v^5\sum_{i=1}^n\E[\langle \xi_i^{\otimes3}\otimes (\bar T-\tau_i(\xi_i)),\nabla^5 h_v(W)\rangle]dv}du}ds
\notag\\
&\quad+\frac{1}{4}\int_t^1\weight_s\bra{\int_s^1\weight_u^2\bra{\E[\langle  B^{\otimes2},\nabla^2 h_u(W)\rangle]-\sum_{i=1}^n\E[\langle \res_i(\xi_i)^{\otimes2},\nabla^2 h_u(W)\rangle]}du}ds
\notag\\
&\quad+\frac{1}{4}\int_t^1\weight_s\bra{\int_s^1\bra{\weight_u^3+\weight_u^4}\sum_{i=1}^n\E[\langle \res_i(\xi_i)\otimes(\bar T-\tau_i(\xi_i)),\nabla^{3} h_u(W)\rangle]du}ds
\notag\\
&\quad-\frac{1}{8}\int_t^1\weight_s\bra{\int_s^1\weight_u^4\sum_{i=1}^n\E[\langle \res_i(\xi_i)\otimes\xi_i^{\otimes2},\nabla^{3} h_u(W)\rangle]du}ds
\notag\\
&\quad+\frac{1}{8}\int_t^1\weight_s\bra{\int_s^1\weight_u\bra{\int_u^1\weight_v^4\sum_{i=1}^n\E[\langle \res_i(\xi_i)\otimes\xi_i\otimes (\bar T-\tau_i(\xi_i)),\nabla^{4} h_v(W)\rangle]dv}du}ds
\notag\\
&\quad+\frac{1}{8}\int_t^1\weight_s\bra{\int_s^1\weight_u\bra{\int_u^1\weight_v^4\sum_{i=1}^n\E[\langle \res_i(\xi_i)\otimes\xi_i\otimes (\xi_i+\Res-\res(\xi_i)),\nabla^{3} h_v(W)\rangle]dv}du}ds,
\label{decomp-aim}
}
where $m_s=1/\sqrt{1-s}$ for $s\in[0,1)$. 
Also, under the assumptions of \cref{thm:gcomp},
\ben{\label{decomp-aim-gcomp}
\E[h_t(W)]-\int_{\mathbb R^d} h_t(z)p_W(z)dz
=\frac{1}{4}\int_t^1\bra{\int_s^1m_u^4\E[\langle (\Sigma_W-\Sigma)^{\otimes2}, \nabla^4 h_u(W)\rangle]du}ds.
}
\end{lemma}

\begin{proof}
See Appendix \ref{sec:decomp}.
\end{proof}

\begin{lemma}\label{coro:aht}
Let $h=1_A$ with $A\in\mcl R$. Then, for any $s\in(0,1)$ and $r\in\mathbb N$,
\ben{
\sup_{x\in\mathbb R^d}\|\nabla^rh_s(x)\|_1\leq C_r\bra{\frac{1-s}{\sigma_*^2s}\log d}^{r/2},
}
where $C_r>0$ is a constant depending only on $r$. 
\end{lemma}

\begin{proof}
Observe that $\{z\in\mathbb R^d:\sqrt{1-s}x+\sqrt sz\in A\}\in\mcl R$ for any $x\in\mathbb R^d$. 
Thus, the claim follows from \eqref{eq:ht-deriv} and \cref{lem:aht}. 
\end{proof}

\begin{proof}[Proof of \cref{thm:main}]
Without loss of generality, we may assume that $W$ and $Z$ are independent. 
We are going to apply \cref{l3} to $\mu,\nu$ and $K$ defined as
\[
\mu(A)=P(\sqrt{1-t}W\in A),\quad
\nu(A)=\int_{\mathbb R^d}1_A(\sqrt{1-t}z)p_W(z)dz,\quad
K(A)=P(\sqrt{t}Z\in A).
\]
They are chosen so that $\sqrt{1-t}W+\sqrt tZ\sim\mu*K$ and 
\[
\nu*K(A)=\int_{\mathbb R^d}\E[1_A(\sqrt{1-t}z+\sqrt t Z)]p_W(z)dz
\] 
for every $A\in\mathcal R$. 
Since 
\[
P(\|Z\|_\infty>\ol\sigma\sqrt{2\log(2d)})\leq\sum_{j=1}^dP(|Z_j|>\sigma_j\sqrt{2\log(2d)})\leq \frac{d}{2}e^{-\log(2d)}=\frac{1}{4},
\]
we have $\alpha:=K([-\eps,\eps]^d)\geq3/4>1/2$ with $\eps=\ol\sigma\sqrt{2t\log(2d)}$. 
Let $h=1_A$ with $A=\prod_{j=1}^d[a_j,b_j]\in\mcl R$. Then we have $M_h(x;\eps)=1_{A^\eps}(x)$ and $m_h(x;\eps)=1_{A^{-\eps}}(x)$, where we set $A^r:=\prod_{j=1}^d[a_j-r,b_j+r]$ for any $r\in\mathbb R$ with interpreting $[a,b]=\emptyset$ if $a>b$. Hence
\ben{\label{gamma-est}
\gamma(h;\eps)\leq\sup_{h=1_A,A\in\mcl R}\abs{\E[h_t(W)]-\int h_t(z)p_W(z)dz}
}
and
\[
\int [M_h(x; \epsilon)-h(x)]\nu(dx)\leq\sup_{A\in\mcl R;u,v\in [0,\eps]^d}\abs{\int 1_{A^{u,v}\setminus A}(\sqrt{1-t}z)p_W(z)dz}.
\]
Further, for each $j=1,\dots,d$, set 
\[
\begin{cases}
I_j=[a_j+\eps,b_j-\eps],~u_j=v_j=\eps & \text{if }a_j+\eps<b_j-\eps,\\
I_j=\{(a_j+b_j)/2\},~u_j=v_j=(b_j-a_j)/2 & \text{if }b_j\leq a_j+2\eps.
\end{cases}
\]
Then we have $\tilde A:=\prod_{j=1}^dI_j\in\mcl R$, $u:=(u_1,\dots,u_d)^\top\in [0,\eps]^d$, $v:=(v_1,\dots,v_d)^\top\in [0,\eps]^d$ and
\[
\int [h(x)-m_h(x; \epsilon)]\nu(dx)=\int 1_{\tilde A^{u,v}\setminus \tilde A}(\sqrt{1-t}z)p_W(z)dz.
\]
Consequently,
\ben{\label{tau-est}
\tau(h;\eps)\leq\sup_{A\in\mcl R;u,v\in [0,\eps]^d}\abs{\int 1_{A^{u,v}\setminus A}(\sqrt{1-t}z)p_W(z)dz}.
}
Besides, for any $x\in\mathbb R^d$ and $y\in [-\eps,\eps]^d$, we have $h(x+y)-h(x)=1_{(A+y)\setminus A}(x)-1_{A\setminus(A+y)}(x)$. For each $j=1,\dots,d$, set 
\[
\begin{cases}
I_j=[a_j+y_j,b_j],u_j=0,v_j=y_j & \text{if }y_j\geq0,a_j+y_j<b_j,\\
I_j=[a_j,b_j+y_j],u_j=-y_j,v_j=0 & \text{if }y_j<0,a_j<b_j+y_j,\\
I_j=\{a_j+y_j\},u_j=0,v_j=b_j-a_j & \text{otherwise}.
\end{cases}
\]
Then we have $\tilde A:=\prod_{j=1}^dI_j\in\mcl R$, $u:=(u_1,\dots,u_d)^\top\in [0,\eps]^d$, $v:=(v_1,\dots,v_d)^\top\in [0,\eps]^d$ and
\[
\int 1_{(A+y)\setminus A}(x)\nu(dx)=\int 1_{\tilde A^{u,v}\setminus \tilde A}(\sqrt{1-t}z)p_W(z)dz.
\]
Also, observe that $A\setminus(A+y)=[(A+y)-y]\setminus(A+y)$ and $A+y\in\mcl R$. Hence we conclude
\ben{\label{tautilde-est}
\tilde\tau(h;\eps)\leq2\sup_{A\in\mcl R;u,v\in [0,\eps]^d}\abs{\int 1_{A^{u,v}\setminus A}(\sqrt{1-t}z)p_W(z)dz}.
}
In addition, observe that $h_y=1_{A-y}$ for any $y\in\mathbb R^d$. 
Hence, Eqs.\eqref{gamma-est}--\eqref{tautilde-est} still hold if the left hand side is replaced by the one marked with an asterisk. 
As a result, \cref{l3} gives
\ban{
&\sup_{A\in\mcl R}\abs{P(\sqrt{1-t}W\in A)-\int_{\mathbb R^d}1_A(\sqrt{1-t}z)p_W(z)dz}\notag\\
&\leq2\sup_{h=1_A,A\in\mcl R}\abs{\E[h_t(W)]-\int h_t(z)p_W(z)dz}
+6\sup_{A\in\mcl R;u,v\in [0,\eps]^d}\abs{\int 1_{A^{u,v}\setminus A}(\sqrt{1-t}z)p_W(z)dz}.
\label{ae:smoothing}
} 
Note that $A/\sqrt{1-t}\in\mcl R$ for any $A\in\mcl R$. Thus, the left hand side of \eqref{ae:smoothing} is equal to
\[
\sup_{A\in\mcl R}\abs{P(W\in A)-\int_{A}p_W(z)dz}.
\]
Further, by Lemmas \ref{nazarov} and \ref{lem:anti}
\ban{
&\sup_{A\in\mcl R;u,v\in [0,\eps]^d}\abs{\int 1_{A^{u,v}\setminus A}(\sqrt{1-t}z)p_W(z)dz}\notag\\
&\lesssim \frac{\eps}{\sqrt{1-t}}\bra{\frac{\sqrt{\log d}}{\ul\sigma}
+\frac{\log^{3/2}d}{\sigma_*^3}\|\Sigma_W-\Sigma\|_\infty
+\frac{\log^{2}d}{\sigma_*^4}\norm{\sum_{i=1}^n\E[\xi_i^{\otimes3}]}_\infty}\notag\\
&\lesssim \ol\sigma\sqrt t\bra{\frac{\log d}{\ul\sigma}
+\frac{\log^{2}d}{\sigma_*^3}\|\Sigma_W-\Sigma\|_\infty
+\frac{\log^{5/2}d}{\sigma_*^4}\norm{\sum_{i=1}^n\E[\xi_i^{\otimes3}]}_\infty}.\label{ae:anti}
}
Meanwhile, set
$
\kappa:=(\log d)/\sigma_*^2
$. 
Then, by \cref{coro:aht}, we have for any $u,v\in(0,1)$ 
\ba{
|\E[\langle \bar T^{\otimes2},\nabla^4 h_u(W)\rangle]|
&\lesssim \E\|\bar T\|_\infty^2\bra{\frac{1-u}{u}\kappa}^2,\\
\abs{\sum_{i=1}^n\E[\langle \tau_i(\xi_i)^{\otimes2},\nabla^4 h_u(W)\rangle]}
&\lesssim \E\norm{\sum_{i=1}^n \tau_i(\xi_i)^{\otimes2}}_\infty\bra{\frac{1-u}{u}\kappa}^2,\\
\abs{\sum_{i=1}^n\E[\langle \tau_i(\xi_i)\otimes \xi_i^{\otimes2},\nabla^4 h_u(W)\rangle]}
&\lesssim \E\norm{\sum_{i=1}^n \tau_i(\xi_i)\otimes\xi_i^{\otimes2}}_\infty\bra{\frac{1-u}{u}\kappa}^2,\\
\abs{\sum_{i=1}^n\E[\langle \xi_i^{\otimes4},\nabla^4 h_v(W)\rangle]}
&\lesssim \E\norm{\sum_{i=1}^n \xi_i^{\otimes4}}_\infty\bra{\frac{1-v}{v}\kappa}^2,\\
\abs{\sum_{i=1}^n\E[\langle \xi_i^{\otimes3}\otimes\Res,\nabla^4 h_v(W)\rangle]}
&\lesssim \E\norm{\Res\otimes\sum_{i=1}^n \xi_i^{\otimes3}}_\infty\bra{\frac{1-v}{v}\kappa}^2,\\
\abs{\sum_{i=1}^n\E[\langle \xi_i^{\otimes3}\otimes\res_i(\xi_i),\nabla^4 h_v(W)\rangle]}
&\lesssim \E\norm{\sum_{i=1}^n \xi_i^{\otimes3}\otimes\res_i(\xi_i)}_\infty\bra{\frac{1-v}{v}\kappa}^2,\\
\abs{\sum_{i=1}^n\E[\langle \xi_i^{\otimes3}\otimes \bar T,\nabla^5 h_v(W)\rangle]}
&\lesssim \E\norm{\bar T\otimes\sum_{i=1}^n\xi_i^{\otimes3}}_\infty\bra{\frac{1-v}{v}\kappa}^{5/2},\\
\abs{\sum_{i=1}^n\E[\langle \xi_i^{\otimes3}\otimes \tau_i(\xi_i),\nabla^5 h_v(W)\rangle]}
&\lesssim \E\norm{\sum_{i=1}^n \xi_i^{\otimes3}\otimes\tau_i(\xi_i)}_\infty\bra{\frac{1-v}{v}\kappa}^{5/2},\\
|\E[\langle \Res^{\otimes2},\nabla^2 h_u(W)\rangle]|
&\lesssim \E\|\Res\|_\infty^2\bra{\frac{1-u}{u}\kappa},\\
\abs{\sum_{i=1}^n\E[\langle \res_i(\xi_i)^{\otimes2},\nabla^2 h_u(W)\rangle]}
&\lesssim \E\norm{\sum_{i=1}^n \res_i(\xi_i)^{\otimes2}}_\infty\bra{\frac{1-u}{u}\kappa},\\
\abs{\sum_{i=1}^n\E[\langle \res_i(\xi_i)\otimes \bar T,\nabla^3 h_u(W)\rangle]}
&\lesssim \E\norm{\bar T\otimes\Res}_\infty\bra{\frac{1-u}{u}\kappa}^{3/2},\\
\abs{\sum_{i=1}^n\E[\langle \res_i(\xi_i)\otimes \tau_i(\xi_i),\nabla^3 h_u(W)\rangle]}
&\lesssim \E\norm{\sum_{i=1}^n \res_i(\xi_i)\otimes\tau_i(\xi_i)}_\infty\bra{\frac{1-u}{u}\kappa}^{3/2},\\
\abs{\sum_{i=1}^n\E[\langle \res_i(\xi_i)\otimes \xi_i^{\otimes2},\nabla^3 h_u(W)\rangle]}
&\lesssim \E\norm{\sum_{i=1}^n \res_i(\xi_i)\otimes\xi_i^{\otimes2}}_\infty\bra{\frac{1-u}{u}\kappa}^{3/2},\\
\abs{\sum_{i=1}^n\E[\langle \res_i(\xi_i)\otimes\xi_i\otimes\bar T,\nabla^4 h_v(W)\rangle]}
&\lesssim \E\norm{\bar T\otimes\sum_{i=1}^n \res_i(\xi_i)\otimes\xi_i}_\infty\bra{\frac{1-u}{u}\kappa}^2,\\
\abs{\sum_{i=1}^n\E[\langle \res_i(\xi_i)\otimes\xi_i\otimes\tau_i(\xi_i),\nabla^4 h_v(W)\rangle]}
&\lesssim \E\norm{\sum_{i=1}^n \res_i(\xi_i)\otimes\xi_i\otimes\tau_i(\xi_i)}_\infty\bra{\frac{1-v}{v}\kappa}^2,\\
\abs{\sum_{i=1}^n\E[\langle \res_i(\xi_i)\otimes\xi_i\otimes\Res,\nabla^3 h_v(W)\rangle]}
&\lesssim \E\norm{\Res\otimes\sum_{i=1}^n \res_i(\xi_i)\otimes\xi_i}_\infty\bra{\frac{1-v}{v}\kappa}^{3/2},\\
\abs{\sum_{i=1}^n\E[\langle \res_i(\xi_i)\otimes\xi_i\otimes\res_i(\xi_i),\nabla^3 h_v(W)\rangle]}
&\lesssim \E\norm{\sum_{i=1}^n \res_i(\xi_i)^{\otimes2}\otimes\xi_i}_\infty\bra{\frac{1-v}{v}\kappa}^{3/2}.
}
Also, by the AM--GM inequality,
\ba{
2\E\norm{\sum_{i=1}^n \tau_i(\xi_i)\otimes\xi_i^{\otimes2}}_\infty\kappa^2
&\leq\E\norm{\sum_{i=1}^n \tau_i(\xi_i)^{\otimes2}}_\infty\kappa^2
+\E\norm{\sum_{i=1}^n \xi_i^{\otimes4}}_\infty\kappa^2,\\
2\E\norm{\bar T\otimes\Res}_\infty\kappa^{3/2}
&\leq\E\norm{\bar T}_\infty^2\kappa^2
+\E\norm{\Res}_\infty^2\kappa,\\
2\E\norm{\sum_{i=1}^n \res_i(\xi_i)\otimes\tau_i(\xi_i)}_\infty\kappa^{3/2}
&\leq\E\norm{\sum_{i=1}^n \res_i(\xi_i)^{\otimes2}}_\infty\kappa
+\E\norm{\sum_{i=1}^n \tau_i(\xi_i)^{\otimes2}}_\infty\kappa^2,\\
2\E\norm{\sum_{i=1}^n \res_i(\xi_i)\otimes\xi_i^{\otimes2}}_\infty\kappa^{3/2}
&\leq\E\norm{\sum_{i=1}^n \res_i(\xi_i)^{\otimes2}}_\infty\kappa
+\E\norm{\sum_{i=1}^n \xi_i^{\otimes4}}_\infty\kappa^2.
}
Further, observe that
\ba{
\frac{1}{w^a\sqrt{1-w}}\leq\frac{\sqrt 2}{w^a}+\frac{2^a}{\sqrt{1-w}}\quad\text{for any }w\in(0,1),a>0
}
and
\ba{
\int_t^1\frac{1}{s}ds=|\log t|,\quad
\int_t^1\frac{1}{\sqrt{s}}ds\leq2,\quad
\int_t^1\frac{1}{\sqrt{1-s}}ds\leq2,\quad
\int_t^1\frac{-\log s}{\sqrt{1-s}}ds\lesssim1.
}
Combining these estimates with \eqref{ae:smoothing}--\eqref{ae:anti} and \eqref{decomp-aim} gives the desired result. 
\end{proof}

\begin{proof}[Proof of \cref{thm:gcomp}]
The claim follows by replacing \eqref{decomp-aim} with \eqref{decomp-aim-gcomp} in the proof of \cref{thm:main}.  
\end{proof}

\subsection{Proof of Theorem \ref{coro1}}

Below we will frequently use the following identity without reference: 
For any $x_1,\dots,x_n\in\mathbb R^d$ and $r\in\mathbb N$, 
\[
\norm{\sum_{i=1}^nx_i^{\otimes (2r)}}_\infty=\max_{1\leq j\leq d}\sum_{i=1}^nx_{ij}^{2r}.
\]
This follows from the AM-GM inequality. 
\begin{proof}[\bf\upshape Proof of \cref{coro1}]
Observe that for every $i=1,\dots,n$, $(\tau_i^X+(\tau_i^X)^\top)/2$ is a Stein kernel for $X_i$ and $\max_{j,k}\|(\tau_{i,jk}^X(X_i)+\tau_{i,kj}^X(X_i))/2\|_{\psi_{1/2}}\lesssim b^2$. 
Therefore, we may assume $\tau^X_{i,jk}=\tau^X_{i,kj}$ for all $j,k\in\{1,\dots,d\}$ without loss of generality. This particularly implies $\E[\tau_i^X(X_i)]=\E[X_i^{\otimes2}]$.

We apply \cref{thm:main} to $\xi_i=X_i/\sqrt n$. Observe that $\Sigma_W=\Sigma$ and that $\xi_i$ has a Stein kernel $\tau_i$ satisfying $\tau_i(\xi_i)=\tau^X_i(X_i)/n$. 

Let us bound the quantities appearing on the right hand side of \eqref{eq:main}. 
First, noting that $\E[\tau_i(\xi_i)]=\E[\xi_i^{\otimes2}]$, we have
\[
\bar T=\sum_{i=1}^n(\tau_i(\xi_i)-\E[\tau_i(\xi_i)]).
\]
Thus, by \cref{max-weibull} with $K=b^2/n$, $\alpha=1/2$ and $r=2$, we obtain
\ben{\label{tbar-bound}
\E[\|\bar T\|_\infty^2]
\lesssim \frac{b^4}{n^2}\bra{\sqrt{n\log d}+(\log d)^{2}}^2
\lesssim \frac{b^4\log d}{n},
}
where the second inequality follows from $\log^3 d\leq n$. 
Next, by \cref{prod-psi} and \cref{max-weibull} with $K=b^2/n$, $\alpha=1/4$ and $r=1$,
\ba{
\E\norm{\sum_{i=1}^n\{\tau_i(\xi_i)^{\otimes2}-\E[\tau_i(\xi_i)^{\otimes2}]\}}_\infty
\lesssim \frac{b^4}{n^2}\bra{\sqrt{n\log d}+(\log d)^{4}}.
}
Therefore, 
\besn{\label{tau2-bound}
\E\norm{\sum_{i=1}^n\tau_i(\xi_i)^{\otimes2}}_\infty
&\leq\norm{\sum_{i=1}^n\E[\tau_i(\xi_i)^{\otimes2}]}_\infty
+\E\norm{\sum_{i=1}^n\{\tau_i(\xi_i)^{\otimes2}-\E[\tau_i(\xi_i)^{\otimes2}]\}}_\infty\\
&\lesssim \frac{b^4}{n}+\frac{b^4\sqrt{\log d}}{n^{3/2}}+\frac{b^4(\log d)^{4}}{n^2}
\lesssim \frac{b^4}{n}+\frac{b^4(\log d)^{4}}{n^2},
}
where the last inequality follows from $\sqrt{\log d}\leq n^{1/6}\leq\sqrt n$. 
Similarly, we can show that
\ban{
\E\norm{\sum_{i=1}^n\xi_i^{\otimes4}}_\infty
&\lesssim \frac{b^4}{n}+\frac{b^4(\log d)^{4}}{n^2},
\label{xi4-bound}\\
\E\norm{\sum_{i=1}^n\xi_i^{\otimes3}\otimes\tau_i(\xi_i)}_\infty
&\lesssim \frac{b^5}{n^{3/2}}+\frac{b^5(\log d)^{5}}{n^{5/2}}.
\label{xi3-tau-bound}
}
In addition, by the Schwarz inequality,
\ba{
\E\norm{\bar T\otimes\sum_{i=1}^n\xi_i^{\otimes3}}_\infty
\leq\sqrt{\E\norm{\bar T}^2_\infty}\sqrt{\E\norm{\sum_{i=1}^n\xi_i^{\otimes3}}^2_\infty}.
}
Similarly to the proof of \eqref{tau2-bound}, we can prove
\ben{\label{xi3-bound}
\E\norm{\sum_{i=1}^n\xi_i^{\otimes3}}^2_\infty
\lesssim \bra{\frac{b^3}{\sqrt n}+\frac{b^3(\log d)^{3}}{n^{3/2}}}^2
\lesssim \frac{b^6}{n},
}
where the second inequality follows by the assumption $\log^3 d\leq n$. 
Combining this with \eqref{tbar-bound} gives
\ben{\label{tbar-xi3-bound}
\E\norm{\bar T\otimes\sum_{i=1}^n\xi_i^{\otimes3}}_\infty
\lesssim \frac{b^5\sqrt{\log d}}{n}.
}

Now, by \eqref{tbar-bound}--\eqref{xi4-bound},
\ba{
&\frac{\log^2d}{\sigma_*^4}\bra{
\E\|\bar T\|_\infty^2
+\E\norm{\sum_{i=1}^n\tau_i(\xi_i)^{\otimes2}}_\infty
+\E\norm{\sum_{i=1}^n\xi_i^{\otimes4}}_\infty
}\\
&\lesssim \frac{\log^2d}{\sigma_*^4}\bra{\frac{b^4\log d}{n}+\frac{b^4(\log d)^{4}}{n^2}}
\lesssim \frac{b^4\log^3 d}{\sigma_*^4n},
}
where the second inequality follows by the assumption $\log^3 d\leq n$. 
Also, by \eqref{xi3-tau-bound} and \eqref{tbar-xi3-bound},
\ba{
&\frac{\log^{5/2}d}{\sigma_*^{5}}\bra{
\E\norm{\bar T\otimes\sum_{i=1}^n\xi_i^{\otimes3}}_\infty
+\E\norm{\sum_{i=1}^n\xi_i^{\otimes3}\otimes\tau_i(\xi_i)}_\infty
}\\
&\lesssim \frac{\log^{5/2}d}{\sigma_*^{5}}\bra{\frac{b^5\sqrt{\log d}}{n}+\frac{b^5(\log d)^{5}}{n^{5/2}}}
\lesssim \frac{b^5\log^3 d}{\sigma_*^{5}n},
}
where we used the assumption $\log^3 d\leq n$ in the second inequality. 
All together, we obtain by \cref{thm:main}
\ba{
&\sup_{A\in\mcl R}\abs{P(S_n\in A)-\int_{A}p_n(z)dz}\\
&\lesssim |\log t|\frac{b^4\log^3 d}{\sigma_*^4n}+\frac{b^5\log^3 d}{\sigma_*^{5}n}
+\ol\sigma\sqrt t\bra{\frac{\log d}{\ul\sigma}
+\frac{\log^{5/2}d}{\sigma_*^4}\norm{\sum_{i=1}^n\E[\xi_i^{\otimes3}]}_\infty
}
}
for any $t\in(0,1/2]$. 
With $t=1/n^2$, we obtain
\ba{
\ol\sigma\sqrt t\bra{\frac{\log d}{\ul\sigma}
+\frac{\log^{5/2}d}{\sigma_*^4}\norm{\sum_{i=1}^n\E[\xi_i^{\otimes3}]}_\infty
}
&\lesssim \frac{b\log d}{\sigma_*n}+\frac{b^4\log^{5/2}d}{\sigma_*^4n^{3/2}}
\lesssim \frac{b^5\log^3 d}{\sigma_*^{5}n}.
}
Consequently, we obtain the desired result. 
\end{proof}

\subsection{Proof of Theorem \ref{prop:boot}}

First consider Case (i). 
Set $\wt X_i:=X_i-\bar X$ for $i=1,\dots,n$. 
We apply \cref{thm:main} with $\xi_i=w_i\wt X_i/\sqrt n$ and $t=1/n^3$ conditional on the data. Note that conditional on the data, $\xi_i$ has a Stein kernel $\tau_i$ satisfying $\tau_i(\xi_i)=\E^*[\tau_i^*(w_i)\wt X_i^{\otimes2}/n\mid\xi_i]$ by \cref{sk-affine}. Therefore, we have
\ban{
&\sup_{A\in\mcl R}\abs{P^*(S_n^*\in A)-\int_{A}\tilde p_{n,\gamma}(z)dz}\notag\\
&\lesssim (\log n)\frac{\log^2d}{\sigma_*^4}\left(
\E^*\sbra{\norm{\bar T^*}_\infty^2}
+\E^*\sbra{\norm{\frac{1}{n^2}\sum_{i=1}^n\tau_i^*(w_i)^2\wt X_i^{\otimes4}}_\infty}
+\E^*\sbra{\norm{\frac{1}{n^2}\sum_{i=1}^nw_i^4\wt X_i^{\otimes4}}_\infty}
\right)\notag\\
&\quad+\frac{\log^{5/2}d}{\sigma_*^{5}}\bra{
\E^*\sbra{\norm{\bar T^*\otimes\frac{1}{n^{3/2}}\sum_{i=1}^nw_i^3\wt X_i^{\otimes3}}_\infty}
+\E^*\sbra{\norm{\frac{1}{n^{5/2}}\sum_{i=1}^nw_i^3\tau^*_i(w_i)\wt X_i^{\otimes5}}_\infty}
}\notag\\
&\quad+\frac{\ol\sigma}{n^{3/2}}\bra{\frac{\log d}{\ul\sigma}
+\frac{\log^2d}{\sigma_*^3}\|\wh\Sigma_n-\Sigma\|_\infty
+\frac{\log^{5/2}d}{\sigma_*^4}\norm{\frac{1}{n^{3/2}}\sum_{i=1}^n\E[w_i^3]\wt X_i^{\otimes3}}_\infty\label{boot1}
},
}
where
\[
\tilde p_{n,\gamma}(z)=\phi_\Sigma(z)+\frac{1}{2}\langle\wh{\Sigma}_n-\Sigma,\nabla^2\phi_\Sigma(z)\rangle-\frac{\gamma}{6n^{3/2}}\sum_{i=1}^n\langle (X_i-\bar X)^{\otimes3},\nabla^3\phi_\Sigma(z)\rangle
\]
and
\[
\bar T^*:=\frac{1}{n}\sum_{i=1}^n\tau_i^*(w_i)\wt X_i^{\otimes2}-\Sigma.
\]
Next, by Lemmas \ref{max-weibull} and \ref{hj-bound}, there exists a universal constant $c$ such that the event 
\besn{\label{def:En}
\mcl E_n&:=\bigcap_{r=1}^2\cbra{\norm{\frac{1}{n}\sum_{i=1}^n(X_{i}^{\otimes r}-\E[X_{i}^{\otimes r}])}_\infty\leq cb^{r}\bra{\sqrt{\frac{\log(dn)}{n}}+\frac{\log^r(dn)}{n}}}\\
&\qquad\cap\bigcap_{r=3}^5\cbra{\max_{1\leq j\leq d}\abs{\frac{1}{n}\sum_{i=1}^n(|X_{ij}|^{r}-\E[|X_{ij}|^{r}])}\leq cb^{r}\bra{\sqrt{\frac{\log(dn)}{n}}+\frac{\log^r(dn)}{n}}}\\
&\qquad\cap\cbra{\max_{1\leq j,k\leq d}\frac{1}{n}\sum_{i=1}^n|X_{ij}||X_{ik}|1_{\{|X_{ij}|\vee|X_{ik}|>2b\log n\}}\leq cb^2\bra{\sqrt{\frac{\log(dn)}{n}}+\frac{\log^2(dn)}{n}}}\\
&\qquad\cap\cbra{\max_{1\leq j,k\leq d}\frac{1}{n}\sum_{i=1}^nX_{ij}^2X_{ik}^21_{\{|X_{ij}|\vee|X_{ik}|\leq 2b\log n\}}\leq cb^4\bra{1+\frac{\log(dn)\log^4n}{n}}}
}
occurs with probability at least $1-1/n$. 
Recall that $\log^3 d\leq n$ by assumption. Hence, on $\mcl E_n$, we have
\ben{\label{mean-bound}
\|\bar X\|_\infty\lesssim b\sqrt{\frac{\log(dn)}{n}}
}
and
\ben{\label{2-mom-bound}
\norm{\frac{1}{n}\sum_{i=1}^nX_{i}^{\otimes 2}-\Sigma}_\infty\lesssim b^2\sqrt{\frac{\log(dn)}{n}}.
}
Thus, on $\mcl E_n$,
\ben{\label{scov-bound}
\norm{\wh\Sigma_n-\Sigma}_\infty
\leq\norm{\frac{1}{n}\sum_{i=1}^nX_{i}^{\otimes 2}-\Sigma}_\infty+\|\bar X\|_\infty^2
\lesssim b^{2}\sqrt{\frac{\log (dn)}{n}}.
} 
Meanwhile, for every $r\in\{2,3,4,5\}$, we have on $\mcl E_n$
\ben{\label{boot-mom0}
\norm{\frac{1}{n}\sum_{i=1}^nX_i^{\otimes r}}
\leq\max_{1\leq j\leq d}\frac{1}{n}\sum_{i=1}^n|X_{ij}|^r
\lesssim b^r\bra{1+\frac{\log^r (dn)}{n}}.
}
Combining this bound with \eqref{mean-bound}, we have on $\mcl E_n$
\ben{\label{boot-mom}
\max_{1\leq j\leq d}\frac{1}{n}\sum_{i=1}^n|\wt X_{ij}|^r\lesssim b^r\bra{1+\frac{\log^r (dn)}{n}}.
}
%
Further, by \eqref{mean-bound}, \eqref{boot-mom0}, the construction of $\mcl E_n$ and $\log^3d\leq n$, we have on $\mcl E_n$
\ben{\label{hj-boot-1}
\max_{1\leq j,k\leq d}\frac{1}{n}\sum_{i=1}^n|\wt X_{ij}\wt X_{ik}|1_{\{|X_{ij}|\vee|X_{ik}|>2b\log n\}}
\lesssim b^2\bra{\sqrt{\frac{\log d}{n}}+\frac{\log^2(dn)}{n}},
}
and
\ben{\label{hj-boot-2}
\max_{1\leq j,k\leq d}\frac{1}{n}\sum_{i=1}^n\wt X_{ij}^2\wt X_{ik}^21_{\{|X_{ij}|\vee|X_{ik}|\leq 2b\log n\}}
\lesssim b^4.
}

Now we bound the right hand side of \eqref{boot1} on the event $\mcl E_n$. 
First, we have
\ben{\label{tbar-decomp}
\E^*[\|\bar T^*\|_\infty^2]\leq 2\E^*\sbra{\max_{1\leq j,k\leq d}R_{jk}^2}+2\|\wh\Sigma_n-\Sigma\|_\infty^2,
}
where
$
R_{jk}:=n^{-1}\sum_{i=1}^n\{\tau_i^*(w_i)-1\}\wt X_{ij}\wt X_{ik}.
$
We decompose $R_{jk}$ as 
\ba{
R_{jk}&=\frac{1}{n}\sum_{i=1}^n\{\tau_i^*(w_i)-1\}\wt X_{ij}\wt X_{ik}\bra{1_{\{|X_{ij}|\vee|X_{ik}|>2b\log n\}}+1_{\{|X_{ij}|\vee|X_{ik}|\leq 2b\log n\}}}\\
&=:R_{1,jk}+R_{2,jk}.
}
Since
\ba{
\max_{1\leq j,k\leq d}|R_{1,jk}|
&\leq(b_w^2+1)\max_{1\leq j,k\leq d}\frac{1}{n}\sum_{i=1}^n|\wt X_{ij}\wt X_{ik}|1_{\{|X_{ij}|\vee|X_{ik}|>2b\log n\}},
}
we have on $\mcl E_n$
\ben{\label{r1-est}
\max_{1\leq j,k\leq d}|R_{1,jk}|
\lesssim b_w^2b^2\bra{\sqrt{\frac{\log d}{n}}+\frac{\log^2(dn)}{n}}
}
by \eqref{hj-boot-1}. 
Meanwhile, by Nemirovski's inequality (e.g.~Lemma 14.24 in \cite{BvdG11}), 
\ba{
\E^*\sbra{\max_{1\leq j,k\leq d}R_{2,jk}^2}
&\lesssim \frac{b_w^4\log d}{n}\max_{1\leq j,k\leq d}\frac{1}{n}\sum_{i=1}^n\wt X_{ij}^2\wt X_{ik}^21_{\{|X_{ij}|\vee|X_{ik}|\leq 2b\log n\}}.
}
Hence we have on $\mcl E_n$
\ben{\label{r2-est}
\E^*\sbra{\max_{1\leq j,k\leq d}R_{2,jk}^2}
\lesssim \frac{b_w^4b^4\log d}{n}
}
by \eqref{hj-boot-2}. 
Combining \eqref{tbar-decomp}--\eqref{r2-est} with \eqref{scov-bound}, we obtain
\ben{\label{tbar-bound*}
\E^*\sbra{\|\bar T^*\|_\infty^2}
\lesssim b_w^4b^4\bra{\frac{\log(dn)}{n}+\frac{\log^4(dn)}{n^2}}.
}
Consequently, we have
\ben{\label{boot-aim1}
\frac{\log^2d}{\sigma_*^4}\E^*\sbra{\|\bar T^*\|_\infty^2}
\lesssim \frac{b_w^4b^4}{\sigma_*^4}\bra{\frac{\log ^3(dn)}{n}+\frac{\log^6(dn)}{n^2}}
\lesssim \frac{b_w^4b^4}{\sigma_*^4}\frac{\log ^3(dn)}{n},
}
where we used the assumption $\log^3d\leq n$ for the last inequality. 
Next, we have
\besn{\label{boot-bound2}
\norm{\frac{1}{n^2}\sum_{i=1}^n\tau_i^*(w_i)^2\wt X_i^{\otimes4}}_\infty
&\leq b_w^4\max_{j,k,l,m}\frac{1}{n^2}\sum_{i=1}^n|\wt X_{ij}\wt X_{ik}\wt X_{il}\wt X_{im}|
\lesssim \frac{b_w^4b^4}{n}\bra{1+\frac{\log^4 (dn)}{n}},
}
where the second inequality follows by the AM-GM inequality and \eqref{boot-mom}. 
Similarly, we can prove
\ban{
\norm{\frac{1}{n^2}\sum_{i=1}^nw_i^4\wt X_i^{\otimes4}}_\infty&\lesssim \frac{b_w^4b^4}{n}\bra{1+\frac{\log^4 (dn)}{n}},
\label{boot-bound4}\\
\norm{\frac{1}{n^{3/2}}\sum_{i=1}^nw_i^3\wt X_i^{\otimes3}}_\infty&\lesssim \frac{b_w^3b^3}{\sqrt n}\bra{1+\frac{\log^3 (dn)}{n}}
\lesssim \frac{b_w^3b^3}{\sqrt n},
\label{boot-bound5}\\
\norm{\frac{1}{n^{5/2}}\sum_{i=1}^nw_i^3\tau^*_i(w_i)\wt X_i^{\otimes5}}_\infty&\lesssim \frac{b_w^5b^5}{n^{3/2}}\bra{1+\frac{\log^5 (dn)}{n}}.
\label{boot-bound6}
}
By \eqref{boot-bound2}--\eqref{boot-bound4}, 
\ban{
&\frac{\log^2d}{\sigma_*^4}\bra{
\E^*\sbra{\norm{\frac{1}{n^2}\sum_{i=1}^n\tau_i^*(w_i)^2\wt X_i^{\otimes4}}_\infty}
+\E^*\sbra{\norm{\frac{1}{n^2}\sum_{i=1}^nw_i^4\wt X_i^{\otimes4}}_\infty}}
\notag\\
&\lesssim \frac{b_w^4b^4}{\sigma_*^4}\bra{\frac{\log^2d}{n}+\frac{\log^6 (dn)}{n^2}}
\lesssim \frac{b_w^4b^4}{\sigma_*^4}\frac{\log^3 (dn)}{n}.
\label{boot-aim2}
}
Also, by \eqref{boot-bound6},
\ban{
\frac{\log^{5/2}d}{\sigma_*^{5}}\E^*\sbra{\norm{\frac{1}{n^{5/2}}\sum_{i=1}^nw_i^3\tau^*_i(w_i)\wt X_i^{\otimes5}}_\infty}
&\lesssim \frac{b_w^5b^5}{\sigma_*^5}\bra{\frac{\log^{5/2}d}{n^{3/2}}+\frac{\log^{15/2} (dn)}{n^{5/2}}}\notag\\
&\lesssim \frac{b_w^5b^5}{\sigma_*^5}\frac{\log^3(dn)}{n}.\label{boot-aim4}
}
Further, by the Schwarz inequality, \eqref{tbar-bound*} and \eqref{boot-bound5},
\besn{\label{boot-aim5}
\frac{\log^{5/2}d}{\sigma_*^{5}}\E^*\sbra{\norm{\bar T^*\otimes\frac{1}{n^{3/2}}\sum_{i=1}^nw_i^3\wt X_i^{\otimes3}}_\infty}
&\lesssim \frac{\log^{5/2}d}{\sigma_*^{5}}\sqrt{b_w^4b^4\bra{\frac{\log (dn)}{n}+\frac{\log^4 (dn)}{n^2}}}\cdot\frac{b_w^3b^3}{\sqrt n}\\
&\leq\frac{b_w^5b^5}{\sigma_*^{5}}\bra{\frac{\log^3 (dn)}{n}+\frac{\log^{9/2}(dn)}{n^{3/2}}}
\lesssim \frac{b_w^5b^5}{\sigma_*^{5}}\frac{\log^3 (dn)}{n}.
}
Combining \eqref{boot1}, \eqref{scov-bound}, \eqref{boot-aim1}, \eqref{boot-bound5}, \eqref{boot-aim2}--\eqref{boot-aim5} and $b_w\geq1,b/\sigma_*\geq1$, we have, on $\mcl E_n$,
\bes{
&\sup_{A\in\mcl R}\abs{P^*(S_n^*\in A)-\int_{A}\tilde p_{n,\gamma}(z)dz}\\
&\lesssim \frac{b_w^5b^5}{\sigma_*^5}\frac{\log ^3(dn)}{n}\log n
+\frac{\ol\sigma}{n^{3/2}}\bra{\frac{\log d}{\ul\sigma}
+b^{2}\frac{\log^{5/2} d}{\sigma_*^3\sqrt n}
+b^3b_w^3\frac{\log^{5/2}d}{\sigma_*^4\sqrt n}
}
\lesssim \frac{b_w^5b^5}{\sigma_*^5}\frac{\log ^3(dn)}{n}\log n.
}
It remains to prove
\ben{\label{boot-aim-final}
\sup_{A\in\mcl R}\abs{\int_{A}\tilde p_{n,\gamma}(z)dz-\int_{A} \hat{p}_{n,\gamma}(z)dz}
\lesssim \frac{b_w^5b^5}{\sigma_*^5}\frac{\log ^3(dn)}{n}\log n
\quad\text{on }\mcl E_n.
}
Observe that
\ben{\label{ae-boot-2nd}
\sup_{A\in\mcl R}|\langle \wh\Sigma_n-\ol{X^2},\int_A\nabla^2\phi_\Sigma(z)dz\rangle|\lesssim\|\bar X\|_\infty^2\frac{\log d}{\sigma_*^2}
}
and
\bes{
&\sup_{A\in\mcl R}\abs{\frac{1}{n^{3/2}}\sum_{i=1}^n\langle (X_i-\bar X)^{\otimes 3}-X_i^{\otimes 3},\int_A\nabla^3\phi_\Sigma(z)dz\rangle}\\
&\lesssim\frac{1}{\sqrt n}\bra{\norm{\frac{1}{n}\sum_{i=1}^nX_i^{\otimes2}}_\infty\|\bar X\|_\infty+\|\bar X\|_\infty^3}\frac{\log^{3/2} d}{\sigma_*^3}.
}
Also, note that $|\gamma|\leq b_w\E[w_1^2]=b_w$. Hence \eqref{boot-aim-final} follows from \eqref{mean-bound} and \eqref{boot-mom0}. 

Next consider Case (ii). In this case, $S_n^*\sim N(0,\wh\Sigma_n)$ conditional on the data. 
Hence, applying \cref{thm:gcomp} and using the bounds \eqref{mean-bound}, \eqref{scov-bound} and \eqref{ae-boot-2nd}, we obtain the desired bound with a simplified argument of the proof for Case (i).  
\qed

\subsection{Proof of Proposition \ref{prop:copula}}\label{pr-copula}

The proof uses the Malliavin--Stein method. We refer to \cite[Chapter 2]{NoPe12} for undefined notation and concepts used below. 

Denote by $r_j$ the $j$-th row vector of $R^{1/2}$. Then $Z$ has the same law as $(r_1\cdot G,\dots,r_d\cdot G)^\top$ with $G\sim N(0,I_d)$. 
Hence we may assume $X$ is of the form $X=\psi(G)$ with $\psi:\mathbb R^d\to\mathbb R^d$ defined as $\psi_j(x)=F_j^{-1}(\Phi(r_j\cdot x))$ for $j=1,\dots,d$ and $x\in\mathbb R^d$. 
Note that $F_j^{-1}$ is absolutely continuous and satisfies $|(F_j^{-1})'(t)|\leq(\kappa\min\{t,1-t\})^{-1}$ a.e.~by \eqref{eq:cheeger}; this follows from arguments in Section 5.3 of \cite{BoLe19} (see also Propositions A.17 and A.19 in \cite{BoLe19}). 
In particular, $F_j^{-1}$ is locally Lipschitz, and hence $\psi_j$ is locally Lipschitz and its gradient is given by $\nabla\psi_j(x)=\phi(r_j\cdot x)(F_j^{-1})'(\Phi(r_j\cdot x))r_j$ a.e. 
Since $|r_j|^2=R_{jj}=1$, we obtain
\ben{\label{est:g-copula}
|\nabla\psi_j(G)|=\phi(Z_j)(F_j^{-1})'(\Phi(Z_j))\leq\frac{\phi(Z_j)}{\kappa\min\{\Phi(Z_j),1-\Phi(Z_j)\}}
\leq\frac{1+|Z_j|}{\kappa},
}
where the last inequality follows by Birnbaum's inequality. 
Now, let $H=\mathbb R^{d}$ be the Hilbert space equipped with the canonical inner product. 
Consider an isonormal Gaussian process over $H$ given by $W(h)=h\cdot G$, $h\in H$. We consider Malliavin calculus with respect to $W$. 
The above argument implies that $X_j=\psi_j(G)\in\mathbb D^{1,2}$ and $DX_j=\nabla \psi_j(G)$ for every $j=1,\dots,d$. 
Therefore, by Proposition 3.7 in \cite{NPS14}, the map $\tau:\mathbb R^d\to(\mathbb R^d)^{\otimes2}$ defined by $\tau_{jk}(x)=\E[-DL^{-1}X_j\cdot DX_k\mid X=x]$ for $x\in\mathbb R^d$ and $j,k=1,\dots,d$ gives a Stein kernel for $X$. 
For any $p\geq1$ and $j,k=1,\dots,d$, we have
\ba{
\E[|\tau_{jk}(X)|^p]\leq\sqrt{\E[|DL^{-1}X_j|^{2p}]\E[|DX_k|^{2p}]}\leq\sqrt{\E[|DX_j|^{2p}]\E[|DX_k|^{2p}]},
}
where the first inequality is by the Jensen and Schwarz inequalities and the second by Lemma 5.3.7 in \cite{NoPe12}. 
If $p$ is an even integer, we also have $\E[X_j^p]\leq(p-1)^{p/2}\E[|DX_j|^{p}]$ by Lemma 5.3.7 in \cite{NoPe12}. 
Since $\|\nabla\psi_j(G)\|_p\lesssim\sqrt p\kappa^{-1}$ by \eqref{est:g-copula}, we obtain the desired result. 
\qed

\subsection{Cheeger constant of the gamma distribution}

\begin{proposition}\label{cheeger-gamma}
Any gamma distribution has a positive Cheeger constant. 
\end{proposition}

\begin{proof}
Let $\mu$ be the gamma distribution with shape $\nu$ and rate $\alpha$. If $\nu\geq1$, then $\mu$ is log-concave, so the claim follows by Proposition 4.1 in \cite{Bo99}. When $\nu<1$, the density $f$ of $\mu$ satisfies $\inf_{0<t<M}f(t)>0$ for every $M>0$. Hence, in view of Theorem 1.3 in \cite{BoHo97}, it suffices to prove $\liminf_{p\uparrow1}f(F^{-1}(p))/(1-p)>0$, where $F$ is the distribution function of $\mu$. 
Since
\[
\frac{d}{dp}f(F^{-1}(p))=\frac{\nu-1}{F^{-1}(p)}-\alpha\to-\alpha\quad(p\uparrow1),
\]
we have by L'H\^{o}pital's rule $\liminf_{p\uparrow1}f(F^{-1}(p))/(1-p)=\alpha>0$. This completes the proof.  
\end{proof}

\subsection{Proof of Corollary \ref{coro:lb-ga}}

Replacing $X_i$ by $X_i/\sigma$, we may assume $\sigma=1$ without loss of generality. 

Observe that $P(T_n\geq c^G_{1-\alpha})-\alpha=P(Z< c^G_{1-\alpha})-P(T_n< c^G_{1-\alpha})$. 
Thus, in view of \eqref{eq:ae}, it suffices to prove
\ben{\label{aim:lb-ga}
\frac{1}{(\log d)^{3/2}}\int_{(-\infty,c^G_{1-\alpha}]^d}\langle\E[\ol{X^{3}}],\nabla^3\phi_\Sigma(z)\rangle dz
\to-2\sqrt 2\gamma_X\log(1-\alpha).
}
We decompose the left hand side as
\ba{
&\int_{(-\infty,c^G_{1-\alpha}]^d}\langle\E[\ol{X^{3}}],\nabla^3\phi_\Sigma(z)\rangle dz\\
&=\frac{1}{n}\sum_{i=1}^n\sum_{j=1}^d\E[X_{ij}^3]\phi''(c^G_{1-\alpha})\Phi(c^G_{1-\alpha})^{d-1}
+\frac{3}{n}\sum_{i=1}^n\sum_{j,k:j\neq k}\E[X_{ij}^2X_{ik}]\phi'(c^G_{1-\alpha})\phi(c^G_{1-\alpha})\Phi(c^G_{1-\alpha})^{d-2}\\
&\quad+\frac{1}{n}\sum_{i=1}^n\sum_{\begin{subarray}{c}
j,k,l=1\\
j,k,l\text{ are distinct}
\end{subarray}}^d\E[X_{ij}X_{ik}X_{il}]\phi(c^G_{1-\alpha})^3\Phi(c^G_{1-\alpha})^{d-3}\\
&=:\mathbb{I}+\mathbb{II}+\mathbb{III}.
}
Since $\Sigma=I_d$ by assumption, we have $c^G_{1-\alpha}=\Phi^{-1}((1-\alpha)^{1/d})$. 
Hence we have $c^G_{1-\alpha}=\sqrt{2\log d}+o(1)$ (cf.~the proof of \cite[Proposition 2.1]{Ko21}). We also have $d\phi(c^G_{1-\alpha})/\sqrt{2\log d}\to-\log(1-\alpha)$ by \cite[Lemma 10.3]{BLM13}. 
Consequently,
\[
\frac{\mathbb I}{(2\log d)^{3/2}}\to-\gamma_X\log(1-\alpha).
\]
Meanwhile, noting $n^{-1}\sum_{i=1}^n\E[X_i^{\otimes2}]=\Sigma=I_d$, we obtain
\ba{
\abs{\frac{1}{n}\sum_{i=1}^n\sum_{j,k:j\neq k}\E[X_{ij}^2X_{ik}]}
&\leq\sum_{j=1}^d\sqrt{\frac{1}{n}\sum_{i=1}^n\E[X_{ij}^4]}\sqrt{\frac{1}{n}\sum_{i=1}^n\E\sbra{\bra{\sum_{k:k\neq j}X_{ik}}^2}}
\lesssim b^2d^{3/2}
}
and
\ba{
&\abs{\frac{1}{n}\sum_{i=1}^n\sum_{\begin{subarray}{c}
j,k,l=1\\
j,k,l\text{ are distinct}
\end{subarray}}^d\E[X_{ij}X_{ik}X_{il}]}\\
&\leq\sum_{j,k:j\neq k}\sqrt{\frac{1}{n}\sum_{i=1}^n\E[X_{ij}^2X_{ik}^2]}\sqrt{\frac{1}{n}\sum_{i=1}^n\E\sbra{\bra{\sum_{l:l\neq j,k}X_{il}}^2}}
\lesssim b^2d^{5/2}.
}
Consequently, $\mathbb{II}/(\log d)^{3/2}\to0$ and $\mathbb{III}/(\log d)^{3/2}\to0$. 
Hence we obtain the desired result.
\qed

\section{Proofs for Sections \ref{sec:coverage} and \ref{sec:db}}\label{sec:proof-2nd-order}

Given a random vector $W$, we denote by $F_W$ the distribution function of $W^\vee$. 

\subsection{Cornish--Fisher expansion}\label{sec:cf}

Asymptotic expansion of coverage probability is conventionally derived with the help of Cornish--Fisher expansion (cf.~Section 3.5.2 in \cite{Ha92}), so we first develop such expansions for $T_n$ and $T_n^*$ in our setting. 

%
\begin{theorem}[Cornish--Fisher expansion for $T_n$]\label{cf-tstat}
Under the assumptions of \cref{coro1}, let $\lambda>0$ be a constant such that $b/\sigma_*\leq\lambda$. 
Then, for any $\eps\in(0,1/2)$, there exist positive constants $c$ and $C$ depending only on $\lambda$ and $\eps$ such that if
\ben{\label{cf-ass}
\frac{\varsigma_d^3}{\sigma_*^3}\frac{\log^3d}{n}\log n\leq c,
}
then
\ben{\label{eq:cf-tstat}
\sup_{\eps<p<1-\eps}\abs{c_p-\bra{c_p^G-\frac{Q_n(c_{p}^G)}{f_{\Sigma}(c_{p}^G)}}}
\leq \frac{C}{\sqrt{\log d}}\frac{\varsigma_d^3}{\sigma_*^2}\frac{\log ^3d}{n}\log n,
}
where $c_p$ is the $p$-quantile of $T_n$. 
\end{theorem}

\begin{theorem}[Cornish--Fisher expansion for $T_n^*$]\label{cf-boot}
Under the assumptions of \cref{thm:coverage}, for any $\eps\in(0,1/2)$, there exist positive constants $c$ and $C$ depending only on $\lambda,\eps$ and $b_w$ such that if \eqref{cf-boot-ass} holds, then
\ben{\label{eq:cf-boot}
\sup_{\eps<p<1-\eps}\abs{\hat c_{p}-\bra{c_{p}^G-\frac{\hat Q_{n,\gamma}(c_{p}^G)}{f_{\Sigma}(c_{p}^G)}}}
\leq \frac{C}{\sqrt{\log d}}\frac{\varsigma_d^3}{\sigma_*^2}\frac{\log ^3(dn)}{n}\log n
}
with probability at least $1-1/n$, where
\ba{
\hat Q_{n,\gamma}(t)&:=\int_{A(t)}\{\hat p_{n,\gamma}(z)-\phi_\Sigma(z)\}dz.
}
\end{theorem}



The proofs are based on the following abstract result. 
\begin{proposition}[Abstract Cornish--Fisher type expansion for maximum statistics]\label{thm:cf}
Let $\eps\in(0,1/2)$. 
Also, let $W$ be a random vector in $\mathbb R^d$. 
Suppose that there exist arrays $U\in(\mathbb R^{d})^{\otimes2},V\in(\mathbb R^d)^{\otimes3}$ and a constant $\Delta>0$ such that
\ben{\label{formal-ae-w}
\sup_{t\in\mathbb R}\abs{P(W^\vee\leq t)-\int_{A(t)}p_{U,V}(z)dz}\leq\Delta,
}
where $p_{U,V}(z)=\phi_\Sigma(z)+\langle U,\nabla^2\phi_\Sigma(z)\rangle+\langle V,\nabla^3\phi_\Sigma(z)\rangle.$ 
Set
\[
\delta:=\|U\|_\infty\frac{\log d}{\sigma_*^2}+\|V\|_\infty\frac{\log^{3/2}d}{\sigma_*^3},\qquad
\ol\Delta:=\Delta+\frac{\varsigma_d}{\sigma_*}(\Delta+\delta)\delta.
\]
Then, there exist positive constants $c$ and $C$ depending only on $\eps$ such that if $\delta+\ol\Delta\leq c$, then
\ba{
\sup_{\eps<p<1-\eps}\abs{F_{W}^{-1}(p)-\bra{F_{Z}^{-1}(p)-\frac{Q_{U,V}(F_{Z}^{-1}(p))}{f_{\Sigma}(F_{Z}^{-1}(p))}}}
&\leq \frac{C}{\sqrt{\log d}}\bra{\varsigma_d\ol\Delta
+\frac{\varsigma_d^3}{\sigma_*^2}(\delta+\ol\Delta)^2},
}
where
\[
Q_{U,V}(t)=\int_{A(t)}\{p_{U,V}(z)-\phi_\Sigma(z)\}dz.
\]
\end{proposition}

The proof relies on the following novel isoperimetric-type inequality for $Z^\vee$ and its consequence to the second derivative of $F_Z^{-1}$:
\begin{lemma}\label{gmax-quantile}
Let $Z$ be a centered Gaussian vector in $\mathbb R^d$. 
If $Z^\vee$ has a continuous density $f$, then
\ben{\label{f-lower}
f(F_Z^{-1}(p))\geq\frac{1}{4\sqrt{\Var[Z^\vee]}}\min\cbra{\frac{p}{\sqrt 2},(1-p)^{3/2}}
}
for all $p\in(0,1)$. Moreover, if $\Cov[Z]=\Sigma$, there exists a universal constant $C>0$ such that 
\ben{\label{gmax-q-2nd}
|(F_Z^{-1})''(p)|\leq C\bra{\frac{\Var[Z^\vee]}{\min\{p^2,(1-p)^3\}}}^{3/2}\frac{\log d}{\sigma_*^2}
}
for all $p\in(0,1)$. 
\end{lemma}

\begin{proof}
See Appendix \ref{sec:gmax-quantile}. 
\end{proof}

\begin{rmk}\label{rmk:gmax-quantile}
Under the first assumption of \cref{gmax-quantile}, we can also derive the following Gaussian-type isoperimetric inequality for $Z^\vee$: For all $p\in(0,1)$,
\ben{\label{gmax-giso}
f(F_Z^{-1}(p))\geq\frac{1}{\sigma}\phi(\Phi^{-1}(p)),
}
where $\sigma:=\max_{1\leq j\leq d}\sqrt{\Var[Z_j]}$. 
In fact, by \cite[Proposition 5]{BaMa00}, \eqref{gmax-giso} follows once we prove
\[
\phi(\Phi^{-1}(\E[g(Z^\vee)]))\leq\E\sbra{\sqrt{\phi(\Phi^{-1}(g(Z^\vee)))^2+\sigma^{2}g'(Z^\vee)^2}}
\]
for any locally Lipschitz function $g:\mathbb R\to[0,1]$. 
The latter follows by applying Bobkov's functional Gaussian isoperimetric inequality to the function $x\mapsto g(\max_{1\leq j\leq d}(\Cov[Z]^{1/2}x)_j)$ (cf.~Eq.(2) of \cite{BaMa00}). 
While \eqref{gmax-giso} has a better dependence on $p$ than \eqref{f-lower}, it is often the case that $\Var[Z^\vee]=O(1/\log d)$ as already mentioned at the beginning of \cref{sec:coverage}, so \eqref{f-lower} is preferable to \eqref{gmax-giso} in terms of the dimension dependence. 
\end{rmk}

\begin{proof}[\bf\upshape Proof of \cref{thm:cf}]
Observe that
\ba{
Q_{U,V}(t)=\langle U,\int_{A(t)}\nabla^2\phi_\Sigma(z)dz\rangle+\langle V,\int_{A(t)}\nabla^3\phi_\Sigma(z)dz\rangle.
} 
Hence, by Lemmas \ref{lem:aht} and \ref{lem:anti}, there exists a universal constant $C_1\geq1$ such that
\ben{\label{q-aht}
|Q_{U,V}(t)|\leq C_1\delta
}
and
\ben{\label{q-anti}
|Q_{U,V}(t)-Q_{U,V}(s)|\leq C_1\delta\frac{\sqrt{\log d}}{\sigma_*}|t-s|
}
for all $t,s\in\mathbb R$. 
Also, for any $p\in(\eps,1-\eps)$, we have by \eqref{formal-ae-w}
\ben{\label{qz-upper}
p \leq F_W(F_W^{-1}(p))\leq F_Z(F_W^{-1}(p))+Q_{U,V}(F_W^{-1}(p))+\Delta
}
and
\ben{\label{qz-lower}
p\geq F_W(F_W^{-1}(p)-)\geq F_Z(F_W^{-1}(p))+Q_{U,V}(F_W^{-1}(p))-\Delta.
}
Combining these bounds with \eqref{q-aht} gives
$
p-\Delta-C_1\delta\leq F_Z(F_W^{-1}(p))\leq p+\Delta+C_1\delta.
$
Therefore, provided that $\Delta+C_1\delta<\eps/2$, we have by the mean value theorem and \eqref{f-lower}
\ba{
|F_Z^{-1}(p\pm(\Delta+C_1\delta))-F_Z^{-1}(p)|\leq C_2\sqrt{\Var[Z^\vee]}(\Delta+C_1\delta)
}
for some constant $C_2\geq1$ depending only on $\eps$. Thus we obtain
\ba{
|F_W^{-1}(p)-F_Z^{-1}(p)|\leq C_2\sqrt{\Var[Z^\vee]}(\Delta+C_1\delta).
}
This and \eqref{q-anti} give
\ba{
|Q_{U,V}(F_{W}^{-1}(p))-Q_{U,V}(F_{Z}^{-1}(p))|\leq C_1C_2\delta\frac{\varsigma_d}{\sigma_*}(\Delta+C_1\delta)=:\Delta'.
}
Combining this with \eqref{qz-upper} and \eqref{qz-lower}, we obtain
\ben{\label{fz-qw}
p-Q_{U,V}(F_Z^{-1}(p))-\Delta-\Delta'\leq F_Z(F_W^{-1}(p))\leq p-Q_{U,V}(F_Z^{-1}(p))+\Delta+\Delta'.
}
Thus, provided that $C_1\delta+\Delta+\Delta'<\eps/2$, we have by Taylor's theorem and \eqref{gmax-q-2nd}
\ba{
&\abs{F_Z^{-1}(p-Q_{U,V}(F_Z^{-1}(p))\pm(\Delta+\Delta'))-\bra{F_Z^{-1}(p)-\frac{Q_{U,V}(F_Z^{-1}(p))\mp(\Delta+\Delta')}{f_{\Sigma}(F_Z^{-1}(p))}}}\\
&\leq C_3\frac{\varsigma_d^3}{\sigma_*^2\sqrt{\log d}}|Q_{U,V}(F_Z^{-1}(p))\mp(\Delta+\Delta')|^2
}
for some constant $C_3\geq1$ depending only on $\eps$. 
Combining this with \eqref{f-lower}, \eqref{q-aht} and \eqref{fz-qw} gives
\ba{
&\abs{F_W^{-1}(p)-\bra{F_Z^{-1}(p)-\frac{Q_{U,V}(F_Z^{-1}(p))}{f_{\Sigma}(F_Z^{-1}(p))}}}\\
&\leq\frac{4\sqrt{\Var[Z^\vee]}}{\eps^{3/2}}(\Delta+\Delta')
+C_3\frac{\varsigma_d^3}{\sigma_*^2\sqrt{\log d}}(C_1\delta+\Delta+\Delta')^2.
}
Since $\Delta+\Delta'\leq C_1^2C_2\ol\Delta$, this completes the proof. 
\end{proof}

Now we turn to the proof of Theorems \ref{cf-tstat} and \ref{cf-boot}.
\begin{proof}[\bf\upshape Proof of \cref{cf-tstat}]
First, observe that \cref{gmax-var} yields
\ben{\label{gmax-lb}
\varsigma_d/\sigma_*\gtrsim\ul\sigma/\ol\sigma\geq\lambda^{-1}.
} 
Hence, due to \eqref{cf-ass}, we may assume 
\ben{\label{cf-wlog}
\bra{1+\frac{\varsigma_d^3}{\sigma_*^3}}\frac{\log^3d}{n}\log n\leq 1.
}
Then, we have \eqref{eq:ae} by \cref{coro1}. 
Also, observe that $\E[\ol{X^3}]\lesssim b^3/\sqrt n$. Hence, in this setting, $\delta$ and $\ol\Delta$ in \cref{thm:cf} are bounded as
\be{
\delta\leq C_\lambda\sqrt{\frac{\log ^3d}{n}},\qquad
\ol\Delta\leq C_{\lambda}\bra{1+\frac{\varsigma_d}{\sigma_*}}\frac{\log ^3d}{n}\log n,
}
where we used \eqref{cf-wlog} for the second inequality. Combining these bounds with \eqref{gmax-lb} and \eqref{cf-wlog} gives
\be{
\delta+\ol\Delta\leq C_\lambda\sqrt{\frac{\log ^3d}{n}\log n},\qquad
\ol\Delta\leq C_\lambda\frac{\varsigma_d}{\sigma_*}\frac{\log ^3d}{n}\log n.
}
Consequently, the desired result follows from \cref{thm:cf}. 
\end{proof}

\begin{proof}[\bf\upshape Proof of \cref{cf-boot}]
By the same reasoning as in the proof of \cref{cf-tstat}, we may assume
\ben{\label{cf-boot-wlog}
\bra{1+\frac{\varsigma_d^3}{\sigma_*^3}}\frac{\log^3(dn)}{n}\log n\leq 1.
}
Let $\mcl E_n$ be the event defined by \eqref{def:En}. Recall that $P(\mcl E_n)\geq1-1/n$. 
Also, by the proof of \cref{prop:boot}, we have \eqref{eq:ae-boot} on $\mcl E_n$. 
Further, recall that we have \eqref{2-mom-bound} and \eqref{boot-mom0} on $\mcl E_n$.  Hence, on $\mcl E_n$
\ba{
\frac{1}{2}\|\ol{X^2}-\Sigma\|_\infty\frac{\log d}{\sigma_*^2}+\frac{|\gamma|}{6}\|\ol{X^3}\|_\infty\frac{\log^{3/2}d}{\sigma_*^3}
\lesssim b^2\frac{\log^{3/2}(dn)}{\sigma_*^2\sqrt n}+\frac{|\gamma|b^3\log^{3/2}d}{\sigma_*^3\sqrt n}
\leq C_{\lambda,b_w}\sqrt{\frac{\log^3(dn)}{n}}.
}
Consequently, a similar argument to that in the proof of \cref{cf-tstat} gives the desired result. 
\end{proof}

\subsection{Proof of Theorem \ref{thm:coverage}}

\begin{lemma}\label{phi-deriv-upper}
For any $r\in\mathbb N$ and $t\in\mathbb R$,
\[
\int_{A(t)^c}\nabla^r\phi_\Sigma(z)dz=-\int_{A(t)}\nabla^r\phi_\Sigma(z)dz.
\]
\end{lemma}

\begin{proof}
Let $Z\sim N(0,\Sigma)$. Then, for any $x\in\mathbb R^d$,
\ba{
\int_{A(t)^c}\phi_\Sigma(z+x)dz=P(Z-x\in A(t)^c)=1-P(Z-x\in A(t))=1-\int_{A(t)}\phi_\Sigma(z+x)dz.
}
Differentiating both sides $r$ times with respect to $x$ and setting $x=0$, we obtain the desired result.  
\end{proof}

\begin{lemma}[Anti-concentration inequality for $T_n$]\label{anti-tstat}
Under the assumptions of \cref{coro1}, there exists a universal constant $C>0$ such that
\ba{
P\bra{t\leq T_n\leq t+\eps}\leq C\bra{\frac{b^5}{\sigma_*^5}\frac{\log^3 d}{n}\log n+\eps\bra{\frac{\sqrt{\log d}}{\ul\sigma}+\frac{b^3}{\sigma_*^4}\frac{\log^2d}{\sqrt n}}}
}
for all $t\in\mathbb R$ and $\eps>0$. 
\end{lemma}

\begin{proof}
The claim immediately follows by combining \cref{coro1} with Lemmas \ref{nazarov} and \ref{lem:anti}. 
\end{proof}

%

\begin{lemma}\label{lem:sk-bias}
Let $\eta$ be a centered random vector in $\mathbb R^r$ having a Stein kernel $\tau$ such that $\tau_{jk}=\tau_{kj}$ for all $j,k$. Assume $\E[|\eta|^2]<\infty$. 
Also, let $L$ and $M$ be a $d\times r$ matrix and an $r\times r$ symmetric matrix, respectively. 
Then, for any $a\in\mathbb R^d$, $\xi:=L\eta+\langle M,\eta^{\otimes2}-\E[\eta^{\otimes2}]\rangle a$ has an approximate Stein kernel $(\bar\tau,\res)$ such that
\ba{
\bar\tau(\xi)&=\E[L\tau(\eta)L^\top+4(L\tau(\eta)M\eta)\otimes a\mid\xi],\\
\res(\xi)&=\E[\langle M,2\tau(\eta)-\eta^{\otimes2}-\E[\eta^{\otimes2}]\rangle a\mid\xi].
}
\end{lemma}

\begin{proof}
Define a function $g:\mathbb R^r\to\mathbb R^d$ as $g(y)=Ly+\langle M,y^{\otimes2}-\E[\eta^{\otimes2}]\rangle a$, $y\in\mathbb R^r$. 
Take $h\in C^2_b(\mathbb R^d)$ arbitrarily and set $H:=h\circ g$. 
A straightforward computation shows that $H\in C^2_b(\mathbb R^d)$ and
\ba{
\nabla H(y)&=L^\top\nabla h(g(y))+2(a\cdot\nabla h(g(y)))My,\\
\nabla^2 H(y)&=L^\top\nabla^2 h(g(y))L+2\{L^\top\nabla^2h(g(y))a(My)^\top+My(L^\top\nabla^2h(g(y))a)^\top\}\\&\quad+2(a\cdot\nabla h(g(y)))M
}
for every $y\in\mathbb R^r$. 
In particular, $\eta\cdot\nabla H(\eta)=(\xi+\langle M,\eta^{\otimes2}+\E[\eta^{\otimes2}]\rangle a)\cdot\nabla h(\xi)$. 
Hence
\ba{
\E[(\xi-\E[\xi])\cdot\nabla h(\xi)]
&=\E[\eta\cdot\nabla H(\eta)]-\E[\langle M,\eta^{\otimes2}+\E[\eta^{\otimes2}]\rangle a\cdot\nabla h(\xi)]\\
&=\E[\langle\tau(\eta),\nabla^2 H(\eta)\rangle]-\E[\langle M,\eta^{\otimes2}+\E[\eta^{\otimes2}]\rangle a\cdot\nabla h(\xi)]\\
&=\E[\langle\bar\tau(\xi),\nabla^2 h(\xi)\rangle]+\E[\res(\xi)\cdot\nabla h(\xi)].
}
This completes the proof. 
\end{proof}

\begin{proof}[\bf\upshape Proof of \cref{thm:coverage}]
By the same reasoning as in the proof of \cref{coro1}, we may assume $\tau^X_{i,jk}=\tau^X_{i,kj}$ for all $i\in\{1,\dots,n\}$ and $j,k\in\{1,\dots,d\}$ without loss of generality. 

By Theorems \ref{cf-tstat} and \ref{cf-boot}, there exist positive constants $c$ and $C$ depending only on $\lambda,\eps$ and $b_w$ such that if \eqref{cf-boot-ass} holds, then we have \eqref{eq:cf-tstat} and \eqref{eq:cf-boot} with probability at least $1-1/n$. In the sequel we assume \eqref{cf-boot-ass} is satisfied with this $c$ and fix $\alpha\in(\eps,1-\eps)$ arbitrarily. By \eqref{f-lower} and Lemmas \ref{lem:aht} and \ref{lem:tensor} 
\ba{
\frac{1}{f_{\Sigma}(c_{1-\alpha}^G)}\abs{\hat Q_{n,\gamma}(c_{1-\alpha}^G)-\gamma Q_n(c_{1-\alpha}^G)-\frac{1}{2}\langle\ol{X^2}-\Sigma,\Psi_\alpha\rangle}
\lesssim\frac{|\gamma|\varsigma_d}{\eps^{3/2}\sqrt{\log d}}\frac{b^3}{\sigma_*^3}\frac{\log^{3/2}d}{n}\sqrt{\log n}
}
with probability at least $1-1/n$. 
Combining this with \eqref{eq:cf-tstat} and \eqref{eq:cf-boot}, we have
\ben{\label{cf-diff}
\abs{\hat c_{1-\alpha}-\tilde c_{1-\alpha}+\frac{\langle\ol{X^2}-\Sigma,\Psi_\alpha\rangle}{2f_{\Sigma}(c_{1-\alpha}^G)}}
\leq\frac{C_{\lambda,\eps,b_w}}{\sqrt{\log d}}\frac{\varsigma_d^3}{\sigma_*^2}\frac{\log ^3(dn)}{n}\log n
}
with probability at least $1-2/n$, where $\tilde c_{1-\alpha}:=c_{1-\alpha}+(1-\gamma)Q_n(c_{1-\alpha}^G)/f_{\Sigma}(c_{1-\alpha}^G)$. 
This and \cref{anti-tstat} give
\besn{\label{anti-tstat-applied}
&\abs{P(T_n\geq \hat c_{1-\alpha})-P\bra{T_n\geq \tilde c_{1-\alpha}-\frac{\langle\ol{X^2}-\Sigma,\Psi_\alpha\rangle}{2f_{\Sigma}(c_{1-\alpha}^G)}}}\\
&\leq C_{\lambda,\eps,b_w}\bra{\frac{\log^3 d}{n}\log n+\frac{\varsigma_d^3}{\sigma_*^3}\frac{\log ^3(dn)}{n}(\log n)\bra{1+\sqrt{\frac{\log^3d}{n}}}}+\frac{1}{n}\\
&\leq C_{\lambda,\eps,b_w}\frac{\varsigma_d^3}{\sigma_*^3}\frac{\log ^3(dn)}{n}\log n,
}
where the second inequality follows by \eqref{cf-boot-ass}. 
Now, observe that
\ben{\label{coverage-key}
\wt T_n:=T_n+\frac{\langle\ol{X^2}-\Sigma,\Psi_\alpha\rangle}{2f_{\Sigma}(c_{1-\alpha}^G)}=\max_{1\leq j\leq d}\frac{1}{\sqrt n}\sum_{i=1}^n\bra{X_{ij}+U_i},
}
where 
\ben{\label{def:Psi-tilde}
U_i:=\frac{\langle X_i^{\otimes2}-\E[X_i^{\otimes2}],\wt\Psi_\alpha\rangle}{\sqrt n},\qquad
\wt\Psi_\alpha:=\frac{\Psi_\alpha}{2f_{\Sigma}(c_{1-\alpha}^G)}.
}
By \cref{lem:sk-bias}, for every $i$, $\xi_i:=(X_i+U_i\bs1_d)/\sqrt n$ has an approximate Stein kernel $(\tau_i,\res_i)$ such that
\ba{
\tau_i(\xi_i)&=n^{-1}\E[\tau_i^X(X_i)+4n^{-1/2}\tau_i^X(X_i)(\wt\Psi_\alpha X_i)\otimes \bs1_d\mid\xi_i],\\
\res_i(\xi_i)&=n^{-1}\E[\langle \wt\Psi_\alpha,2\tau^X_i(X_i)-X_i^{\otimes2}-\E[X_i^{\otimes2}]\rangle \bs1_d\mid\xi_i].
}
Hence we can derive an Edgeworth expansion for $\wt T_n$ by applying \cref{thm:main} to $\xi_i$. 
We are going to bound the quantities appearing on the right hand side of \eqref{eq:main}. 
First, by \eqref{f-lower} and \cref{lem:aht}
\ben{\label{Psi-bound}
\|\wt\Psi_\alpha\|_1\leq C_\eps\frac{\varsigma_d\sqrt{\log d}}{\sigma_*^2}.
}
Hence, by \cref{sum-psi} and \eqref{cf-boot-ass},
\ben{\label{U-bound}
\norm{U_i}_{\psi_{1/2}}
\leq C_\eps\frac{b^2}{\sqrt n}\frac{\varsigma_d\sqrt{\log d}}{\sigma_*^2}
\leq C_{\lambda,\eps,b_w}\frac{b}{\log d},
}
and
\ben{\label{V-bound}
\max_{j,k}\norm{n^{-1/2}\bra{\tau_i^X(X_i)(\wt\Psi_\alpha X_i)\otimes\bs1_d}_{jk}}_{\psi_{1/3}}
\leq C_\eps\frac{b^3\varsigma_d\sqrt{\log d}}{\sigma_*^2\sqrt n}
\leq C_{\lambda,\eps,b_w}\frac{b^2}{\log d}.
}
These estimates allow us to prove \eqref{tau2-bound}--\eqref{xi3-tau-bound} with $b$ replaced by $C_{\lambda,\eps,b_w}b$ in a similar manner to the proof of \cref{coro1}. 
Further, observe that
\ba{
\max_{j}|\E[X_{ij}U_{i}]|\leq C_{\lambda,\eps}\frac{b^2}{\sqrt n}\frac{\varsigma_d\sqrt{\log d}}{\sigma_*},\qquad
|\E[U_{i}^2]|\leq C_{\lambda,\eps,b_w}\frac{b^2}{\sqrt n}\frac{\varsigma_d}{\sigma_*\sqrt{\log d}}.
}
Combining these estimates with \eqref{V-bound}, we can also prove \eqref{tbar-bound} with $b$ replaced by $C_{\lambda,\eps}b\sqrt{\varsigma_d/\sigma_*}$ similarly to the proof of \cref{coro1}. 
Next, by \cref{sum-psi} and \eqref{Psi-bound},
\be{
\max_j\|\res_{ij}(\xi_i)\|_{\psi_{1/2}}
\leq C_\eps\frac{b^2}{n}\frac{\varsigma_d\sqrt{\log d}}{\sigma_*^2}.
}
Therefore, similarly to the proof of \eqref{tau2-bound}, we obtain
\ba{
&\frac{\log^{2}d}{\sigma_*^{4}}\bra{\E\norm{\sum_{i=1}^n \xi_i^{\otimes3}\otimes\res_i(\xi_i)}_\infty
+\E\norm{\sum_{i=1}^n\res_i(\xi_i)\otimes\xi_i\otimes\tau_i(\xi_i)}_\infty}
\leq C_{\lambda,\eps,b_w}\frac{\varsigma_d}{\sigma_*}\frac{\log^3d}{n}
}
and
\ba{
\frac{\log^{3/2}d}{\sigma_*^3}\E\norm{\sum_{i=1}^n\res_i(\xi_i)^{\otimes2}\otimes\xi_i}_\infty
\leq C_{\lambda,\eps,b_w}\frac{\varsigma_d^2}{\sigma_*^2}\frac{\log^3d}{n}
}
and
\ben{\label{bias-xi-bound}
\frac{\log^2d}{\sigma_*^4}\E\norm{\sum_{i=1}^n \res_i(\xi_i)\otimes\xi_i}_\infty^2
\leq C_{\lambda,\eps,b_w}\frac{\varsigma_d^2}{\sigma_*^2}\frac{\log^3d}{n}.
}
In addition, since $\res_i(\xi_i)$ has identical components, 
\ben{\label{bias-bound}
\frac{\log d}{\sigma_*^2}\bra{\E\|\Res\|_\infty^2+\E\norm{\sum_{i=1}^n\res_i(\xi_i)^{\otimes2}}_\infty}
=\frac{2\log d}{\sigma_*^2}\sum_{i=1}^n\E[\res_{i1}(\xi_i)^2]
\leq C_{\lambda,\eps}\frac{\log^3 d}{n}.
}
Combining \eqref{bias-xi-bound} and \eqref{bias-bound} with \eqref{tbar-bound}, \eqref{xi3-bound} and the AM-GM inequality gives
\ba{
\frac{\log^{2}d}{\sigma_*^{4}}\bra{
\E\norm{\Res\otimes\sum_{i=1}^n \xi_i^{\otimes3}}_\infty
+\E\norm{\bar T\otimes\sum_{i=1}^n \res_i(\xi_i)\otimes\xi_i}_\infty
}\leq C_{\lambda,\eps,b_w}\frac{\varsigma_d^2}{\sigma_*^2}\frac{\log^3d}{n}
} 
and
\ba{
\frac{\log^{3/2} d}{\sigma_*^3}\E\norm{\Res\otimes\sum_{i=1}^n \res_i(\xi_i)\otimes\xi_i}_\infty
\leq C_{\lambda,\eps,b_w}\frac{\varsigma_d^2}{\sigma_*^2}\frac{\log^3d}{n}.
}
All together, we can proceed as in the proof of \cref{coro1} and then obtain
\ba{
\sup_{t\in\mathbb R}\abs{P(\wt T_n\leq t)-\int_{A(t)}\bra{p_{n}(z)+q_n(z)}dz}
&\leq C_{\lambda,\eps,b_w}\frac{\varsigma_d^2}{\sigma_*^2}\frac{\log^3 d}{n}\log n,
}
where 
\[
q_n(z)=\frac{1}{2n}\sum_{i=1}^n\langle 2\E[X_iU_i]\bs1_d^\top+\E[U_i^2]\bs1_d^{\otimes2},\nabla^2\phi_\Sigma(z)\rangle
-\frac{1}{6n^{3/2}}\sum_{i=1}^n\langle \E[(X_i+U_i\bs1_d)^{\otimes3}-X_i^{\otimes3}],\nabla^3\phi_\Sigma(z)\rangle.
\]
Therefore, in view of \cref{coro1}, it remains to prove
\ban{
\abs{\int_{A(\tilde c_{1-\alpha})^c}p_n(z) dz-\int_{A(c_{1-\alpha})^c}p_n(z) dz+(1-\gamma)Q_n(c_{1-\alpha}^G)}
&\leq C_{\lambda,\eps,b_w}\frac{\varsigma_d^2}{\sigma_*^2}\frac{\log ^3(dn)}{n},
\label{coverage-last1}\\
\abs{\int_{A(\tilde c_{1-\alpha})^c}q_n(z) dz+\E[R_n(\alpha)]}
&\leq C_{\lambda,\eps,b_w}\frac{\varsigma_d^3}{\sigma_*^3}\frac{\log ^3(dn)}{n}\log n.\label{coverage-last2}
}
Let us prove \eqref{coverage-last1}. By \cref{lem:anti},
\be{
\abs{\int_{A(\tilde c_{1-\alpha})^c}p_n(z) dz-\int_{A(c_{1-\alpha})^c}p_n(z) dz+\{F_Z(\tilde c_{1-\alpha})-F_Z(c_{1-\alpha})\}}\lesssim\frac{|1-\gamma|b^3\log^2d}{\sigma_*^4\sqrt n}\abs{\frac{Q_n(c_{1-\alpha}^G)}{f_{\Sigma}(c_{1-\alpha}^G)}}.
}
Also, by Taylor's theorem and \cref{gmax-2nd-deriv},
\be{
\abs{F_Z(\tilde c_{1-\alpha})-F_Z(c_{1-\alpha})-(1-\gamma)Q_n(c_{1-\alpha}^G)}
\lesssim\frac{|1-\gamma|\log d}{\sigma_*^2}\abs{\frac{Q_n(c_{1-\alpha}^G)}{f_{\Sigma}(c_{1-\alpha}^G)}}^2.
}
Further, \cref{lem:aht} and \eqref{f-lower} yield
\ben{\label{q/f-bound}
\abs{\frac{Q_n(c_{1-\alpha}^G)}{f_{\Sigma}(c_{1-\alpha}^G)}}\lesssim\frac{\varsigma_d}{\sqrt{\eps^3\log d}}\cdot \frac{b^3}{\sqrt n}\cdot\frac{\log^{3/2}d}{\sigma_*^3}
\leq C_{\lambda,\eps}\frac{\varsigma_d\log d}{\sqrt n}.
}
Combining these three estimates gives \eqref{coverage-last1}. 
Next, to prove \eqref{coverage-last2}, consider the following decomposition:
\ba{
&\int_{A(\tilde c_{1-\alpha})^c}q_n(z) dz\\
&=\frac{1}{n}\sum_{i=1}^n\langle \E[X_{i}U_i]\bs1_d^\top,\int_{A(\tilde c_{1-\alpha})^c}\nabla^2\phi_\Sigma(z)dz\rangle
+\frac{1}{2n}\sum_{i=1}^n\langle \E[U_i^2]\bs1_d^{\otimes2},\int_{A(\tilde c_{1-\alpha})^c}\nabla^2\phi_\Sigma(z)dz\rangle\\
&\quad-\frac{1}{6n^{3/2}}\sum_{i=1}^n\langle \E[(X_i+U_i\bs1_d)^{\otimes3}-X_i^{\otimes3}],\int_{A(\tilde c_{1-\alpha})^c}\nabla^3\phi_\Sigma(z)dz\rangle\\
&=:I+II+III.
}
We can rewrite $I$ as
\ba{
I&=\frac{1}{n}\sum_{i=1}^n\sum_{j,k=1}^d\E[X_{ij}U_i]\int_{A(\tilde c_{1-\alpha})^c}\partial_{jk}\phi_\Sigma(z)dz\\
&=\frac{1}{n^{3/2}}\sum_{i=1}^n\sum_{j,k,l,m=1}^d\E[X_{ij}X_{il}X_{im}]\wt\Psi_{\alpha,lm}\int_{A(\tilde c_{1-\alpha})^c}\partial_{jk}\phi_\Sigma(z)dz\\
&=\frac{1}{\sqrt n}\langle\E[\ol{X^3}]\otimes\bs1_d,\wt\Psi_\alpha\otimes\int_{A(\tilde c_{1-\alpha})^c}\nabla^2\phi_\Sigma(z)dz\rangle
=-\frac{1}{\sqrt n}\langle\E[\ol{X^3}]\otimes\bs1_d,\wt\Psi_\alpha\otimes\int_{A(\tilde c_{1-\alpha})}\nabla^2\phi_\Sigma(z)dz\rangle,
}
where we used \cref{phi-deriv-upper} for the last equality. 
We are going to prove $\tilde c_{1-\alpha}$ in the last expression can be replaced by $c_{1-\alpha}^G$. 
By \cref{lem:anti} and \eqref{Psi-bound},
\[
\norm{\wt\Psi_\alpha\otimes\bra{\int_{A(\tilde c_{1-\alpha})}\nabla^2\phi_\Sigma(z)dz-\int_{A(c_{1-\alpha}^G)}\nabla^2\phi_\Sigma(z)dz}}_1\lesssim\frac{\varsigma_d\log^{2}d}{\sigma_*^5}|\tilde c_{1-\alpha}-c_{1-\alpha}^G|.
\]
Also, by \eqref{eq:cf-tstat} and \eqref{q/f-bound},
\[
|\tilde c_{1-\alpha}-c_{1-\alpha}^G|\leq C_{\lambda,\eps,b_w}\bra{\frac{\varsigma_d\log d}{\sqrt n}+\frac{\varsigma_d^3}{\sigma_*^2}\frac{\log ^{5/2}(dn)}{n}\log n}.
\]
Consequently, we deduce
\ba{
\abs{I+\E[R_n(\alpha)]}
&\leq C_{\lambda,\eps,b_w}\frac{b^3\varsigma_d\log^{2}d}{\sigma_*^5\sqrt n}|\tilde c_{1-\alpha}-c_{1-\alpha}^G|
\leq C_{\lambda,\eps,b_w}\frac{\varsigma_d^3}{\sigma_*^3}\frac{\log^3(dn)}{n}\log n,
}
where we also used \eqref{cf-boot-ass} and \eqref{gmax-lb} for the last inequality. 
Meanwhile, by \cref{lem:aht} and \eqref{U-bound},
\ba{
|II|\leq C_\eps\frac{b^4}{n}\cdot\frac{\varsigma_d^2\log d}{\sigma_*^4}\frac{\log d}{\sigma_*^2}
\leq C_{\lambda,\eps}\frac{\varsigma_d^2}{\sigma_*^2}\frac{\log^2d}{n}
}
and
\ba{
|III|&\lesssim\frac{1}{\sqrt n}\max_{1\leq i\leq n}\bra{\|\E[X_i^{\otimes2}U_i]\|_\infty+\|\E[X_iU_i^2]\|_\infty+|\E[U_i^3]|}\frac{\log^{3/2}d}{\sigma_*^3}\\
&\leq C_{\lambda,\eps,b_w}\frac{b^4}{n}\frac{\varsigma_d\sqrt{\log d}}{\sigma_*^2}\frac{\log^{3/2}d}{\sigma_*^3}
\leq C_{\lambda,\eps,b_w}\frac{\varsigma_d}{\sigma_*}\frac{\log^2d}{n}.
}
All together, we complete the proof. 
\end{proof}

\subsection{Proof of Corollary \ref{coro:coverage}}

First, replacing $X_i$ by $X_i/\sigma$, we may assume $\sigma=1$ without loss of generality. Note that we have $\sigma_*^{-1}\lesssim\lambda$ under this assumption. 
Next, set $\rho:=\max_{1\leq j<k\leq d}|\Sigma_{jk}|<1$. Note that we have $1-\rho^2\geq\sigma_*^2$ because $1-\rho^2$ coincides with the minimum principal minor of $\Sigma$ of size 2. 

We begin by proving the following inequalities for every $\alpha\in(\eps,1-\eps)$:
\ben{\label{quantile-bounds}
|c_{1-\alpha}^G|\leq C_\eps\sqrt{\log d},\qquad
\exp\bra{-(c^G_{1-\alpha})^2/2}\leq C_{\lambda,\eps,K}d^{-1}\sqrt{\log d}.
}
The first one is an immediate consequence of Lemma A.6 in \cite{BCCHK18}. 
Meanwhile, by Eq.(4.2.9) in \cite{LLR83} (see also Eq.(4.2.1) in \cite{LLR83}), we have for any $u>0$
\ba{
|P(Z^\vee\leq u)-\Phi(u)^d|
&\leq C_\lambda\sum_{1\leq j<k\leq d}|\Sigma_{jk}|\exp\bra{-\frac{u^2}{1+|\Sigma_{jk}|}}.
}
Set $\mcl I:=\{(j,k)\in\{1,\dots,d\}^2:|\Sigma_{jk}|>u^{-2}\}$. 
Observe that
$
dK^2\geq\sum_{j,k=1}^d\Sigma_{jk}^2
\geq u^{-2}\sum_{(j,k)\in\mcl I}|\Sigma_{jk}|.
$
Hence we obtain
\ba{
|P(Z^\vee\leq u)-\Phi(u)^d|
&\leq C_\lambda\sum_{(j,k)\in\mcl I}|\Sigma_{jk}|e^{-u^2/(1+\rho)}
+C_\lambda\sum_{(j,k)\not\in\mcl I}|\Sigma_{jk}|e^{-u^2/(1+u^{-2})}\\
&\leq C_\lambda dK^2u^2e^{-u^2/(1+\rho)}+C_\lambda d^{3/2}Ke^{-u^2}.
}
Further, observe that $1-(\eps/2)^{1/d}<-d^{-1}\log(\eps/2)\leq1/2$. 
Thus, with $u=\Phi^{-1}((\eps/2)^{1/d})=-\Phi^{-1}(1-(\eps/2)^{1/d})$, we have by Lemma 10.3 in \cite{BLM13}
\ben{\label{blm-bound}
e^{-u^2/2}=\sqrt{2\pi}\phi(u)\leq-\sqrt{2\pi}d^{-1}\log(\eps/2)\sqrt{2\log(-d/\log(\eps/2))}
\leq C_\eps d^{-1}\sqrt{\log d}.
}
In addition, by the well-known inequality $\Phi(-s)=1-\Phi(s)\leq e^{-s^2/2}$ for all $s\geq0$, we deduce $u\leq\sqrt{-2\log (1-(\eps/2)^{1/d})}\leq C_\eps\sqrt{\log d}$. 
Hence we obtain
\ba{
P(Z^\vee\leq u)\leq\eps/2+C_{\lambda,\eps} \bra{d^{-(1-\rho)/(1+\rho)}K^2(\log d)^{1+1/(1+\rho)}+Kd^{-1/2}\log d}.
}
Therefore, if the second term on the right hand side is less than $\eps/2$, then $P(Z^\vee\leq u)\leq\eps<1-\alpha$, so $c_{1-\alpha}^G> u$. 
Hence the second bound in \eqref{quantile-bounds} follows by \eqref{blm-bound}. 
Otherwise, we have $d\leq d_0$ for some constant $d_0$ depending only on $\lambda,\eps$ and $K$, so the second bound in \eqref{quantile-bounds} trivially holds with sufficiently large $C_{\lambda,\eps,K}$. 

Now we turn to the main body of the proof. 
Since the left hand side of \eqref{eq:coro:coverage} is bounded by 1, we may assume \eqref{cf-boot-ass} holds with the constant $c$ in \cref{thm:coverage}. 
Also, we have $f_\Sigma(c_{1-\alpha}^G)\geq C_\eps\varsigma_d/\sqrt{\log d}$ by \eqref{f-lower}. 
Thus, the proof completes once we show that
\ben{\label{coro:coverage-aim}
|\E[\langle\ol{X^3}\otimes\bs1_d,\Psi_{\alpha}^{\otimes2}\rangle]|\leq C_{\lambda,\eps,K}d^{-1/2}\log^2d
}
for any $\alpha\in(\eps,1-\eps)$. 
Observe that
\ba{
\langle\ol{X^3}\otimes\bs1_d,\Psi_{\alpha}^{\otimes2}\rangle
=\frac{1}{n}\sum_{i=1}^n\abra{X_i^{\otimes3}\otimes\bs1_d,\Psi_\alpha^{\otimes2}}
=\frac{1}{n}\sum_{i=1}^n\langle X_i^{\otimes2},\Psi_\alpha\rangle\abra{X_i\otimes\bs1_d,\Psi_\alpha}.
}
Thus, by the Schwarz inequality
\ba{
|\E[\langle\ol{X^3}\otimes\bs1_d,\Psi_{\alpha}^{\otimes2}\rangle]|
\leq\sqrt{\frac{1}{n}\sum_{i=1}^n\E[\langle X_i^{\otimes2},\Psi_\alpha\rangle^2]\frac{1}{n}\sum_{i=1}^n\E[\abra{X_i\otimes\bs1_d,\Psi_\alpha}^2]}.
}
We have
\ba{
\frac{1}{n}\sum_{i=1}^n\E[\langle X_i^{\otimes2},\Psi_\alpha\rangle^2]
=\frac{1}{n}\sum_{i=1}^n\abra{\E[X_i^{\otimes4}],\Psi_\alpha^{\otimes2}}
\leq b^4\|\Psi_\alpha\|_1^2
\lesssim \lambda^4\log^2d,
}
where the last inequality follows from \cref{lem:aht}. 
Also, 
\ba{
\frac{1}{n}\sum_{i=1}^n\E[\abra{X_i\otimes\bs1_d,\Psi_\alpha}^2]
&=\frac{1}{n}\sum_{i=1}^n\sum_{j,k,l,m=1}^d\E[X_{ij}X_{ik}]\Psi_{\alpha,jl}\Psi_{\alpha,km}\\
&=\sum_{j,k,l,m=1}^d\Sigma_{jk}\Psi_{\alpha,jl}\Psi_{\alpha,km}
\leq K\sum_{j=1}^d\bra{\sum_{l=1}^d\Psi_{\alpha,jl}}^2.
}
Therefore, \eqref{coro:coverage-aim} follows once we show
\ben{\label{coro:coverage-aim2}
\sum_{j=1}^d\bra{\sum_{l=1}^d\Psi_{\alpha,jl}}^2\leq C_{\lambda,\eps,K}d^{-1}\log^2d.
}
Below we write $t=c_{1-\alpha}^G$ for short. 
Fix $j\in\{1,\dots,d\}$ arbitrarily. 
A straightforward computation shows
\ba{
\Psi_{\alpha,jj}=-\E[(\Sigma^{-1}Z)_j1_{\{\max_{k:k\neq j}Z_k\leq t\}}\mid Z_j=t]\phi(t),
}
where $Z\sim N(0,\Sigma)$. Hence
\ben{\label{Psi-jj-bound1}
|\Psi_{\alpha,jj}|\leq\E[|(\Sigma^{-1}Z)_j|\mid Z_j=t]\phi(t)
\leq(\Sigma^{-1})_{jj}|t|\phi(t)+\E[|(\Sigma^{-1})_{j,-j}\cdot Z_{-j}|\mid Z_j=t]\phi(t),
}
where $(\Sigma^{-1})_{j,-j}:=((\Sigma^{-1})_{jk})_{1\leq k\leq d:k\neq j}$ and $Z_{-j}:=(Z_k)_{1\leq k\leq d:k\neq j}$. 
It is well-known that the conditional distribution of $Z_{-j}$ given $Z_j=t$ is $N(t\Sigma_{j,-j},\Sigma_{-j,-j}-\Sigma_{j,-j}^{\otimes2})$ with $\Sigma_{j,-j}:=(\Sigma_{jk})_{1\leq k\leq d:k\neq j}$ and $\Sigma_{-j,-j}:=(\Sigma_{kl})_{1\leq k,l\leq d:k,l\neq j}$ (see e.g.~Theorem 1.2.5 in \cite{FUS10}). 
Therefore,
\ba{
\E[(\Sigma^{-1})_{j,-j}\cdot Z_{-j}\mid Z_j=t]
=(\Sigma^{-1})_{j,-j}\cdot t\Sigma_{j,-j}
=t\sum_{k:k\neq j}(\Sigma^{-1})_{jk}\Sigma_{jk}
=t\bra{1-(\Sigma^{-1})_{jj}}
}
and
\ba{
\Var[(\Sigma^{-1})_{j,-j}\cdot Z_{-j}\mid Z_j=t]
&=(\Sigma^{-1})_{j,-j}^\top(\Sigma_{-j,-j}-\Sigma_{j,-j}^{\otimes2})(\Sigma^{-1})_{j,-j}\\
&=\sum_{k,l:k,l\neq j}(\Sigma^{-1})_{jk}(\Sigma^{-1})_{jl}(\Sigma_{kl}-\Sigma_{jk}\Sigma_{jl})\\
&=\sum_{k:k\neq j}(\Sigma^{-1})_{jk}\bra{-(\Sigma^{-1})_{jj}\Sigma_{kj}-\Sigma_{jk}+\Sigma_{jk}(\Sigma^{-1})_{jj}}\\
&=(\Sigma^{-1})_{jj}-1.
}
Consequently, 
\ben{\label{Psi-jj-bound2}
|\E[|(\Sigma^{-1})_{j,-j}\cdot Z_{-j}|\mid Z_j=t]|
\leq \abs{t\bra{1-(\Sigma^{-1})_{jj}}}+\sqrt{(\Sigma^{-1})_{jj}-1}
\leq C_{\lambda}(1+|t|).
}
Combining \eqref{Psi-jj-bound1} and \eqref{Psi-jj-bound2} with \eqref{quantile-bounds} gives
\ben{\label{Psi-jj-bound}
|\Psi_{\alpha,jj}|\leq C_{\lambda}(1+|t|)\phi(t)\leq C_{\lambda,\eps,K}d^{-1}\log d.
}
Next, fix $l\in\{1,\dots,d\}\setminus\{j\}$ arbitrarily. Then we have
\ba{
|\Psi_{\alpha,jl}|\leq P\bra{\max_{k:k\neq j,l}Z_k\leq t\mid Z_j=t,Z_l=t}\phi_{\Sigma[j,l]}(t,t)
\leq \phi_{\Sigma[j,l]}(t,t),
}
where $\Sigma[j,l]=(\Sigma_{pq})_{p,q\in\{j,l\}}$. Hence
\ba{
|\Psi_{\alpha,jl}|\leq\frac{1}{2\pi\sqrt{1-\rho^2}}\exp\bra{-\frac{t^2}{1+\Sigma_{jl}}}.
}
Let $\mcl I_j:=\{k\in\{1,\dots,d\}:|\Sigma_{jk}|>t^{-2}\}$. 
Then we have
\ba{
K^2\geq\sum_{k=1}^d\Sigma_{jk}^2\geq\sum_{k\in \mcl I_j}\Sigma_{jk}^2>t^{-4}\#\mcl I_j,
}
where $\#\mcl I_j$ is the number of elements in $\mcl I_j$. 
Hence we obtain
\ba{
\sum_{l:l\neq j}|\Psi_{\alpha,jl}|
&\leq C_\lambda\sum_{l\in \mcl I_j}e^{-t^2/(1+\rho)}
+C_\lambda\sum_{l\notin \mcl I_j}e^{-t^2/(1+t^{-2})}
\leq C_\lambda(K^2t^4e^{-t^2/(1+\rho)}+de^{-t^2})\\
&\leq C_{\lambda,\eps,K}\bra{d^{-2/(1+\rho)}(\log d)^{2+1/(1+\rho)}+d^{-1}\log d}
\leq C_{\lambda,\eps,K}d^{-1}\log d,
}
where the third inequality follows from \eqref{quantile-bounds}. 
Combining this with \eqref{Psi-jj-bound} gives \eqref{coro:coverage-aim2}.
 \qed

\subsection{Proof of Corollary \ref{coro:lb-gwb}}

In view of \cref{thm:coverage}, the asserted claim immediately follows from \eqref{aim:lb-ga}, \eqref{coro:coverage-aim} and \eqref{f-lower}. 
\qed

\subsection{Proof of Corollary \ref{coro:factor}}

Let us prove \eqref{eq:factor}. 
Since $\sigma_*\geq\sqrt{1-\rho}$ and $\varsigma_d=O(b\sqrt{\log d})$, we have $\frac{\varsigma_d^3}{\sigma_*^3}\frac{\log ^3(dn)}{n}\log n=o(n^{-1/2})$ by assumption. 
Also, observe that
\ba{
&(1-\gamma)Q_n(c_{1-\alpha}^G)+\E[R_n(\alpha)]\\
&=-(1-\gamma)\frac{\rho^{3/2}\E[U^3]}{6\sqrt n}\langle\bs1_d^{\otimes3},\int_{A(c^G_{1-\alpha})}\nabla^3\phi_\Sigma(z)dz\rangle
+\frac{\rho^{3/2}\E[U^3]}{2\sqrt n}\frac{\abra{\bs1_d^{\otimes2},\Psi_{\alpha}}^2}{f_{\Sigma}(c_{1-\alpha}^G)}
+(1-\rho)^{3/2}\Upsilon_{n,\gamma}(\alpha).
}
Therefore, in view of \cref{thm:coverage}, it suffices to prove
\ba{
\langle\bs1_d^{\otimes3},\int_{A(c^G_{1-\alpha})}\nabla^3\phi_\Sigma(z)dz\rangle
&=\rho^{-\frac{3}{2}}(z_\alpha^2-1)\phi(z_\alpha)+o(1),\\
\frac{\abra{\bs1_d^{\otimes2},\Psi_{\alpha}}^2}{f_{\Sigma}(c_{1-\alpha}^G)}
&=\rho^{-\frac{3}{2}}z_\alpha^2\phi(z_\alpha)+o(1).
}

Observe that $Z$ has the same law as $\sqrt\rho\zeta\bs1_d+\sqrt{1-\rho}G$ with $\zeta\sim N(0,1)$ and $G\sim N(0,I_d)$ independent. 
Hence, $F_Z(t)=\E[\Phi((t-\sqrt{1-\rho}G^\vee)/\sqrt\rho)]$ for every $t\in\mathbb R$. 
Set $a_d:=\sqrt{1-\rho}\E[G^\vee]$. 
Since $\Var[G^\vee]=O(1/\log d)=o(1)$, we have $F_Z(t)=\Phi((t-a_d)/\sqrt\rho)+o(1)$ uniformly in $t\in\mathbb R$ by the Lipschitz continuity of $\Phi$. 
Hence $1-\alpha=F_Z(c^G_{1-\alpha})=\Phi((c^G_{1-\alpha}-a_d)/\sqrt\rho)+o(1)$ and thus $(c^G_{1-\alpha}-a_d)/\sqrt\rho=\Phi^{-1}(1-\alpha)+o(1)=-z_\alpha+o(1)$. 
Then, for any integer $r\geq0$, the Lipschitz continuity of $\phi^{(r)}$ gives
\ban{
f_\Sigma^{(r)}(c^G_{1-\alpha})
&=\rho^{-\frac{r+1}{2}}\E[\phi^{(r)}((c^G_{1-\alpha}-\sqrt{1-\rho}G^\vee)/\sqrt\rho)]\notag\\
&=(-1)^{r}\rho^{-\frac{r+1}{2}}\phi^{(r)}(z_\alpha)+o(1).\label{factor-dens}
}
Combining this with \eqref{gmax-dens-deriv} gives
\ba{
\langle\bs1_d^{\otimes3},\int_{A(c^G_{1-\alpha})}\nabla^3\phi_\Sigma(z)dz\rangle
=\rho^{-\frac{3}{2}}\phi''(z_\alpha)+o(1)
=\rho^{-\frac{3}{2}}(z_\alpha^2-1)\phi(z_\alpha)+o(1)
}
and
\ba{
\frac{\abra{\bs1_d^{\otimes2},\Psi_{\alpha}}^2}{f_{\Sigma}(c_{1-\alpha}^G)}
=\frac{\rho^{-2}\phi'(z_\alpha)^2}{\rho^{-\frac{1}{2}}\phi(z_\alpha)}+o(1)
=\rho^{-\frac{3}{2}}z_\alpha^2\phi(z_\alpha)+o(1).
}
This completes the proof of \eqref{eq:factor}. 

Next we prove $\Upsilon_{n,\gamma}(\alpha)=o(n^{-1/2})$ when $\gamma=1$. 
In view of \eqref{factor-dens} with $r=0$, it suffices to prove
\[
\abra{\E[V^{\otimes3}]\otimes\bs1_d,\Psi_{\alpha}^{\otimes2}}=o(1).
\]
Observe that
\ba{
|\abra{\E[V^{\otimes3}]\otimes\bs1_d,\Psi_{\alpha}^{\otimes2}}|
&=|\E[\abra{V^{\otimes2},\Psi_{\alpha}}\abra{V\otimes\bs1_d,\Psi_{\alpha}}]|\\
&\leq\sqrt{\E[\abra{V^{\otimes2},\Psi_{\alpha}}^2]\E[\abra{V\otimes\bs1_d,\Psi_{\alpha}}^2]}.
}
We have $\E[\abra{V^{\otimes2},\Psi_{\alpha}}^2]=\abra{\E[V^{\otimes4}],\Psi_{\alpha}^{\otimes2}}=O(\log^2d)$ by \cref{lem:aht}. 
Meanwhile, observe that $\phi_\Sigma$ is symmetric as a $d$-variate function due to the structure of $\Sigma$. Hence, $\Psi_{\alpha,jj}=\Psi_{\alpha,11}$ and $\Psi_{\alpha,jk}=\Psi_{\alpha,12}$ if $j\neq k$. 
Therefore,
\ba{
\E[\abra{V\otimes\bs1_d,\Psi_{\alpha}}^2]
&=\E\sbra{\bra{\sum_{j=1}^dV_j}^2}\bra{\Psi_{\alpha,11}+(d-1)\Psi_{\alpha,12}}^2\\
&=d\bra{\Psi_{\alpha,11}+(d-1)\Psi_{\alpha,12}}^2
=\frac{1}{d}\bra{\sum_{j,k=1}^d\Psi_{jk}}^2.
}
Since $\sum_{j,k=1}^d\Psi_{jk}=O(1)$ by \eqref{gmax-dens-deriv} and \eqref{factor-dens}, we conclude $\E[\abra{V\otimes\bs1_d,\Psi_{\alpha}}^2]=O(d^{-1})$. 
Consequently, we obtain $\abra{\E[V^{\otimes3}]\otimes\bs1_d,\Psi_{\alpha}^{\otimes2}}=O(d^{-1/2}\log d)=o(1)$. 
\qed

\subsection{Proof of Theorem \ref{thm:db}}

For $\alpha\in(0,1)$, we denote by $\hat c_{1-\alpha}^*$ the $(1-\alpha)$-quantile of $T_n^{**}$ under $P^{**}$. 
%
\begin{lemma}\label{db-coverage}
Under the assumptions of \cref{thm:db},  there exist positive constants $c$ and $C$ depending only on $\lambda,\eps,b_w$ and $b_v$ such that if \eqref{cf-boot-ass} holds, then
\be{
\sup_{\eps<\alpha<1-\eps}P\bra{\abs{P^*(T_n^*\geq \hat c_{1-\alpha}^{*})-\bra{\alpha-R_n(\alpha)}}
> C\frac{\varsigma_d^3}{\sigma_*^3}\frac{\log ^3(dn)}{n}\log n}\leq\frac{5}{n}.
}
\end{lemma}

\begin{proof}
The proof is basically a straightforward modification of that of \cref{thm:coverage}. 
We only give a sketch of the proof with emphasis on relatively major changes. 

Fix $\alpha\in(\eps,1-\eps)$ arbitrarily. 
First, it is not difficult to see that an analogous result to \cref{cf-boot} holds for $T_n^{**}$. 
Thus, by a similar argument to the proof of \eqref{cf-diff}, we can find a constant $c$ depending only on $\lambda,\eps,b_w$ and $b_v$ and an event $\mcl E_n^*(\alpha)$ satisfying the following conditions: 
\begin{enumerate}[label=(\roman*)]

\item If \eqref{cf-boot-ass} holds, then we have on $\mcl E_n^*(\alpha)$
\ba{
\abs{\hat c_{1-\alpha}^*-\hat c_{1-\alpha}+\frac{1}{2f_{\Sigma}(c_{1-\alpha}^G)}\langle\frac{1}{n}\sum_{i=1}^n(w_i^2-1)X_i^{\otimes2},\Psi_\alpha\rangle}
\leq\frac{C_{\lambda,\eps,b_w,b_v}}{\sqrt{\log d}}\frac{\varsigma_d^3}{\sigma_*^2}\frac{\log ^3(dn)}{n}\log n.
}

\item \eqref{eq:ae-boot} holds on $\mcl E_n^*(\alpha)$. 

\item\label{event*-3} $P(\mcl E_n^*(\alpha))\geq1-1/n^2$. 

\end{enumerate}
In the sequel we assume \eqref{cf-boot-ass} is satisfied with the above $c$. Then, by a similar argument to the proof of \eqref{anti-tstat-applied}, we obtain
\ba{
\abs{P^*(T_n^*\geq \hat c_{1-\alpha}^*)-P^*\bra{T_n^*\geq \hat c_{1-\alpha}-J_n^*(\alpha)}}
&\leq C_{\lambda,\eps,b_w,b_v}\frac{\varsigma_d^3}{\sigma_*^3}\frac{\log ^3(dn)}{n}\log n+P^*(\mcl E_n^*(\alpha)^c),
}
where $J_n^*(\alpha):=\langle n^{-1}\sum_{i=1}^n(w_i^2-1)X_i^{\otimes2},\wt\Psi_\alpha\rangle$ and $\wt\Psi_\alpha$ is defined as in \eqref{def:Psi-tilde}. Since $P(P^*(\mcl E_n^*(\alpha)^c)\geq1/n)\leq n\E[P^*(\mcl E_n^*(\alpha)^c)]\leq1/n$ by Markov's inequality and \ref{event*-3}, we conclude
\[
\abs{P^*(T_n^*\geq \hat c_{1-\alpha}^*)-P^*\bra{T_n^*\geq \hat c_{1-\alpha}-J_n^*(\alpha)}}
\leq C_{\lambda,\eps,b_w,b_v}\frac{\varsigma_d^3}{\sigma_*^3}\frac{\log ^3(dn)}{n}\log n
\] 
with probability at least $1-1/n$. 
As in the proof of \cref{thm:coverage}, we derive an Edgeworth expansion for $T_n^*+J_n^*(\alpha)$ by applying \cref{thm:main} to $\xi_i:=(w_i\wt X_i+(w_i^2-1)U_i\bs1_d)/\sqrt n$ conditional on the data, where $\wt X_i:=X_i-\bar X$ and $U_i:=\langle X_i^{\otimes2},\wt\Psi_\alpha\rangle/\sqrt n$. 
By \cref{lem:sk-bias}, conditional on the data, $\xi_i$ has an approximate Stein kernel $(\tau_i,\res_i)$ such that 
\ba{
\tau_i(\xi_i)&=n^{-1}\E^*[\tau^*(w_i)\wt X_i^{\otimes2}+V_i\mid\xi_i],\\
\res_i(\xi_i)&=n^{-1}\E^*[\langle X_i^{\otimes2},\wt\Psi_\alpha\rangle(2\tau^*(w_i)-w_i^2-1) \bs1_d\mid\xi_i],
}
where $V_i:=4n^{-1/2}(\tau^*(w_i)\langle X_i^{\otimes2},\wt\Psi_\alpha\rangle w_i)\wt X_i\otimes \bs1_d$. 
As in the proof of \cref{thm:coverage}, it is not difficult to check that we have the estimates corresponding to \eqref{boot-bound2}--\eqref{boot-bound6} in the present setting with probability at least $1-1/(2n)$ by a similar argument to the proof of \cref{prop:boot}. 
Meanwhile, by the Schwarz inequality,
\ba{
\norm{\frac{1}{n}\sum_{i=1}^nV_i}_{\infty}
\leq\frac{4b_w^3}{\sqrt n}\sqrt{\frac{1}{n}\sum_{i=1}^n\langle X_i^{\otimes2},\wt\Psi_\alpha\rangle^2}\sqrt{\max_{1\leq j\leq d}\frac{1}{n}\sum_{i=1}^n\wt X_{ij}^2}.
}
Observe that $\langle X_i^{\otimes2},\wt\Psi_\alpha\rangle^2=\langle X_i^{\otimes4},\wt\Psi_\alpha^{\otimes2}\rangle$. 
Hence, by \eqref{Psi-bound} and \cref{lem:tensor},
\ben{\label{x2-psi-bound}
\frac{1}{n}\sum_{i=1}^n\langle X_i^{\otimes2},\wt\Psi_\alpha\rangle^2
\leq C_\eps\frac{b^4\varsigma_d^2\log d}{\sigma_*^4}
\leq C_{\lambda,\eps}\frac{b^2\varsigma_d^2\log d}{\sigma_*^2}
}
with probability at least $1-1/(2n)$. 
Combining these estimates with \eqref{boot-mom} and the argument to prove \eqref{tbar-bound*}, we obtain the estimate corresponding to \eqref{tbar-bound*} with $b_wb$ replaced by $C_{\lambda,\eps,b_w,b_v}b\sqrt{\varsigma_d/\sigma_*}$ with probability at least $1-1/n$.  
Moreover, by \eqref{mean-bound}, \eqref{boot-mom0}, \eqref{Psi-bound} and \eqref{x2-psi-bound}, we have
\ba{
\frac{\log^{2}d}{\sigma_*^{4}}\bra{\E^*\norm{\sum_{i=1}^n \xi_i^{\otimes3}\otimes\res_i(\xi_i)}_\infty
+\E^*\norm{\sum_{i=1}^n\res_i(\xi_i)\otimes\xi_i\otimes\tau_i(\xi_i)}_\infty}
&\leq C_{\lambda,\eps,b_w,b_v}\frac{\varsigma_d}{\sigma_*}\frac{\log^3(dn)}{n},\\
\frac{\log^{3/2}d}{\sigma_*^3}\E^*\norm{\sum_{i=1}^n\res_i(\xi_i)^{\otimes2}\otimes\xi_i}_\infty
&\leq C_{\lambda,\eps,b_w,b_v}\frac{\varsigma_d^2}{\sigma_*^2}\frac{\log^3(dn)}{n},\\
\frac{\log^2d}{\sigma_*^4}\E^*\norm{\sum_{i=1}^n \res_i(\xi_i)\otimes\xi_i}_\infty^2
&\leq C_{\lambda,\eps,b_w,b_v}\frac{\varsigma_d^2}{\sigma_*^2}\frac{\log^3(dn)}{n},\\
\frac{\log d}{\sigma_*^2}\bra{\E^*\norm{\sum_{i=1}^n\res_i(\xi_i)}_\infty^2+\E^*\norm{\sum_{i=1}^n\res_i(\xi_i)^{\otimes2}}_\infty}
&\leq C_{\lambda,\eps,b_w}\frac{\log^3 d}{n}
}
with probability at least $1-1/n$. 
All together, we can proceed as in the proof of \cref{prop:boot} and then obtain
\ba{
\sup_{t\in\mathbb R}\abs{P^*(T_n^*+J_n^*(\alpha)\leq t)-\int_{A(t)}\bra{\hat p_{n,1}(z)+\hat q_n(z)}dz}
&\leq C_{\lambda,\eps,b_w,b_v}\frac{\varsigma_d^2}{\sigma_*^2}\frac{\log^3 (dn)}{n}\log n
}
with probability at least $1-3/n$, where 
\ba{
\hat q_n(z)=\frac{1}{2n}\sum_{i=1}^n\langle 2\wt X_iU_i\bs1_d^\top+U_i^2\bs1_d^{\otimes2},\nabla^2\phi_\Sigma(z)\rangle
-\frac{1}{6n^{3/2}}\sum_{i=1}^n\langle \E^*[(w_i\wt X_i+U_i\bs1_d)^{\otimes3}]-\wt X_i^{\otimes3},\nabla^3\phi_\Sigma(z)\rangle.
}
The remaining proof is a minor modification of the proof of \eqref{coverage-last2}, so we omit the details. 
\end{proof}

\begin{lemma}\label{lem:rn-lip}
Under the assumptions of \cref{thm:db}, there exists a constant $C>0$ depending only on $\lambda$ and $\eps$ such that
\ben{\label{rn-lip}
|\E[R_n(\alpha+\delta)]-\E[R_n(\alpha)]|
\leq C\delta\frac{\varsigma_d^3}{\sigma_*^3}\sqrt{\frac{\log^{3}d}{n}}
}
for any $\alpha\in(\eps,1-2\eps)$ and $\delta\in(0,\eps]$. 
\end{lemma}

\begin{proof}
By \cref{lem:aht} and \eqref{f-lower},
\ba{
|\E[R_n(\alpha+\delta)]-\E[R_n(\alpha)]|
\leq C_{\lambda,\eps}\frac{b^3}{\sqrt n}\bra{\frac{\log^2 d}{\sigma_*^4}\abs{\frac{1}{f_{\Sigma}(c_{1-\alpha-\delta}^G)}-\frac{1}{f_{\Sigma}(c_{1-\alpha}^G)}}+\frac{\log d}{\sigma_*^2}\frac{\varsigma_d}{\sqrt{\log d}}|\Psi_{\alpha+\delta}-\Psi_\alpha|}.
}
Noting that $(F_Z^{-1})'(p)=1/f_\Sigma(c^G_p)$ for all $p\in(0,1)$, we obtain by the mean value theorem and \eqref{gmax-q-2nd}
\ba{
\abs{\frac{1}{f_{\Sigma}(c_{1-\alpha-\delta}^G)}-\frac{1}{f_{\Sigma}(c_{1-\alpha}^G)}}
\leq C_\eps\delta\frac{\varsigma_d^3}{\sigma_*^2\sqrt{\log d}}.
}
Also, by \cref{lem:anti}, the mean value theorem and \eqref{f-lower},
\ba{
|\Psi_{\alpha+\delta}-\Psi_\alpha|
\lesssim \frac{\log^{3/2}d}{\sigma_*^3}|c^G_{1-\alpha-\delta}-c^G_{1-\alpha}|
\leq C_\eps\delta\frac{\varsigma_d\log d}{\sigma_*^3}.
}
Combining these bounds gives \eqref{rn-lip}.
\end{proof}

\begin{proof}[Proof of \cref{thm:db}]
Denote by $c_1$ and $C_1$ the constants $c$ and $C$ in \cref{thm:coverage}, respectively. 
Also, denote by $c_2$ and $C_2$ the constants $c$ and $C$ in \cref{db-coverage}, respectively. 
Since the left hand side of \eqref{db-aim} is bounded by 1, we may assume \eqref{cf-boot-ass} holds with $c=c_1\wedge c_2$ without loss of generality. 
Then, for each $\alpha\in(\eps,1-\eps)$, the event
\be{
\mcl E_n(\alpha):=\cbra{\abs{P^*(T_n^*\geq \hat c_{1-\alpha}^{*})-\bra{\alpha-R_n(\alpha)}}
\leq C_2\frac{\varsigma_d^3}{\sigma_*^3}\frac{\log ^3(dn)}{n}\log n}
}
occurs with probability at least $1-5/n$. 
Meanwhile, by \eqref{cf-boot-ass}, \eqref{f-lower} and Lemmas \ref{lem:aht} and \ref{max-weibull}, there exists a constant $C_3>0$ depending only on $\lambda,\eps$ such that the event
\be{
\mcl E_n:=\cbra{\sup_{\eps<\alpha<1-\eps}|R_n(\alpha)-\E[R_n(\alpha)]|
\leq C_3\frac{\varsigma_d}{\sigma_*}\frac{\log^3(dn)}{n}}
}
occurs with probability at least $1-1/n$ and
\ben{\label{rn-bound}
\sup_{\eps<\alpha<1-\eps}|\E[R_n(\alpha)]|
\leq C_3\frac{\varsigma_d}{\sigma_*}\sqrt{\frac{\log^{3}d}{n}}.
}
Further, let $C_4$ be the constant $C$ in \cref{lem:rn-lip}. 
Set
\be{
\Delta_n:=C_2\frac{\varsigma_d^3}{\sigma_*^3}\frac{\log ^3(dn)}{n}\log n+C_3\frac{\varsigma_d}{\sigma_*}\frac{\log^3(dn)}{n}+C_3C_4\frac{\varsigma_d^4}{\sigma_*^4}\frac{\log^3d}{n}.
}
Since the left hand side of \eqref{db-aim} is bounded by 1, we may assume without loss of generality
\ben{\label{wlog-db}
C_3\frac{\varsigma_d}{\sigma_*}\sqrt{\frac{\log^{3}d}{n}}+3\Delta_n\leq \eps.
}

Now fix $\alpha\in(2\eps,1-2\eps)$ arbitrarily. 
Set $\alpha_{n,1}:=\alpha-\Delta_n$ and $\alpha_{n,1}':=\alpha_{n,1}+\E[R_n(\alpha_{n,1})]$. 
By \eqref{rn-bound} and \eqref{wlog-db}, $\alpha_{n,1},\alpha_{n,1}'\in(\eps,1-\eps)$. 
Hence, on $\mcl E_n(\alpha_{n,1}')\cap\mcl E_n$,
\ba{
&P^*(\hat F_n^*(T_n^*)>1-\alpha_{n,1}')
\leq P^*(T_n^*\geq\hat c_{1-\alpha_{n,1}'}^*)
\leq\alpha_{n,1}'-R_n(\alpha_{n,1}')+C_2\frac{\varsigma_d^3}{\sigma_*^3}\frac{\log ^3(dn)}{n}\log n\\
&\quad\leq\alpha_{n,1}+\E[R_n(\alpha_{n,1})]-\E[R_n(\alpha_{n,1}')]+C_3\frac{\varsigma_d}{\sigma_*}\frac{\log^{3}(dn)}{n}+C_2\frac{\varsigma_d^3}{\sigma_*^3}\frac{\log ^3(dn)}{n}\log n
\leq\alpha,
}
where the last inequality follows from \eqref{rn-lip} and \eqref{rn-bound}. 
This yields $\hat\beta_\alpha\leq1-\alpha_{n,1}'$ on $\mcl E_n(\alpha_{n,1}')\cap\mcl E_n$. 
Hence
\ban{
P(T_n\geq\hat c_{\hat\beta_\alpha})
&\geq P(T_n\geq\hat c_{1-\alpha_{n,1}'})-\frac{6}{n}
\geq\alpha_{n,1}'-\E[R_n(\alpha_{n,1}')]-C_{\lambda,\eps,b_w}\frac{\varsigma_d^3}{\sigma_*^3}\frac{\log ^3(dn)}{n}\log n
\notag\\
&\geq\alpha+\E[R_n(\alpha_{n,1})]-\E[R_n(\alpha_{n,1}')]-C_{\lambda,\eps,b_w,b_v}\frac{\varsigma_d^4}{\sigma_*^4}\frac{\log ^3(dn)}{n}\log n
\notag\\
&\geq\alpha-C_{\lambda,\eps,b_w,b_v}\frac{\varsigma_d^4}{\sigma_*^4}\frac{\log ^3(dn)}{n}\log n,\label{db-lower}
}
where the second inequality is by \cref{thm:coverage}, the third by $\Delta_n\leq C_{\lambda,\eps,b_w,b_v}\frac{\varsigma_d^4}{\sigma_*^4}\frac{\log^3d}{n}\log n$ and the fourth by \eqref{rn-lip} and \eqref{rn-bound}. 
Similarly, with $\alpha_{n,2}:=\alpha+2\Delta_n$ and $\alpha_{n,2}':=\alpha_{n,2}+\E[R_n(\alpha_{n,2})]$, we have on $\mcl E_n(\alpha_{n,2}')\cap\mcl E_n$
\ba{
&P^*(\hat F_n^*(T_n^*)>1-\alpha'_{n,2}-\Delta_n)
\geq P^*(\hat F_n^*(T_n^*)\geq 1-\alpha'_{n,2})
=P^*(T_n^*\geq\hat c_{1-\alpha'_{n,2}}^*)\\
&\quad\geq\alpha'_{n,2}-R_n(\alpha_{n,2}')-C_2\frac{\varsigma_d^3}{\sigma_*^3}\frac{\log ^3(dn)}{n}\log n
>\alpha.
}
Hence $\hat\beta_\alpha>1-\alpha_{n,2}'-\Delta_n$ on $\mcl E_n(\alpha_{n,2}')\cap\mcl E_n$. 
Therefore, a similar argument to \eqref{db-lower} yields
\[
P(T_n\geq\hat c_{\hat\beta_\alpha})\leq\alpha+C_{\lambda,\eps,b_w,b_v}\frac{\varsigma_d^4}{\sigma_*^4}\frac{\log ^3(dn)}{n}\log n.
\]
Combining this and \eqref{db-lower} gives the desired result. 
\end{proof}

\appendix

\paragraph{Appendix}

\section{Nearly optimal high-dimensional CLT under the sub-exponential condition}\label{sec:cck23}

\begin{theorem}
Set $\sigma_j:=\sqrt{\Var[S_{n,j}]}$ for $j=1,\dots,d$. 
Suppose that there exists a constant $B\geq1$ such that $\max_{i,j}\|X_{ij}/\sigma_j\|_{\psi_1}\leq B$ and $\max_jn^{-1}\sum_{i=1}^n\E[(X_{ij}/\sigma_j)^4]\leq B^2$. 
Then there exists a universal constant $C>0$ such that
\ben{\label{eq:nondeg}
\sup_{A\in\mcl R}|P(S_n\in A)-P(Z\in A)|\leq \frac{C}{\rho_*^2}\sqrt{\frac{B^2\log^3(dn)}{n}}\log n,
}
where $Z\sim N(0,\Cov[S_n])$ and $\rho_*^2$ is the minimum eigenvalue of the correlation matrix of $S_n$. 
\end{theorem}

\begin{proof}
As announced, the proof is a combination of \cite[Theorem 2.1]{CCK23} and a simple truncation argument used in \cite[Section 5.2]{Ko21} and \cite[Section 4.3]{FKLZ23}. 
Denote by $\delta$ the left hand side of \eqref{eq:nondeg}. 
Considering $X_{ij}/\sigma_j$ instead of $X_{ij}$, we may assume $\sigma_j=1$ for all $j$ without loss of generality. 
Further, since $\delta\leq1$, we may also assume
\begin{equation}\label{wlog-nondeg}
\frac{1}{\rho_*^2}\sqrt{\frac{B^2\log^3(dn)}{n}}\log n\leq1.
\end{equation} 
Next, let $\kappa_n:=2B\log n$. For $i=1,\dots,n$ and $j=1,\dots,d$, define $\hat X_{ij}:=X_{ij}1_{\{|X_{ij}|\leq\kappa_n\}}-\E[X_{ij}1_{\{|X_{ij}|\leq\kappa_n\}}]$ and set $\hat X_i=(\hat X_{i1},\dots,\hat X_{id})^\top$ and $\hat S_n:=n^{-1/2}\sum_{i=1}^n\hat X_i$. 
Note that $\max_{1\leq i\leq n}\|\hat X_{i}\|_\infty\leq2\kappa_n$. 
Then, by a similar argument to the proof of Eq.(4.19) in \cite{FKLZ23}, we obtain
\[
\delta\lesssim \frac{1}{n}+\frac{B\log(dn)\sqrt{\log d}}{\sqrt n}\log n+\hat\delta,
\]
where $\hat\delta:=\sup_{A\in\mcl R}|P(\hat S_n\in A)-P(Z\in A)|$. 
Since $\rho_*^2\leq\sigma_1^2=1$, it remains to prove
\ben{\label{nondeg-aim}
\hat\delta\lesssim \frac{1}{\rho_*^2}\sqrt{\frac{B^2\log^3(dn)}{n}}\log n.
}
We prove this bound by applying Theorem 2.1 in \cite{CCK23} with $\psi=2\kappa_n$. This gives
\ben{\label{cck23-apply}
\hat\delta\lesssim (\log n)\left(\Delta_0+\sqrt{\frac{B^2\log^3d}{n\rho_*^4}}+\frac{\kappa_n^2\log^2d}{n\rho_*^2}\right)+\frac{\kappa_n\log^{3/2}d}{\rho_*\sqrt n},
}
where $\Delta_0:=\frac{\log d}{\rho_*^2}\|\Cov(\hat S_n)-\Cov(S_n)\|_\infty$. 
By \eqref{wlog-nondeg} and $\rho_*\leq1$, 
\[
\frac{\kappa_n^2\log^2d}{n\rho_*^2}=\frac{4}{\rho_*^2}\sqrt{\frac{B^2\log^3d}{n}}\sqrt{\frac{B^2(\log d)(\log^2n)}{n}}\log n\leq\frac{4}{\rho_*^2}\sqrt{\frac{B^2\log^3(dn)}{n}}
\]
and
\[
\frac{\kappa_n\log^{3/2}d}{\rho_*\sqrt n}\leq\frac{2}{\rho_*^2}\sqrt{\frac{B^2\log^3(dn)}{n}}\log n.
\]
Further, by Eq.(26) in \cite{Ko21} and \eqref{wlog-nondeg}, 
\begin{equation*}
\Delta_0\lesssim \frac{\log d}{\rho_*^2}e^{-\kappa_n/(2B)}B^2\log n
=\frac{1}{\rho_*^2}\frac{B^2\log d}{n}\log n
\leq \frac{1}{\rho_*^2}\sqrt{\frac{B^2\log^3(dn)}{n}}.
\end{equation*}
Consequently, we obtain \eqref{nondeg-aim} from \eqref{cck23-apply}. 
\end{proof}

\section{Proofs of the auxiliary results in Section \ref{sec:ae-general}}

\subsection{Proof of Lemma \ref{lem:anti}}\label{sec:anti}

We divide the proof into four steps. 
\smallskip

\noindent\textbf{Step 1}. First we reduce the proof to the case $\Sigma=I_d$. 
Let $Z\sim N(0,\Sigma)$ and $Z'\sim N(0,\Sigma-\sigma_*^2I_d)$. 
Then, for any $A\in\mcl R$, $u,v\in\mathbb R^d_+$ and $x\in\mathbb R^d$, we have
\ba{
\int_{A^{u,v}\setminus A}\phi_\Sigma(x+z)dz
&=\E[1_{A^{u,v}\setminus A}(Z-x)]
=\E\sbra{\int_{\mathbb R^d}1_{A^{u,v}\setminus A}(\sigma_*z+Z')\phi_d(z+x/\sigma_*)dz}.
}
Differentiating both sides $r$ times with respect to $x$ and then setting $x=0$, we obtain
\[
\int_{A^{u,v}\setminus A}\nabla^r\phi_\Sigma(z)dz=\frac{1}{\sigma_*^r}\E\sbra{\int_{\mathbb R^d}1_{A^{u,v}\setminus A}(\sigma_*z+Z')\phi_d(z)dz}.
\]
Observe that $1_{A^{u,v}\setminus A}(\sigma_*z+Z')=1_{(\sigma_*^{-1}(A-Z'))^{u/\sigma_*,v/\sigma_*}\setminus (\sigma_*^{-1}(A-Z'))}(z)$ and $\sigma_*^{-1}(A-Z')\in\mcl R$. 
Hence
\ba{
\norm{\int_{A^{u,v}\setminus A}\nabla^r\phi_\Sigma(z)dz}_1
\leq\frac{1}{\sigma_*^r}\sup_{A\in\mcl R}\norm{\int_{A^{u/\sigma_*,v/\sigma_*}\setminus A}\nabla^r\phi_d(z)dz}_1.
}
Therefore, the claim for general $\Sigma$ follows from that for $\Sigma=I_d$. 
\smallskip

\noindent\textbf{Step 2}. 
In this and the next steps, we show that the quantity inside $\sup_{A\in\mcl R}$ on the left hand side of \eqref{eq:anti} can be replaced by a weighted surface integral of $\nabla^r\phi_\Sigma$ over the boundary of $A$. Note that an analogous result for the case $r=0$ is standard in the literature; see e.g.~Proposition 1.1 in \cite{Ra19}. 
For $A\in\mcl R$, $u,v\in\mathbb R^d_+$ and $\eps>0$, set
\[
I_{A}(u,v)=\norm{\int_{A^{u,v}\setminus A}\nabla^r\phi_d(z) dz}_1,\qquad
K(\eps)=\sup_{A\in\mcl R;u,v\in [0,\eps]^d}\frac{I_{A}(u,v)}{\eps}.
\]
In this step, we prove
\[
\sup_{\eps>0}K(\eps)=\limsup_{\eps\downarrow0}K(\eps).
\]
Take $\eps>0$ arbitrarily. For any $A\in\mcl R$ and $u,v\in\mathbb R^d_+$, observe that $A^{u/2,v/2}\in\mcl R$ and $A^{u,v}\setminus A$ is the disjoint union of $(A^{u/2,v/2})^{u/2,v/2}\setminus (A^{u/2,v/2})$ and $A^{u/2,v/2}\setminus A$. This implies that 
\[
\sup_{A\in\mcl R;u,v\in [0,\eps]^d}I_{A}(u,v)\leq2\sup_{A\in\mcl R;u,v\in [0,\eps/2]^d}I_A(u,v),
\] 
and thus $K(\eps)\leq K(\eps/2)$. Repeating this procedure gives $K(\eps)\leq K(\eps/2^n)$ for $n=1,2,\dots$. Hence
\[
K(\eps)\leq\limsup_{n\to\infty}K(\eps/2^n)\leq\limsup_{\eta\downarrow0}K(\eta).
\]
Since $\eps$ is arbitrary, we obtain the desired result. 
\smallskip

\noindent\textbf{Step 3}. 
For any Borel set $A\subset\mathbb R^{d-1}$, $j\in\{1,\dots,d\}$ and $s\in\mathbb R$, define
\[
J_{A,j}(s)=\int_{A}\nabla^r\phi_d(z|_{z_j=s}) dz_1\cdots \wh{dz_{j}}\cdots dz_d,
\]
where $z|_{z_j=s}=(z_1,\dots,z_{j-1},s,z_{j+1},\dots,z_d)^\top$ and $\wh{dz_j}$ means that $dz_j$ is omitted. 
Then, for $u,v\in \mathbb R^d$ and $A=\prod_{j=1}^d[a_j,b_j]\in\mcl R$, set
\[
L_A(u,v)=\sum_{j=1}^d\{u_jJ_{A^j,j}(a_j)+v_jJ_{A^j,j}(b_j)\},
\]
where $A^j=\prod_{k:k\neq j}[a_k,b_k]$. 
In this step, we prove
\[
\limsup_{\eps\downarrow0}K(\eps)=\limsup_{\eps\downarrow0}\sup_{A\in\mcl R;u,v\in [0,\eps]^d}\frac{\|L_A(u,v)\|_1}{\eps}.
\]
Fix $A=\prod_{j=1}^d[a_j,b_j]\in\mcl R$, $\eps>0$ and $u,v\in[0,\eps]^d$. For $j_1,\dots,j_r\in\{1,\dots,d\}$, we set
\[
A^{u,v}_{j_1,\dots,j_r}=\{x\in A^{u,v}:x_j\notin[a_j,b_j]\text{ for }j\in\{j_1,\dots,j_r\}\text{ and }x_j\in[a_j,b_j]\text{ for }j\notin\{j_1,\dots,j_r\}\}.
\]
Then, $A^{u,v}\setminus A=\bigcup_{r=1}^d\bigcup_{1\leq j_1<\cdots<j_r\leq d}A^{u,v}_{j_1,\dots,j_r}$ and this is a disjoint union. Hence we have
\[
I_{A}(u,v)=\norm{\sum_{r=1}^d\sum_{1\leq j_1<\cdots<j_r\leq d}\int_{A^{u,v}_{j_1,\dots,j_r}}\nabla^r\phi_d(z) dz}_1.
\]
One can easily check that
\[
\sup_{A\in\mcl R;u,v\in [0,\eps]^d}\sum_{r=2}^d\sum_{1\leq j_1<\cdots<j_r\leq d}\int_{A^{u,v}_{j_1,\dots,j_r}}\norm{\nabla^r\phi_d(z)}_1 dz=O(\eps^2)\quad\text{as }\eps\downarrow0.
\]
Hence we obtain
\ba{
\limsup_{\eps\downarrow0}K(\eps)
=\limsup_{\eps\downarrow0}\sup_{A\in\mcl R;u,v\in [0,\eps]^d}\norm{\frac{1}{\eps}\sum_{j=1}^d\int_{A^{u,v}_{j}}\nabla^r\phi_d(z) dz}_1.
}
For any $j\in\{1,\dots,d\}$, observe that $A_j^{u,v}$ is the disjoint union of 
\[
\ul{A}_j^{u}:=[a_1,b_1]\times\cdots[a_{j-1},b_{j-1}]\times[a_j-u_j,a_j)\times[a_{j+1},b_{j+1}]\times\cdots\times[a_d,b_d]
\]
and
\[
\ol{A}_j^v:=[a_1,b_1]\times\cdots[a_{j-1},b_{j-1}]\times(b_j,b_j+v_j]\times[a_{j+1},b_{j+1}]\times\cdots\times[a_d,b_d].
\]
Therefore, we have
\ba{
&\norm{\sum_{j=1}^d\int_{A^{u,v}_{j}}\nabla^r\phi_d(z) dz-L_A(u,v)}_1\\
&=\norm{\sum_{j=1}^d\bra{
\int_{\ul{A}^u_{j}}\{\nabla^r\phi_d(z)-\nabla^r\phi_d(z|_{z_j=a_j})\} dz
+\int_{\ol{A}^v_{j}}\{\nabla^r\phi_d(z)-\nabla^r\phi_d(z|_{z_j=b_j})\} dz
}}_1\\
&\leq2\eps\sum_{j=1}^d\int_{\mathbb R^{d-1}}\sup_{s,t\in\mathbb R:|s-t|\leq\eps}\norm{\nabla^r\phi_d(z|_{z_j=s})-\nabla^r\phi_d(z|_{z_j=t})}_1dz_1\cdots\wh{dz_j}\cdots dz_d.
}
Thus,
\[
\limsup_{\eps\downarrow0}\sup_{A\in\mcl R;u,v\in [0,\eps]^d}\frac{1}{\eps}\norm{\sum_{j=1}^d\int_{A^{u,v}_{j}}\nabla^r\phi_d(z) dz-L_A(u,v)}_1=0.
\]
This gives the desired result. 
\smallskip

\noindent\textbf{Step 4}. 
It remains to prove
\ben{\label{step4-aim}
\limsup_{\eps\downarrow0}\sup_{A\in\mcl R;u,v\in R_+(\eps)}\frac{\|L_A(u,v)\|_1}{\eps}\leq C_r(\log d)^{(r+1)/2}.
}
We first note that if $A$ is an orthant, i.e.~$a_j=-\infty$ for all $j$, then \eqref{step4-aim} immediately follows from the fundamental theorem of calculus and \cref{lem:aht}. 
In fact, we have in this case
\ba{
\frac{\|L_A(u,v)\|_1}{\eps}
\leq\norm{\int_A\nabla^{r+1}\phi_d(z)dz}_1
}
for any $\eps>0$ and $u,v\in [0,\eps]^d$. 
In the following we show that the proof is essentially reduced to this case by a similar argument to the proof of \cite[Lemma 2.2]{FaKo21}.   
For every $q\in\{1,\dots,r\}$, set 
\ba{
\mathcal{N}_q(r)=\{(\nu_1,\dots,\nu_q)\in\mathbb{Z}^q:\nu_1,\dots,\nu_q\geq0,\nu_1+\cdots+\nu_q=r\}.
}
Also, for any $m\in\mathbb N$, let
\[
\mathcal{J}_m(d)=\{(j_1,\dots,j_m)\in\{1,\dots,d\}^m:j_1,\dots,j_m\text{ are distinct}\}.
\]
Then, for any $A\in\mcl R$, $j\in\{1,\dots,d\}$ and $s\in\mathbb R$, we have
\ba{
\norm{J_{A^j,j}(s)}_1
&=\sum_{j_1,\dots,j_r=1}^d\abs{\int_{A^j}\partial_{j_1,\dots,j_r}\phi_d(z|_{z_j=s}) dz_1\cdots\wh{dz_j}\cdots dz_d}\\
&\leq C_r\sum_{q=1}^r\sum_{(\nu_1,\dots,\nu_q)\in\mathcal{N}_q(r)}\sum_{(j_1,\dots,j_q)\in\mathcal{J}_q(d)}\abs{\int_{A^j}\partial_{j_1}^{\nu_1}\cdots\partial_{j_q}^{\nu_q}\phi_d(z|_{z_j=s}) dz_1\cdots\wh{dz_j}\cdots dz_d}\\
&=:C_r\sum_{q=1}^r\sum_{\nu=(\nu_1,\dots,\nu_q)\in\mathcal{N}_q(r)}\sum_{\bs j=(j_1,\dots,j_q)\in\mathcal{J}_q(d)}\Lambda_j(\nu,\bs j).
}
For each $r=1,\dots,q$, the cardinality of the set $\mathcal{N}_q(r)$ is bounded by a constant depending only on $r$. Therefore, to prove \eqref{step4-aim}, it suffices to show that
\ben{\label{step4-aim2}
\sum_{j=1}^d\sum_{\bs j=(j_1,\dots,j_r)\in\mathcal{J}_q(d)}\Lambda_j(\nu,\bs j)
\leq C_r(\log d)^{(r+1)/2}
}
for any (fixed) $q\in\{1,\dots,r\}$, $\nu=(\nu_1,\dots,\nu_q)\in\mathcal{N}_q(r)$, $A=\prod_{j=1}^d[a_j,b_j]\in\mcl R$ and $s\in\{a_j,b_j\}$, $j=1,\dots,d$. 

To prove \eqref{step4-aim2}, we introduce additional notation. 
For a non-negative integer $m$, $H_m$ denotes the $m$-th Hermite polynomial, i.e.~$H_m(t)=(-1)^m\phi(t)^{-1}\phi^{(m)}(t)$. 
When $m\geq1$, we set $h_m(t)=H_{m-1}(t)\phi(t)$. 
Also, we denote by $t_m$ the maximum root of $H_m$. For example, $t_1=0,t_2=1,t_3=\sqrt{3}$. 
Finally, set $M_{m}:=\max_{0\leq t\leq t_m}|H_{m-1}(t)|<\infty$ and define
\[
\tilde h_m(t)=M_{m}\phi(t)1_{[0,t_m]}(t)+h_{m}(t)1_{(t_m,\infty)}(t).
\]
The function $\tilde h_m$ satisfies the following properties by Lemma A.1 in \cite{FaKo21}:
\begin{align}
&\tilde h_m \text{ is decreasing on }[0,\infty).\label{h-decrease}\\
&|h_m(t)|\leq\tilde h_m(|t|)\text{ for all }t\in\mathbb R.\label{h-bound}
\end{align}

Now, we fix $\bs j=(j_1,\dots,j_q)\in\mcl J_q(d)$ and $j\in\{1,\dots,d\}$ for a while. Set 
\[
\nu=\begin{cases}
\nu_p & \text{if }j=j_p\text{ for some }p\in\{1,\dots,q\},\\
0 & \text{otherwise}.
\end{cases}
\]
Then we have
\ba{
\Lambda_j(\nu,\bs j)
&=|h_{\nu+1}(s)|\left(\prod_{p:j_p\neq j}\left|h_{\nu_p}(b_{j_p})-h_{\nu_p}(a_{j_p})\right|\right)
\prod_{k:k\neq j_1,\dots,j_q,j}\left\{\Phi(b_k)-\Phi(a_k)\right\}\\
&\leq\tilde h_{\nu+1}(|s|)\left(\prod_{p:j_p\neq j}\left(\tilde h_{\nu_p}(|b_{j_p}|)+\tilde h_{\nu_p}(|a_{j_p}|)\right)\right)
\prod_{k:k\neq j_1,\dots,j_q,j}\left\{\Phi(b_k)+\Phi(-a_k)-1\right\},
}
where the last inequality follows from \eqref{h-bound} and the identity $1-\Phi(t)=\Phi(-t)$. 
Set $c_k=|a_k|\wedge|b_k|$ for $k=1,\dots,d$. Then we have
$
\Phi(b_k)+\Phi(-a_k)-1
\leq\min\{\Phi(b_k),\Phi(-a_k)\}
\leq\Phi(c_k).
$
Combining this with \eqref{h-decrease} gives
\ba{
\Lambda_j(\nu,\bs j)
&\leq2^q\tilde h_{\nu+1}(c_j)\left(\prod_{p:j_p\neq j}\tilde h_{\nu_p}(c_{j_p})\right)
\prod_{k:k\neq j_1,\dots,j_q,j}\Phi(c_k).
}
Now, observe that $\tilde h_m(t)\leq C_m(1+t^{m-1})\phi(t)$ for any $m\in\mathbb N$ and $t\geq0$ by construction. 
Hence, if $\max_{p=1,\dots,q}c_{j_p}\leq\sqrt{4(r+1)\log d}$, 
\ba{
\Lambda_j(\nu,\bs j)
&\leq \begin{cases}
C_r(\log d)^{(r-q+1)/2}\left(\prod_{p=1}^q\phi(c_{j_p})\right)\prod_{k:k\neq j_1,\dots,j_q}\Phi(c_k) & \text{if }j\in\{j_1,\dots,j_q\},\\
C_r(\log d)^{(r-q)/2}\phi(c_j)\left(\prod_{p=1}^q\phi(c_{j_p})\right)\prod_{k:k\neq j_1,\dots,j_q,j}\Phi(c_k) & \text{otherwise},
\end{cases}
}
where we used the identity $\sum_{p=1}^q\nu_p=r$. 
Besides, if $\max_{p=1,\dots,q}c_{j_p}>\sqrt{4(r+1)\log d}$, 
\ba{
\Lambda_j(\nu,\bs j)
\leq C_r\prod_{p=1}^qe^{-c_{j_p}^2/4}\leq C_rd^{-r-1}.
}
Consequently, 
\ba{
&\sum_{j=1}^d\sum_{\bs j=(j_1,\dots,j_q)\in\mcl J_q(d)}\Lambda_j(\nu,\bs j)\\
&\leq C_r+C_r\sum_{(j_1,\dots,j_q)\in\mcl J_q(d)}(\log d)^{(r-q+1)/2}\left(\prod_{p=1}^q\phi(c_{j_p})\right)\prod_{k:k\neq j_1,\dots,j_q}\Phi(c_k)\\
&\quad+C_r\sum_{(j_1,\dots,j_q)\in\mcl J_q(d)}\sum_{j:j\neq j_1,\dots,j_q}(\log d)^{(r-q)/2}\phi(c_j)\left(\prod_{p=1}^q\phi(c_{j_p})\right)\prod_{k:k\neq j_1,\dots,j_q,j}\Phi(c_k).
}
With $A'=\prod_{j=1}^d(-\infty,c_j]$, we can rewrite the right hand side of the above inequality as
\bm{
C_r+C_r(\log d)^{(r-q+1)/2}\sum_{(j_1,\dots,j_q)\in\mcl J_q(d)}\int_{A'}\partial_{j_1,\dots,j_q}\phi_d(z)dz\\
+C_r(\log d)^{(r-q)/2}\sum_{(j_1,\dots,j_{q+1})\in\mcl J_{q+1}(d)}\int_{A'}\partial_{j_1,\dots,j_{q+1}}\phi_d(z)dz.
}
This quantity is bounded by
\ben{\label{step4-reduced}
C_r\bra{1+(\log d)^{(r-q+1)/2}\norm{\int_{A'}\nabla^q\phi_d(z)dz}_1
+(\log d)^{(r-q)/2}\norm{\int_{A'}\nabla^{q+1}\phi_d(z)dz}_1}.
}
For any $m\in\mathbb N$, observe that
\ba{
\norm{\int_{A'}\nabla^m\phi_d(z)dz}_1
=\lim_{a\to-\infty}\norm{\int_{\prod_{j=1}^d[a,c_j]}\nabla^m\phi_d(z)dz}_1
\leq\sup_{A\in\mcl R}\norm{\int_{A}\nabla^m\phi_d(z)dz}_1.
}
Therefore, by \cref{lem:aht}, the quantity in \eqref{step4-reduced} is bounded by $C_r(\log d)^{(r+1)/2}$. This gives \eqref{step4-aim2}. \qed

\subsection{Proof of Lemma \ref{lem:decomp}}\label{sec:decomp}

The proof of \eqref{decomp-aim} is an almost straightforward multi-dimensional extension of that of \cite[Lemma 2.1]{FaLi22}, and the proof of \eqref{decomp-aim-gcomp} is its simplification. 
Let $Z\sim N(0,\Sigma)$ be independent of $(\xi_i)_{i=1}^n$. 
Set $W(s)=\sqrt{1-s}W+\sqrt sZ$ for every $s\in[0,1]$. 
The following two lemmas summarize the role of the approximate Stein kernel in the proof. 
\begin{lemma}\label{lem:sk-arg}
Under the assumptions of \cref{lem:decomp}, suppose additionally that $h\in C^r_b(\mathbb R^d)$. Then, for any $t\in[0,1]$, we have the following:
\begin{enumerate}[label=(\alph*)]

\item If $U\in(\mathbb R^d)^{\otimes r}$ is non-random,
\ba{
&\E[\langle U,\nabla^r h(W(t))\rangle]-\E[\langle U,\nabla^r h(Z)\rangle]\\
&=\frac{1}{2}\int_t^1\bra{\E[\langle U\otimes \bar T,\nabla^{r+2} h(W(s))\rangle]
+\frac{\E[\langle U\otimes B,\nabla^{r+1} h(W(s))\rangle]}{\sqrt{1-s}}}ds.
}

\item For every $i=1,\dots,n$, let $v_i:\mathbb R^d\to(\mathbb R^d)^{\otimes r}$ be a measurable function such that $v_i(\xi_i),v_i(\xi_i)\otimes\xi_i,v_i(\xi_i)\otimes\tau_i(\xi_i)$ and $v_i(\xi_i)\otimes\res_i(\xi_i)$ are integrable. 
Set $V:=\sum_{i=1}^nv_i(\xi_i)$. Then 
\ba{
&\E[\langle V,\nabla^r h(W(t))\rangle]-\E[\langle V,\nabla^r h(Z)\rangle]\\
&=\frac{1}{2}\int_t^1\bra{\E[\langle V\otimes \bar T,\nabla^{r+2} h(W(s))\rangle]
-\sum_{i=1}^n\E[\langle v_i(\xi_i)\otimes\tau_i(\xi_i),\nabla^{r+2} h(W(s))\rangle]}ds\\
&\quad+\frac{1}{2}\int_t^1\frac{1}{\sqrt{1-s}}\bra{\E[\langle V\otimes \Res,\nabla^{r+1} h(W(s))\rangle]
+\sum_{i=1}^n\E[\langle v_i(\xi_i)\otimes(\xi_i-\res_i(\xi_i)),\nabla^{r+1} h(W(s))\rangle]}ds.
}
\end{enumerate}
\end{lemma}

\begin{proof}
Since the proof of (a) is similar to that of (b) and simpler, we prove (b) only. 

By the fundamental theorem of calculus, we obtain
\ban{
&\E[\langle V,\nabla^r h(W(t))\rangle]-\E[\langle V,\nabla^r h(Z)\rangle]\notag\\
&=-\int_t^1\frac{\partial}{\partial s}\E [\langle V,\nabla^r h(W(s))]ds\notag\\
&=\frac{1}{2}\int_t^1\bra{\frac{\E[\langle V\otimes W,\nabla^{r+1} h(W(s))\rangle]}{\sqrt{1-s}}-\frac{\E[\langle V\otimes Z,\nabla^{r+1} h(W(s))\rangle]}{\sqrt s}}ds.\label{foc}
}
Since $Z$ is independent of $(W,V)$, an application of the multivariate Stein identity conditional on $(W,V)$ gives for any $0<s<1$
\ben{\label{g-ibp}
\frac{\E[\langle V\otimes Z,\nabla^{r+1} h(W(s))\rangle]}{\sqrt s}=\E[\langle V\otimes\Sigma, \nabla^{r+2} h(W(s))\rangle].
}
Meanwhile, we rewrite $\E[\langle V\otimes W,\nabla^3 h(W(s))\rangle]$ as
\ban{
&\E[\langle V\otimes W,\nabla^{r+1} h(W(s))\rangle]\notag\\
&=\sum_{i=1}^n\E[\langle V\otimes \xi_i,\nabla^{r+1} h(W(s))\rangle]\notag\\
&=\sum_{i=1}^n\E[\langle V^{(i)}\otimes \xi_i,\nabla^{r+1} h(W(s))\rangle]
+\sum_{i=1}^n\E[\langle v_i(\xi_i)\otimes \xi_i,\nabla^{r+1} h(W(s))\rangle],\label{vw-decomp}
}
where $V^{(i)}=V-v_i(\xi_i)$. 
For $i=1,\dots,n$, since $(\tau_i,\res_i)$ is an approximate Stein kernel for $\xi_i$ and $\xi_i$ is independent of $V^{(i)}$ and $W-\xi_i$, we have
\ba{
&\E[\langle V^{(i)}\otimes \xi_i,\nabla^{r+1} h(W(s))\rangle]\\
&=\sqrt{1-s}\E[\langle V^{(i)}\otimes \tau_i(\xi_i),\nabla^{r+2} h(W(s))\rangle]
+\E[\langle V^{(i)}\otimes \res_i(\xi_i),\nabla^{r+1} h(W(s))\rangle].
}
Hence
\ban{
&\sum_{i=1}^n\E[\langle V^{(i)}\otimes \xi_i,\nabla^{r+1} h(W(s))\rangle]\notag\\
&=\sqrt{1-s}\bra{\sum_{i=1}^n\E[\langle V\otimes \tau_i(\xi_i),\nabla^{r+2} h(W(s))\rangle]
-\sum_{i=1}^n\E[\langle v_i(\xi_i)\otimes\tau_i(\xi_i),\nabla^{r+2} h(W(s))\rangle]}\notag\\
&\quad+\sum_{i=1}^n\E[\langle V\otimes \res_i(\xi_i),\nabla^{r+1} h(W(s))\rangle]
-\sum_{i=1}^n\E[\langle v_i(\xi_i)\otimes\res_i(\xi_i),\nabla^{r+1} h(W(s))\rangle]\notag\\
&=\sqrt{1-s}\bra{\E[\langle V\otimes T,\nabla^{r+2} h(W(s))\rangle]
-\sum_{i=1}^n\E[\langle v_i(\xi_i)\otimes\tau_i(\xi_i),\nabla^{r+2} h(W(s))\rangle]}\notag\\
&\quad+\E[\langle V\otimes \Res,\nabla^{r+1} h(W(s))\rangle]
-\sum_{i=1}^n\E[\langle v_i(\xi_i)\otimes\res_i(\xi_i),\nabla^{r+1} h(W(s))\rangle].\label{vi-decomp}
}
Inserting \eqref{g-ibp}--\eqref{vi-decomp} into \eqref{foc} gives the desired result. 
\end{proof}

\begin{lemma}\label{lem:tau-xi}
Let $\xi$ be a centered random vector in $\mathbb R^d$. Suppose that $\xi$ has a Stein kernel $\tau$ such that $\E[\|\xi\|_\infty^3]+\E[\|\tau(\xi)\otimes\xi^{\otimes2}\|_\infty]<\infty$. Then, $\E[\|\tau(\xi)\otimes\xi\|_\infty]<\infty$ and for any $f\in C^4_b(\mathbb R^d)$,
\ben{\label{eq:tau-xi}
\E[\langle\tau(\xi)\otimes\xi,\nabla^3 f(\xi)\rangle]=\frac{1}{2}\bra{\E[\langle \xi^{\otimes3},\nabla^3f(\xi)\rangle]-\E[\langle\tau(\xi)\otimes\xi^{\otimes2},\nabla^4f(\xi)\rangle]}.
}
\end{lemma}

\begin{proof}
Take $j,k,l\in\{1,\dots,d\}$ arbitrarily. Then we have
\ba{
\E[|\tau_{jk}(\xi)\xi_l|]
&=\E[|\tau_{jk}(\xi)\xi_l|;|\xi_l|\leq1]+\E[|\tau_{jk}(\xi)\xi_l|;|\xi_l|>1]\\
&\leq\E[|\tau_{jk}(\xi)|]+\E[|\tau_{jk}(\xi)|\xi_l^2]
\leq\E[\|\tau(\xi)\|_\infty]+\E[\|\tau(\xi)\otimes\xi^{\otimes2}\|_\infty]<\infty.
}
Hence $\E[\|\tau(\xi)\otimes\xi\|_\infty]<\infty$. 

Next we prove \eqref{eq:tau-xi}. By a standard approximation argument, it suffices to prove \eqref{eq:tau-xi} when $f$ is compactly supported. 
For every $j=1,\dots,d$, define a function $g_j:\mathbb R^d\to\mathbb R$ as
\[
g_j(x)=\langle x^{\otimes2},\nabla^2\partial_jf(x)\rangle=\sum_{u,v=1}^dx_ux_v\partial_{juv}f(x),\qquad x\in\mathbb R^d.
\]
Observe that $g_j$ is compactly supported because $f$ is compactly supported. 
For $j,k\in\{1,\dots,d\}$ and $x\in\mathbb R^d$, we have
\ba{
\partial_k g_j(x)&=2\sum_{v=1}^dx_v\partial_{jkv}f(x)+\sum_{u,v=1}^dx_ux_v\partial_{jkuv}f(x).
}
Hence we obtain
\ba{
\E[\langle\tau(\xi)\otimes\xi,\nabla^3 f(\xi)\rangle]
&=\sum_{j,k,v=1}^d\E[\tau_{jk}(\xi)\xi_v\partial_{jkv} f(\xi)]\\
&=\frac{1}{2}\sum_{j,k=1}^d\E[\tau_{jk}(\xi)\partial_{k} g_j(\xi)]
-\frac{1}{2}\sum_{j,k,u,v=1}^d\E[\tau_{jk}(\xi)\xi_u\xi_v\partial_{jkuv}f(\xi)].
}
The second term on the last line is equal to $\frac{1}{2}\E[\langle\tau(\xi)\otimes\xi^{\otimes2},\nabla^4f(\xi)\rangle]$. 
To evaluate the first term, define a function $G:\mathbb R^d\to\mathbb R$ as $G(x)=\sum_{k=1}^dx_k\int_0^1g_k(\theta x)d\theta$, $x\in\mathbb R^d$. Then, using the relation $\partial_kg_j=\partial_jg_k$, one can easily verify $\partial_jG=g_j$ for all $j=1,\dots,d$. 
Moreover, $G$ is bounded. In fact, since $g_1,\dots,g_d$ are compactly supported, there exists a constant $K>0$ such that $g_j(x)=0$ for all $j$ whenever $\|x\|_\infty>K$. Then, for any $x\in\mathbb R^d$ with $\|x\|_\infty>K$, 
\ba{
|G(x)|\leq\sum_{k=1}^d|x_k|\int_0^{K/\|x\|_\infty}|g_k(\theta x)|d\theta
\leq dK\max_{1\leq k\leq d}\sup_{x\in\mathbb R^d}|g_k(x)|.
}
It is straightforward to check that this estimate is still true when $\|x\|_\infty\leq K$. Hence $G$ is bounded.  
As a result, $G\in C^2_b(\mathbb R^d)$ and thus
\ba{
\sum_{j,k=1}^d\E[\tau_{jk}(\xi)\partial_{k} g_j(\xi)]
&=\sum_{j,k=1}^d\E[\tau_{jk}(\xi)\partial_{kj} G(\xi)]
=\sum_{j=1}^d\E[\xi_j\partial_jG(\xi)]
=\sum_{j=1}^d\E[\xi_jg_j(\xi)]\\
&=\E[\langle \xi^{\otimes3},\nabla^3f(\xi)\rangle],
}
where the second equality follows from the definition of Stein kernel. 
Combining these identities gives the desired result. 
\end{proof}


\begin{proof}[Proof of \cref{lem:decomp}]
First we prove \eqref{decomp-aim}. 
It suffices to prove the claim when $h\in C^\infty_b(\mathbb R^d)$. To see this, let $t_1=t/2$ and $t_2=t/(2-t)$. One can easily check that $t_1,t_2\in(0,1]$ and $h_t=(h_{t_1})_{t_2}$. Since $h_{t_1}\in C^\infty_b(\mathbb R^d)$, the general case follows by applying the claim to $h=h_{t_1}$ and $t=t_2$. 

Observe that $\E[h_t(W)]=\E[h(W(t))]$. 
Therefore, \cref{lem:sk-arg}(a) gives
\ben{\label{decomp1}
\E [h_t(W)]-\E [h(Z)]
=\frac{1}{2}\int_t^1\bra{\E[\langle \bar T, \nabla^2 h(W(s))\rangle]+\frac{\E[\Res\cdot\nabla h(W(s))]}{\sqrt{1-s}}}ds.
}
Next, for $t<s<1$, \cref{lem:sk-arg} also gives
\besn{\label{decomp11}
&\E[\langle \bar T, \nabla^2 h(W(s))\rangle]-\E[\langle \bar T,\nabla^2 h(Z)\rangle]\\
&=\frac{1}{2}\int_s^1\bra{\E[\langle \bar T^{\otimes2},\nabla^{4} h(W(u))\rangle]
-\sum_{i=1}^n\E[\langle \tau_i(\xi_i)^{\otimes2},\nabla^{4} h(W(u))\rangle]}du\\
&\quad+\frac{1}{2}\int_s^1\frac{1}{\sqrt{1-u}}\bra{\E[\langle \bar T\otimes \Res,\nabla^{3} h(W(u))\rangle]
+\sum_{i=1}^n\E[\langle \tau_i(\xi_i)\otimes(\xi_i-\res_i(\xi_i)),\nabla^{3} h(W(u))\rangle]}du
}
and
\besn{\label{decomp12}
&\E[\Res\cdot\nabla h(W(s))]-\E[\Res\cdot\nabla h(Z)]\\
&=\frac{1}{2}\int_s^1\bra{\E[\langle \Res\otimes\bar T,\nabla^{3} h(W(u))\rangle]
-\sum_{i=1}^n\E[\langle \res_i(\xi_i)\otimes\tau_i(\xi_i),\nabla^{3} h(W(u))\rangle]}du\\
&\quad+\frac{1}{2}\int_s^1\frac{1}{\sqrt{1-u}}\bra{\E[\langle \Res^{\otimes2},\nabla^{2} h(W(u))\rangle]
+\sum_{i=1}^n\E[\langle \res_i(\xi_i)\otimes(\xi_i-\res_i(\xi_i)),\nabla^{2} h(W(u))\rangle]}du.
}
Using the definition of the approximate Stein kernel, we can easily verify that $\beta_i(\xi_i)$ are centered. Hence
\ben{
\E[\Res\cdot\nabla h(Z)]=\E[\Res]\cdot\E[\nabla h(Z)]=0.
}
Meanwhile, for $s<u<1$, \cref{lem:tau-xi} gives 
\besn{\label{apply:tau-xi}
&\sum_{i=1}^n\E[\langle \tau_i(\xi_i)\otimes \xi_i,\nabla^3 h(W(u))\rangle]\\
&=\frac{1}{2}\sum_{i=1}^n\bra{\E[\langle \xi_i^{\otimes3}-\res_i(\xi_i)\otimes\xi_i^{\otimes2},\nabla^3 h(W(u))\rangle]
-\sqrt{1-u}\E[\langle \tau_i(\xi_i)\otimes \xi_i^{\otimes2},\nabla^4 h(W(u))\rangle]}.
}
Then, \cref{lem:sk-arg} again gives
\besn{\label{decomp3}
&\E[\langle \xi_i^{\otimes3},\nabla^3 h(W(u))\rangle]-\E[\langle \xi_i^{\otimes3},\nabla^3 h(Z)\rangle]\\
&=\frac{1}{2}\int_u^1\sum_{i=1}^n\E[\langle \xi_i^{\otimes3}\otimes (\bar T-\tau_i(\xi_i)),\nabla^5 h(W(v))\rangle]dv\\
&\quad+\frac{1}{2}\int_u^1\frac{1}{\sqrt{1-v}}\sum_{i=1}^n\E[\langle \xi_i^{\otimes3}\otimes (\xi_i+\Res-\res_i(\xi_i)),\nabla^{4} h(W(v))\rangle]dv
}
and
\besn{\label{decomp4}
&\sum_{i=1}^n\E[\langle \res_i(\xi_i)\otimes\xi_i,\nabla^{2} h(W(u))\rangle]-\sum_{i=1}^n\E[\langle \langle \res_i(\xi_i)\otimes\xi_i,\nabla^2 h(Z)\rangle]\\
&=\frac{1}{2}\int_u^1\sum_{i=1}^n\E[\langle \res_i(\xi_i)\otimes\xi_i\otimes (\bar T-\tau_i(\xi_i)),\nabla^{4} h(W(v))\rangle]dv\\
&\quad+\frac{1}{2}\int_u^1\frac{1}{\sqrt{1-v}}\sum_{i=1}^n\E[\langle \res_i(\xi_i)\otimes\xi_i\otimes (\xi_i+\Res-\res_i(\xi_i)),\nabla^{3} h(W(v))\rangle]dv.
}
Further, note that we have by integration by parts
\ba{
\int_{\mathbb R^d}h_t(z)\nabla^r\phi_\Sigma(z)dz
&=(-\sqrt{1-t})^{r}\int_{\mathbb R^d}\E[\nabla^r h(\sqrt{1-t}z+\sqrt{t}Z)]\phi_\Sigma(z)dz\\
&=(-\sqrt{1-t})^{r}\E[\nabla^rh(Z)]
}
for any $r\in\mathbb N$. 
Also, from the definition of the approximate Stein kernel, we can easily verify $\langle\E[ \xi_i^{\otimes2}],M\rangle=\langle\E[\tau_i(\xi_i)+\beta_i(\xi_i)\otimes\xi_i],M\rangle$ for any $i=1,\dots,n$ and symmetric $d\times d$ matrix $M$. 
Hence
\ban{
&\frac{1-t}{2}\E[\langle \bar T,\nabla^2 h(Z)\rangle]
+\frac{1}{4}\E\sbra{\left\langle \sum_{i=1}^n\beta_i(\xi_i)\otimes\xi_i,\nabla^2 h(Z)\right\rangle}\int_t^1\frac{1}{\sqrt{1-s}}\bra{\int_s^1\frac{1}{\sqrt{1-u}}du}ds\notag\\
&=\frac{1-t}{2}\langle \Sigma_W-\Sigma,\E[\nabla^2 h(Z)]\rangle
=\frac{1}{2}\int_{\mathbb R^d}h_t(z)\langle\Sigma_W-\Sigma,\nabla^2\phi_\Sigma(z)\rangle dz\label{semi-first-term}
}
and
\ben{\label{first-term}
\frac{\E[\langle \xi_i^{\otimes3},\nabla^3 h(Z)\rangle]}{8}\int_t^1\bra{\int_s^1\frac{1}{\sqrt{1-u}}du}ds
=-\frac{1}{6}\int_{\mathbb R^d}h_t(z)\langle\E[\xi_i^{\otimes3}],\nabla^3\phi_\Sigma(z)\rangle dz.
}
Finally, observe that 
\ben{\label{slepian-deriv}
\nabla^rh_s(W)=(1-s)^{r/2}\E[\nabla^rh(W(s))\mid \xi_1,\dots,\xi_n]
} 
for any $r\in\mathbb N$ and $s\in[0,1]$. 
Combining \eqref{decomp1}--\eqref{slepian-deriv} gives \eqref{decomp-aim}.

Next we prove \eqref{decomp-aim-gcomp}. 
As above, we may assume $h\in C^\infty_b(\mathbb R^d)$.
By \cref{lem:sk-arg}(a),
\[
\E [h_t(W)]-\E [h(Z)]
=\frac{1}{2}\int_t^1\E[\langle \Sigma_W-\Sigma, \nabla^2 h(W(s))\rangle]ds
\]
and
\ba{
\E[\langle \Sigma_W-\Sigma, \nabla^2 h(W(s))\rangle]-\E[\langle \Sigma_W-\Sigma, \nabla^2 h(Z)\rangle]
=\frac{1}{2}\int_s^1\E[\langle (\Sigma_W-\Sigma)^{\otimes2}, \nabla^4 h(W(u))\rangle]du.
}
Combining these two identities with \eqref{slepian-deriv} gives \eqref{decomp-aim-gcomp}. 
\end{proof}

\subsection{Proof of Lemma \ref{l3}}\label{sec:smoothing}

As already mentioned, the proof is a straightforward modification of \cite[Lemma 11.4]{BhRa10}. 
Let 
\[
\delta:=\sup\left\{\left|\int h_yd(\mu-\nu)\right|:y\in\mathbb{R}^d\right\}.
\]
Assume first that
\begin{equation}\label{eq:smooth-1}
\delta=\sup\left\{\int h_yd(\mu-\nu):y\in\mathbb{R}^d\right\}.
\end{equation}
Then, given any $\eta>0$, there exists a vector $z\in\mathbb{R}^d$ such that
$
\int h_{z}d(\mu-\nu)\geq\delta-\eta.
$
In this case, we have
\ba{
&\int_{[-\eps,\eps]^d}\left[\int M_{h_{z}}(y+x; \eps)(\mu-\nu)(dy)\right]K(dx)\\
&\geq\int_{[-\eps,\eps]^d}\left[\int h_{z}(y)\mu(dy)-\int M_{h_{z}}(y+x; \eps)\nu(dy)\right]K(dx)\\
&=\int_{[-\eps,\eps]^d}\left[\int h_{z}(y)(\mu-\nu)(dy)-\int \left\{M_{h_{z}}(y+x; \eps)-h_{z}(y)\right\}\nu(dy)\right]K(dx)\\
&\geq\int_{[-\eps,\eps]^d}\left[\delta-\eta-\int \left\{M_{h_{z}}(y+x; \eps)-h_{z}(y+x)\right\}\nu(dy)-\int \left\{h_{z}(y+x)-h_{z}(y)\right\}\nu(dy)\right]K(dx)\\
&\geq\int_{[-\eps,\eps]^d}\left[\delta-\eta-\tau^*(h;\eps)-\tilde\tau^*(h;\eps)\right]K(dx)
=\alpha\left[\delta-\eta-\tau^*(h;\eps)-\tilde\tau^*(h;\eps)\right]
}
and
\ba{
&\int_{\mathbb{R}^d\setminus [-\eps,\eps]^d}\left[\int M_{h_{z}}(y+x; \eps)(\mu-\nu)(dy)\right]K(dx)\\
&\geq\int_{\mathbb{R}^d\setminus [-\eps,\eps]^d}\left[\int h_{z}(y+x)\mu(dy)-\int M_{h_{z}}(y+x; \eps)\nu(dy)\right]K(dx)\\
&=\int_{\mathbb{R}^d\setminus [-\eps,\eps]^d}\left[\int h_{z}(y+x)(\mu-\nu)(dy)-\int \left\{M_{h_{z}}(y+x; \eps)-h_{z}(y+x)\right\}\nu(dy)\right]K(dx)\\
&\geq\int_{\mathbb{R}^d\setminus [-\eps,\eps]^d}\left[-\delta-\tau^*(h;\eps)\right]K(dx)
=(1-\alpha)\left[-\delta-\tau^*(h;\eps)\right].
}
Consequently, we obtain
\ba{
\gamma^*(h;\eps)
\geq\int M_{h_{z}}(x; \eps) (\mu-\nu)*K(dx)
\geq(2\alpha-1)\delta-\tau^*(h;\eps)-\alpha\tilde\tau^*(h;\eps)-\alpha\eta.
}
Letting $\eta\downarrow0$, we obtain the desired result. 
If instead of \eqref{eq:smooth-1} we have
\[
\delta=\sup\left\{-\int h_yd(\mu-\nu):y\in\mathbb{R}^d\right\},
\]
then, given any $\eta>0$, we can find $z\in\mathbb{R}^d$ such that
$
-\int h_{z}d(\mu-\nu)\geq\delta-\eta.
$
Now look at $-h_{z}$ (instead of $h_{z}$) and note that $M_{-h_y}(\cdot;\eps)=-m_{h_y}(\cdot;\eps)$ and
\[
\int\{h_y-m_{h_y}(x;\eps)\}\nu(dx)=\int\{M_{-h_y}(x;\eps)-(-h_y)\}\nu(dx)
\]
for every $y\in\mathbb{R}^d$. Proceeding exactly as above, we obtain 
\[
\gamma^*(h;\eps)
\geq-\int m_{h_{z}}(x; \eps) (\mu-\nu)*K(dx)
\geq(2\alpha-1)\delta-\tau^*(h;\eps)-\alpha\tilde\tau^*(h;\eps)-\alpha\eta.
\]
Thus we complete the proof. \qed

\section{Proof of Lemma \ref{gmax-quantile}}\label{sec:gmax-quantile}

Let us prove \eqref{f-lower}. 
We write $F=F_Z$ and $x=F^{-1}(p)$ for short. 
First we consider the case $p\geq1/2$. 
Since the function $\mathbb R^d\ni z\mapsto z^\vee\in\mathbb R$ is convex, $\Phi^{-1}\circ F$ is concave by Corollary A.2.9 in \cite{VaWe23}. 
Hence, for any $y>x$,
\[
\frac{(\Phi^{-1}\circ F)(y)-(\Phi^{-1}\circ F)(x)}{y-x}\leq(\Phi^{-1}\circ F)'(x)=\frac{f(x)}{\phi(\Phi^{-1}(F(x)))}.
\]
Let $y=F^{-1}(p)+c$ with $c$ a positive constant specified later. Then,
\[
\frac{f(F^{-1}(p))}{\phi(\Phi^{-1}(p))}\geq\frac{\Phi^{-1}(F(F^{-1}(p)+c))-\Phi^{-1}(p)}{c}.
\]
Thus, we need to choose $c$ so that $F(F^{-1}(p)+c)$ has an appropriate lower bound. 
Noting $p\geq1/2$, we have
\ba{
1-F(F^{-1}(p)+c)
&=P\bra{Z^\vee>F^{-1}(p)+c}
\leq P\bra{Z^\vee>F^{-1}(1/2)+c}\\
&\leq\frac{\E[(Z^\vee-F^{-1}(1/2))^2]}{c^2}
\leq\frac{2\Var[Z^\vee]}{c^2},
}
where we used the following inequality for the last bound: For any random variable $Y$ and its median $m$, $\E[(Y-m)^2]=\Var[Y]+(\E[Y]-m)^2\leq2\Var[Y]$. 
Thus, letting $c=\sqrt{4\Var[Z^\vee]/(1-p)},$ 
we obtain $F(F^{-1}(p)+c))\geq1-(1-p)/2=(1+p)/2$. Consequently, 
\ben{\label{f-lower-1}
f(F^{-1}(p))\geq\phi(\Phi^{-1}(p))\frac{\Phi^{-1}((1+p)/2)-\Phi^{-1}(p)}{2\sqrt{\Var[Z^\vee]}}\sqrt{1-p}.
}
Also, by the fundamental theorem of calculus,
\ben{\label{f-lower-2}
\Phi^{-1}((1+p)/2)-\Phi^{-1}(p)=\int_p^{(1+p)/2}\frac{1}{\phi(\Phi^{-1}(u))}du
\geq\frac{1-p}{2\phi(\Phi^{-1}(p))},
}
where the last inequality holds because $\phi\circ\Phi^{-1}$ is decreasing on $[1/2,1)$. 
Combining \eqref{f-lower-1} with \eqref{f-lower-2} gives \eqref{f-lower}.

Next consider the case $p<1/2$. 
Since the function $\log\circ \Phi$ is increasing and concave, $\log\circ F=(\log\circ\Phi)\circ(\Phi^{-1}\circ F)$ is concave. Hence $(\log\circ F)'=f/F$ is non-increasing. Thus we obtain
\ba{
\frac{f(F^{-1}(p))}{p}\geq\frac{f(F^{-1}(1/2))}{F(F^{-1}(1/2))}\geq\frac{1}{2^{5/2}\sqrt{\Var[Z^\vee]}},
}
where the last inequality follows from \eqref{f-lower} for $p=1/2$ which was already proved in the above. 

It remains to prove \eqref{gmax-q-2nd} when $\Cov[Z]=\Sigma$. An elementary computation shows
\[
(F_Z^{-1})'(u)=\frac{1}{f_\Sigma(F_Z^{-1}(u))},\qquad
(F_Z^{-1})''(u)=-\frac{f_\Sigma'(F_Z^{-1}(u))}{f_\Sigma(F_Z^{-1}(u))^3}
\]
for all $u\in(0,1)$. Thus, the desired result follows from \eqref{f-lower} and \cref{gmax-2nd-deriv}. 
\qed

\section{Technical tools}

\subsection{Inequalities related to multivariate normal distributions}

Let $Z$ be a centered Gaussian vector in $\mathbb R^d$. 
Set $\ol\sigma:=\max_{1\leq j\leq d}\sqrt{\Var[Z_j]}$ and $\ul\sigma:=\min_{1\leq j\leq d}\sqrt{\Var[Z_j]}$.

\begin{lemma}\label{gmax-mom}
$\E[Z^\vee]\leq\ol\sigma\sqrt{2\log d}$ and $\Var[Z^\vee]\leq\ol\sigma^2$. 
\end{lemma}

\begin{proof}
The first bound follows from \cite[Theorem 2.5]{BLM13}. 
The second one follows from \cite[Theorem 5.8]{BLM13}. 
\end{proof}

\begin{lemma}[Nazarov's inequality]\label{nazarov}
If $\ul\sigma>0$, then for any $x\in\mathbb R^d$ and $\eps>0$,
\[
P\bra{0<\max_{1\leq j\leq d}(Z_j-x_j)\leq \eps}\leq\frac{\eps}{\ul\sigma}(\sqrt{2\log d}+2).
\]
\end{lemma}

\begin{proof}
See \cite[Theorem 1]{CCK17nazarov}. 
\end{proof}

\begin{lemma}\label{gmax-var}
There exists a universal constant $c>0$ such that $\sqrt{\Var[Z^\vee]\log d}\geq c\ul\sigma^2/\ol\sigma$.
\end{lemma}

\begin{proof}
By a straightforward modification of the proof of Theorem 1.8 in \cite{DEZ15}, we can prove
\[
\sqrt{\Var[Z^\vee]}(\ul\sigma+\E[Z^\vee])\geq c'\ul\sigma^2
\]
for some universal constant $c'>0$. In fact, this follows by applying the arguments in the proof of Theorem 1.8 in \cite{DEZ15} to $\mathbf X=Z/\ol\sigma$ with $t=\sqrt{1-\ul\sigma^2/(4m^2)}$ when $m:=\E[Z^\vee]>\ul\sigma/2$ (Lemma 2.2 in \cite{DEZ15} is applied with $\lambda=\ul\sigma^2/(4m\ol\sigma)$). 
Then, since $\E[Z^\vee]\leq\ol{\sigma}\sqrt{2\log d}$ by \cref{gmax-mom}, we obtain the desired result. 
\end{proof}

\begin{lemma}[Anderson--Hall--Titterington's bound]\label{lem:aht}
For any $r\in\mathbb N$, 
\[
\sup_{A\in\mcl R}\norm{\int_A\nabla^r\phi_\Sigma(z)dz}_1\leq C_r\frac{\log^{r/2} d}{\sigma_*^r},
\]
where $C_r>0$ is a constant depending only on $r$. 
\end{lemma}

\begin{proof}
Let $Z\sim N(0,\Sigma)$ and $Z'\sim N(0,\Sigma-\sigma_*^2I_d)$. 
Then, for any $A\in\mcl R$ and $x\in\mathbb R^d$, we have
\ba{
\int_A\phi_\Sigma(x+z)dz
&=\E[1_A(Z-x)]
=\E\sbra{\int_{\mathbb R^d}1_A(\sigma_*z+Z')\phi_d(z+x/\sigma_*)dz}.
}
Differentiating both sides $r$ times with respect to $x$ and then setting $x=0$, we obtain
\[
\int_A\nabla^r\phi_\Sigma(z)dz=\frac{1}{\sigma_*^r}\E\sbra{\int_{\mathbb R^d}1_A(\sigma_*z+Z')\nabla^r\phi_d(z)dz}.
\]
Hence
\ba{
\norm{\int_A\nabla^r\phi_\Sigma(z)dz}_1
\leq\frac{1}{\sigma_*^r}\E\sbra{\norm{\int_{\mathbb R^d}1_A(\sigma_*z+Z')\nabla^r\phi_d(z)dz}_1}
\leq C_r\frac{\log^{r/2}d}{\sigma_*^r},
}
where the last inequality follows by Lemma 2.2 in \cite{FaKo21} because $\{z\in\mathbb R^d:\sigma_*z+Z'\in A\}\in\mcl R$. 
\end{proof}

\begin{lemma}\label{gmax-2nd-deriv}
For any integer $r\geq0$,
\ben{\label{gmax-dens-deriv}
f_\Sigma^{(r)}(t)=\int_{(-\infty,t]^d}\langle\bs1_d^{\otimes {(r+1)}},\nabla^{r+1}\phi_\Sigma(z)\rangle dz\quad\text{for all }t\in\mathbb R.
}
Moreover, there exists a constant $C_r>0$ depending only on $r$ such that $\sup_{t\in\mathbb R}|f_\Sigma^{(r)}(t)|\leq C_r(\sqrt{\log d}/\sigma_*)^{r+1}$.
\end{lemma}

\begin{proof}
Observe that
\[
F_Z(t)=\int_{(-\infty,t]^d}\phi_\Sigma(z)dz=\int_{(-\infty,0]^d}\phi_\Sigma(z+t\bs1_d)dz
\]
for all $t\in\mathbb R$. Differentiating this equation $r+1$ times with respect to $t$ gives
\[
f_\Sigma^{(r)}(t)=\int_{(-\infty,0]^d}\langle\bs1_d^{\otimes(r+1)},\nabla^{r+1}\phi_\Sigma(z+t\bs1_d)\rangle dz=\int_{(-\infty,t]^d}\langle\bs1_d^{\otimes(r+1)},\nabla^{r+1}\phi_\Sigma(z)\rangle dz.
\]
Hence we obtain \eqref{gmax-dens-deriv}. 
The second claim follows from \cref{lem:aht}. 
\end{proof}

\subsection{Inequalities related to sub-Weibull norms}


\begin{lemma}\label{moment-psi}
Let $\xi$ be a random variable. 
Suppose that there is a constant $A>0$ such that $\|\xi\|_p\leq Ap^{1/\alpha}$ for all $p\geq1$. Then $\|\xi\|_{\psi_\alpha}\leq C_\alpha A$. 
\end{lemma}

\begin{proof}
See Lemma A.5 in \cite{Ko23}. 
\end{proof}

\begin{lemma}\label{sum-psi}
For any $\alpha\in(0,1)$, there exists a constant $C_\alpha>0$ depending only on $\alpha$ such that
\[
\norm{\sum_{i=1}^n\xi_i}_{\psi_\alpha}\leq C_\alpha\sum_{i=1}^n\|\xi_i\|_{\psi_\alpha}
\]
for any random variables $\xi_1,\dots,\xi_n$. 
\end{lemma}

\begin{proof}
This follows from Lemma C.2 in \cite{ChKa19} and the triangle inequality for the Orlicz norm associated with a convex function.  
\end{proof}

\begin{lemma}\label{prod-psi}
Let $\xi_1,\xi_2$ be two random variables such that $\|\xi_1\|_{\psi_{\alpha_1}}+\|\xi_2\|_{\psi_{\alpha_2}}<\infty$ for some $\alpha_1,\alpha_2>0$. Then we have 
$
\|\xi_1\xi_2\|_{\psi_\alpha}\leq\|\xi_1\|_{\psi_{\alpha_1}}\|\xi_2\|_{\psi_{\alpha_2}},
$ 
where $\alpha>0$ is defined by the equation $1/\alpha=1/\alpha_1+1/\alpha_2$. 
\end{lemma}

\begin{proof}
See \cite[Proposition D.2]{KuCh22}.
\end{proof}

\begin{lemma}\label{lem:weibull}
Let $\xi_1,\dots,\xi_n$ be independent random variables such that $\max_{1\leq i\leq n}\|\xi_i\|_{\psi_\alpha}\leq K$ for some $K>0$ and $\alpha\in(0,1]$. 
Then, there is a constant $C_\alpha>0$ depending only on $\alpha$ such that for any $p\geq1$, 
\[
\left\|\sum_{i=1}^n(\xi_i-\E[\xi_i])\right\|_p\leq C_\alpha K\left(\sqrt{pn}+p^{1/\alpha}\right).
\]
\end{lemma}

\begin{proof}
See Lemma 2.1 in \cite{FaKo23}. 
\end{proof}

\begin{lemma}\label{max-weibull}
Let $Y_1,\dots,Y_n$ be independent random vectors in $\mathbb R^d$. Suppose that there exist constants $K>0$ and $\alpha\in(0,1]$ such that $\max_{1\leq i\leq n}\max_{1\leq j\leq d}\|Y_{ij}\|_{\psi_\alpha}\leq K$. Then, there exists a constant $C_\alpha>0$ depending only on $\alpha$ such that
\ben{\label{eq:max-weibull-1}
\E\norm{\sum_{i=1}^n(Y_{i}-\E[Y_{i}])}_\infty^r\leq C_\alpha^r K^r\bra{\sqrt{nr\log d}+(r\log d)^{1/\alpha}}^r
}
for any $r\geq1$ and
\ben{\label{eq:max-weibull-2}
P\bra{\norm{\sum_{i=1}^n(Y_{i}-\E[Y_{i}])}_\infty>C_\alpha K\bra{\sqrt{an\log(dn)}+a^{1/\alpha}\log^{1/\alpha}(dn)}}\leq\frac{1}{n^a}
}
for any $a\geq1$. 
\end{lemma}

\begin{proof}
By \cref{lem:weibull}, we have for any $p\geq2$
\ben{\label{weibull-est}
\max_{1\leq j\leq d}\norm{\sum_{i=1}^n(Y_{ij}-\E[Y_{ij}])}_p\leq C'_\alpha K\bra{\sqrt{pn}+p^{1/\alpha}},
}
where $C_\alpha'>0$ depends only on $\alpha$. 
Therefore, with $p=r\log d$, we have
\ba{
\E\norm{\sum_{i=1}^n(Y_{i}-\E[Y_{i}])}_\infty^r
&\leq\bra{\E\norm{\sum_{i=1}^n(Y_{i}-\E[Y_{i}])}_\infty^p}^{r/p}
\leq d^{r/p}\max_{1\leq j\leq d}\norm{\sum_{i=1}^n(Y_{ij}-\E[Y_{ij}])}_p^r\\
&\leq d^{r/p}(C'_\alpha)^r K^r\bra{\sqrt{nr\log d}+(r\log d)^{1/\alpha}}^r.
}
Since $d^{r/p}=e^{r\log d/p}=e\leq e^r$, we obtain \eqref{eq:max-weibull-1} with $C_\alpha=eC_\alpha'$. 
Also, by the union bound, Markov's inequality and \eqref{weibull-est}, we have for any $t>0$ and $p\geq2$
\ba{
P\bra{\norm{\sum_{i=1}^n(Y_{i}-\E[Y_{i}])}_\infty>t}
&\leq d\bra{t^{-1}C'_\alpha K\bra{\sqrt{pn}+p^{1/\alpha}}}^p.
}
Applying this estimate with $p=a\log(dn)$ and $t=eC'_\alpha K\bra{\sqrt{pn}+p^{1/\alpha}}$, we obtain \eqref{eq:max-weibull-2}
\end{proof}

\begin{lemma}\label{lem:tensor}
Let $r\in\mathbb N$. If \eqref{ass:psi1} is satisfied, there exists a constant $C_{r}>0$ depending only on $r$ such that
\[
P\bra{\abs{\frac{1}{n}\sum_{i=1}^n(\langle X_i^{\otimes r},V\rangle-\E[\langle X_i^{\otimes r},V\rangle])}>a^rC_{r}\|V\|_1b^r\sqrt{\frac{\log n}{n}}}
\leq \frac{1}{n^a}
\]
for any $a\geq1$ and $V\in(\mathbb R^d)^{\otimes r}$. 
\end{lemma}

\begin{proof}
By Lemmas \ref{sum-psi} and \ref{prod-psi}, for every $i=1,\dots,n$,
\ba{
\norm{\langle X_i^{\otimes r},V\rangle}_{\psi_{1/r}}
&\leq C_r\sum_{j_1,\dots,j_r=1}^d\norm{X_{ij_1}\cdots X_{ij_r}}_{\psi_{1/r}}|V_{j_1,\dots,j_r}|
\leq 
C_r\|V\|_1\max_j\norm{X_{ij}}_{\psi_1}^r.
}
Therefore, by \cref{lem:weibull}, there exists a constant $C_r'>0$ depending only on $r$ such that
\ba{
\norm{\frac{1}{n}\sum_{i=1}^n(\langle X_i^{\otimes r},V\rangle-\E[\langle X_i^{\otimes r},V\rangle])}_p
\leq C_{r}'\|V\|_1b^r\bra{\sqrt{\frac{p}{n}}+\frac{p^r}{n}}
}
for any $p\geq1$. Hence, with $p=a\log n$, we have by Markov's inequality
\ba{
P\bra{\abs{\frac{1}{n}\sum_{i=1}^n(\langle X_i^{\otimes r},V\rangle-\E[\langle X_i^{\otimes r},V\rangle])}>eC'_{r}\|V\|_1b^r\bra{\sqrt{\frac{p}{n}}+\frac{p^r}{n}}}
\leq e^{-p}=\frac{1}{n^a}.
}
Noting $n^{-1}\log^rn=\sqrt{n^{-1}\log n}\cdot n^{-1/2}\log^{r-1/2}n\leq C_{r}\sqrt{n^{-1}\log n}$, we complete the proof. 
\end{proof}

\begin{lemma}\label{hj-bound}
If \eqref{ass:psi1} is satisfied, there exists a universal constant $C>0$ such that for any $a\geq1$,
\ben{\label{eq:hj-bound-1}
P\bra{\max_{1\leq j,k\leq d}\frac{1}{n}\sum_{i=1}^n|X_{ij}||X_{ik}|1_{\{|X_{ij}|\vee|X_{ik}|>2b\log n\}}>Cb^2\bra{\sqrt{\frac{a\log(dn)}{n}}+\frac{a^2\log^2(dn)}{n}}}\leq\frac{1}{n^a}
}
and
\ben{\label{eq:hj-bound-2}
P\bra{\max_{1\leq j,k\leq d}\frac{1}{n}\sum_{i=1}^nX_{ij}^2X_{ik}^21_{\{|X_{ij}|\vee|X_{ik}|\leq 2b\log n\}}>Cb^4\bra{1+a\frac{\log(dn)\log^4n}{n}}}\leq\frac{1}{n^a}.
}
\end{lemma}

\begin{proof}
Let us prove \eqref{eq:hj-bound-1}. 
Write $Y_{i,jk}:=|X_{ij}||X_{ik}|1_{\{|X_{ij}|\vee|X_{ik}|>2b\log n\}}$ for short. Since $\|Y_{i,jk}\|_{\psi_{1/2}}\leq b^2$, we have by \cref{max-weibull}
\[
\max_{1\leq j,k\leq d}\frac{1}{n}\abs{\sum_{i=1}^n(Y_{i,jk}-\E[Y_{i,jk}])}\lesssim b^2\bra{\sqrt{\frac{a\log(dn)}{n}}+\frac{a^2\log^2(dn)}{n}}
\]
with probability at least $1-1/n^a$. 
Moreover, observe that
\ba{
|\E[Y_{i,jk}]|
&\leq\max_{1\leq j,k\leq d}\sqrt{\E[X_{ij}^2X_{ik}^2]P(|X_{ij}|>2b\log n)}\lesssim \frac{b^2}{n}.
}
Hence we conclude
\[
\max_{1\leq j,k\leq d}\frac{1}{n}\abs{\sum_{i=1}^nY_{i,jk}}\lesssim b^2\bra{\sqrt{\frac{a\log(dn)}{n}}+\frac{a\log^2(dn)}{n}}
\]
with probability at least $1-1/n^a$. This proves \eqref{eq:hj-bound-1}. 

Next we prove \eqref{eq:hj-bound-2}. 
Set
\[
\zeta:=\max_{1\leq j,k\leq d}\frac{1}{n}\sum_{i=1}^nX_{ij}^2X_{ik}^21_{\{|X_{ij}|\vee|X_{ik}|\leq 2b\log n\}}.
\]
By Lemma E.5 in \cite{CCK17} with $\eta=3$ and $B=16b^4\log^4n/n$, we have for any $t>0$
\ba{
P\bra{\zeta\geq4\E[\zeta]+16\frac{b^4\log^4n}{n}t}\leq e^{-t}.
}
Further, 
$
\E[\zeta]
\lesssim b^4+\frac{b^4\log^4n}{n}\log d
$
by Lemma 9 in \cite{CCK15}. 
Hence there exists a universal constant $C_2>0$ such that
\ba{
P\bra{\zeta\geq C_2b^4\bra{1+\frac{\log^4n}{n}(\log d+t)}}\leq e^{-t}
}
for any $t>0$. Applying this with $t=a\log n$ gives \eqref{eq:hj-bound-2}. 
\end{proof}

\paragraph{Acknowledgments}

The author thanks Xiao Fang and Ryo Imai for valuable discussions about the subject of this paper. 
Thanks are also due to Kouki Akatsuka for his helpful comments about several imperfections contained in the original version of the paper. 
This work was partly supported by JST CREST Grant Number JPMJCR2115 and JSPS KAKENHI Grant Numbers JP22H00834, JP22H00889, JP22H01139, JP24K14848.

{
\renewcommand*{\baselinestretch}{1}\selectfont
\addcontentsline{toc}{section}{References}

}

\end{document}